\documentclass[10pt,a4paper,makeidx]{amsart}
\makeatletter
\RequirePackage{ifthen}
\provideboolean{omarfont}
\setboolean{omarfont}{false}
\provideboolean{usesvjour}
\setboolean{usesvjour}{false}%
\provideboolean{nohyperref}
\setboolean{nohyperref}{false}%
\provideboolean{usecleveref}
\setboolean{usecleveref}{false}
\provideboolean{landscape}
\setboolean{landscape}{false}%
\makeatletter
      \RequirePackage{ifthen}
      \provideboolean{useutopia}
      \setboolean{useutopia}{false}%
      \provideboolean{usesvjour}
      \setboolean{usesvjour}{false}%
      \provideboolean{usinghyperref}
      \setboolean{usinghyperref}{false}%
\makeatletter
        \RequirePackage{xparse}
        \RequirePackage{ifthen}
        \RequirePackage{iftex}
        \provideboolean{usemathrsfs}%
        \provideboolean{useutopia}
        \setboolean{useutopia}{false}%
        \ifXeTeX%
          \RequirePackage{xltxtra}%
          \ifthenelse{1=0}{
            \RequirePackage{polyglossia}
            \setdefaultlanguage[variant=british]{english}
            \setotherlanguages{french,vietnamese,russian}%
          }{
            \RequirePackage[american,main=british,french,vietnamese]{babel}
          }
          \DeclareSymbolFont{usualmathcal}{OMS}{cmsy}{m}{n}
          \DeclareSymbolFontAlphabet{\mathcalbf}{usualmathcal}
          \ifthenelse{\boolean{useutopia}}{
            \RequirePackage[utopia,euro]{mathdesign}
            \RequirePackage[OMLmathbf,OMLmathsfit]{isomath}
          }{
            \RequirePackage{amssymb}    %
            \RequirePackage{bbold}    %
          }
          \RequirePackage{mathrsfs} %
          \setboolean{usemathrsfs}{true}%
        \else{%
          \RequirePackage[utf8]{inputenc}%
          \ifthenelse{\boolean{useutopia}}{%
            \RequirePackage[utopia,euro]{mathdesign}%
            \RequirePackage[OMLmathrm,OMLmathbf,OMLmathsfit]{isomath}
            \RequirePackage{bbold}    %
            \DeclareSymbolFont{usualmathcal}{OMS}{cmsy}{m}{n}
            \DeclareSymbolFontAlphabet{\mathcalbf}{usualmathcal}
            \RequirePackage{stmaryrd} %
            \providecommand{\diracdelta}[1][]{\ensuremath{\deltaup_{#1}}}
            
            \providecommand{\lap}{\ensuremath{\Deltaup}}
            \providecommand{\pic}{\ensuremath{\mathrm\pi}}
            \providecommand{\measure}[1]{\ensuremath{\mathcalbf{\uppercase{#1}}}}
          }{%
            \RequirePackage{amssymb}    %
            \RequirePackage{mathrsfs} %
            \RequirePackage{bbold}    %
            \RequirePackage{stmaryrd} %
            \setboolean{usemathrsfs}{true}%
            \providecommand{\mathcalbf}{\mathcal}
            \newcommand{\mathsfit}{\mathsf}

          }%
          \RequirePackage[american,main=british,vietnamese]{babel}
        \fi%
        \RequirePackage{xstring}%
        \RequirePackage{chngcntr}%
        \RequirePackage{mathtools}%
        \RequirePackage{nicefrac}
        \RequirePackage{stmaryrd} %
        \RequirePackage{amsthm}%
        \RequirePackage[shortlabels]{enumitem}%
        \RequirePackage{xspace}%
        \RequirePackage{verbatim}%
        \RequirePackage{fvextra}
        \RequirePackage{fancyvrb}
        \RequirePackage{listings}%
        \RequirePackage{xcolor}
        \RequirePackage{epsfig}
        \RequirePackage{graphicx}
        \RequirePackage[space]{grffile}%
        \RequirePackage{booktabs}%
        \RequirePackage{tikz}%
        \usetikzlibrary{calc}%
        \usetikzlibrary{fadings}
        \provideboolean{usebeamer}%
        \provideboolean{isthesis}%
        \provideboolean{isamsltex}%
        \provideboolean{issiamltex}%
      \def\olprovideenvironment{\@star@or@long\provide@environment}
      \def\provide@environment#1{%
              \@ifundefined{#1}%
                      {\def\reserved@a{\newenvironment{#1}}}%
                      {\def\reserved@a{\renewenvironment{dummy@environ}}}%
              \reserved@a
      }
      \def\dummy@environ{}
      \colorlet{a}{magenta}
      \colorlet{b}{green!75!blue}
      \colorlet{c}{yellow!87.5!red}
      \colorlet{d}{cyan}
      \colorlet{e}{red}
      \colorlet{f}{blue}
      \colorlet{g}{white}
      \colorlet{i}{black}
      \colorlet{h}{i!50!g}
      \colorlet{j}{a!75!g}

      \ifthenelse{\boolean{usinghyperref}}{%
        \providecommand{\linkedurl}[1]{\url{1}}%

      }{%
        \providecommand{\linkedurl}[1]{\texttt{#1}}%
      }

      \providecommand{\Ignore}[1]{}
      \providecommand{\ignore}[1]{}
      \providecommand{\freeze}[1]{}%
      
      \providecommand{\crossout}[1]{{\textcolor{i!20}{#1}}}
      \providecommand{\highlightcolor}{a}
      \providecommand{\highlight}[1]{{\color{\highlightcolor}#1}}

      \providecommand{\memo}[1]{%
        \ensuremath{%
          \framebox{\tiny\textbf{\kern-2pt\textsf{#1}}\kern-2pt}%
        }%
        \xspace}

      \RequirePackage{alphalph}
      \provideboolean{shownotes}
      \setboolean{shownotes}{true}%
      \newcounter{margnote}[page]
      \providecommand{\mgcolor}{a}
      \providecommand{\mgcolorset}[1]{\renewcommand{\mgcolor}{\alphalph{#1}}}
      \providecommand{\mgcolorsetbycounter}[1]{%
        \ifthenelse{\value{#1}<11}{%
          \ifthenelse{\value{#1}=0}{
            \renewcommand{\mgcolor}{j}
          }{
            \renewcommand{\mgcolor}{\alph{#1}}%
          }
        }{%
          \renewcommand{\mgcolor}{a}%
        }%
      }
      \providecommand{\mgcolormake}{\mgcolorsetbycounter{margnote}}
      \providecommand{\mgcolorstepby}[1]{
        \setcounter{tmpcounter}{\value{margnote}}%
        \addtocounter{tmpcounter}{#1}%
        \mgcolorsetbycounter{tmpcounter}%
      }%
      \providecommand{\margnotecolor}{%
        \ifthenelse{\value{margnote}=0}{%
          \mgcolorset{10}
        }{%
          \ifthenelse{\value{margnote}<7}{%
            \mgcolormake%
          }{%
            \ifthenelse{\value{margnote}=7}{\mgcolorset{10}}{%
              \ifthenelse{\value{margnote}<11}{%
                \mgcolormake%
              }{%
                \ifthenelse{\value{margnote}<17}{%
                  \mgcolorstepby{-10}%
                }{%
                  \mgcolorset{10}%
                }%
              }%
            }%
          }%
        }%
      }%
      \providecommand{\margnotemark}{{\colorbox{\mgcolor}{\tiny\color{g}\upshape\texttt{\arabic{page}.\arabic{margnote}}}}\,}
      \providecommand{\margnote}[2][]{%
        \ifthenelse{%
          \boolean{shownotes}%
        }{%
          \stepcounter{margnote}%
          \margnotecolor%
          \margnotemark %
          \marginpar{%
            \color{\mgcolor}%
            \texttt{\bfseries{%
              \begin{minipage}{2cm}%
                \raggedright\tiny%
                \margnotemark%
                #2%
                \\
                {\ifx|#1|{}\else{ - #1}\fi}%
              \end{minipage}%
              }%
            }%
          }%
        }{%
        }%
      }%
      \providecommand{\mathnote}[2][]{%
        \ifthenelse{%
          \boolean{shownotes}%
        }{%
          \stepcounter{margnote}%
          \margnotecolor%
          \text{%
            \colorbox{g}{%
              \color{\mgcolor}%
              \texttt{\bfseries{%
                \tiny%
                    \margnotemark\,%
                    #2%
                    \ifx|#1|{}\else{ - #1}\fi%
                }%
              }%
            }%
          }%
        }{%
        }%
      }%
      \providecommand{\textnote}[2][]{%
        \ifthenelse{%
          \boolean{shownotes}%
        }{%
          \stepcounter{margnote}%
          \margnotecolor%
          \ \\
          \text{%
              \begin{minipage}{\textwidth}
              \color{\mgcolor}%
              \texttt{%
                \margnotemark%
                #2%
                \ifx|#1|{}\else{ - #1}\fi%
              }%
              \end{minipage}
          }%
        }{}%
      }%

      \providecommand{\Todo}[2][To do:]{
        \ifthenelse{\boolean{shownotes}}{
          \begin{tikzpicture}
           \node[fill=a!17]{
             \begin{minipage}{\textwidth}
               \tiny
               \texttt{#1}
               \texttt{\bfseries{#2}}
             \end{minipage}
           };
          \end{tikzpicture}
        }{}}
      \provideboolean{showrevisions}
      \setboolean{showrevisions}{true}
      \provideboolean{emphrevisions}
      \setboolean{emphrevisions}{false}
      \newcommand{\revisionsheader}{\ \clearpage\Warning{the following part is under development/revision}}
      \newcommand{\revisionsfooter}{\ \newline\Warning{end of part under development/revision}\clearpage}

      \providecommand{\HighlightBox}[2][a!6.25]{
        \begin{center}
          \begin{tikzpicture}
            \node[fill=#1]{
              \begin{minipage}{\textwidth}
                #2
              \end{minipage}
            };
          \end{tikzpicture}
        \end{center}
      }
      \providecommand{\Warning}[1]{    
        \HighlightBox[b!25]{%
          \texttt{\bfseries{\small Warning: #1}}
        }
      }

      \provideboolean{showcomments}
      \providecommand{\margincomment}[1]{
      \ifthenelse{\boolean{showcomments}}{\marginpar{\tiny #1}}{}
      }
      \provideboolean{showchanges}
      \setboolean{showchanges}{false}
      \providecommand{\changes}[2][]{%
        \ifthenelse{\boolean{showchanges}}{{%
            \ifx|#1|{}\else\margnote{#1}\fi%
            \highlight{#2}%
        }}{%
          #2}}
      \providecommand{\mathchanges}[2][]{%
        \ifthenelse{\boolean{showchanges}}{{\ifx|#1|{}\else\mathnote{#1}\fi\highlight{#2}}}{#2}}

      \providecommand{\changefromto}[3][replace with]{%
        \ifthenelse{\boolean{showchanges}}{%
          {\crossout{#2}\margnote{#1}}{\highlight{#3}}}{%
          #3\xspace}%
      }
      \providecommand{\ChangePar}[3][]{%
        \ifthenelse{\boolean{showchanges}}{
          {\par\textcolor{i!20}{#2}\ifx|#1|\else\margnote{#1}\fi}{\par\textcolor{a}{#3}}
        }{%
          \par #3%
        }%
      }
      \providecommand{\InsertPar}[1]{
        \ifthenelse{\boolean{showchanges}}
        {{\par$\mapsto$ \textcolor{blue}{#1}}}
        {\par #1}
      }
      
      \providecommand{\mathchangefromto}[3][]{\crossout{#2}\ifx|#1|\else\mathnote{#1}\fi\highlight{#3}}

      \let\trueMakeUppercase\MakeUppercase
      \newcommand{\UCmath}[1]{%
        \begingroup
        \ucmathlist\trueMakeUppercase{#1}%
        \endgroup
      }
      \ifthenelse{\(\boolean{useutopia}\)\OR\(\boolean{usesvjour}\)}{
        \newcommand{\ucmathlist}{%
          \def\alpha{A}%
          \def\beta{B}%
          \let\gamma\Gamma
          \let\delta\Delta
          \def\epsilon{E}%
          \def\varepsilon{E}%
          \def\zeta{Z}%
          \def\eta{H}%
          \let\theta\Theta
          \let\vartheta\Theta
          \def\iota{I}%
          \def\kappa{K}%
          \let\lambda\Lambda
          \def\mu{M}%
          \def\nu{N}%
          \let\xi\Xi
          \def\omicron{O}
          \let\pi\Pi
          \let\varpi\Pi
          \def\rho{P}%
          \def\varrho{P}%
          \let\sigma\Sigma
          \def\varsigma{C}
          \def\tau{T}%
          \let\upsilon\Upsilon
          \let\phi\Phi
          \let\varphi\Phi
          \def\chi{X}%
          \let\psi\Psi
          \let\omega\Omega
      }}{
        \newcommand{\ucmathlist}{
          \def\alpha{\mathrm{A}}%
          \def\beta{\mathrm{B}}%
          \let\gamma\Gamma
          \let\delta\Delta
          \def\epsilon{\mathrm{E}}%
          \def\varepsilon{\mathrm{E}}%
          \def\zeta{\mathrm{Z}}%
          \def\eta{\mathrm{H}}%
          \let\theta\Theta
          \let\vartheta\Theta
          \def\iota{\mathrm{I}}%
          \def\kappa{\mathrm{K}}%
          \let\lambda\Lambda
          \def\mu{\mathrm{M}}%
          \def\nu{\mathrm{N}}%
          \let\xi\Xi
          \let\pi\Pi
          \let\varpi\Pi
          \def\rho{\mathrm{P}}%
          \def\varrho{\mathrm{P}}%
          \let\sigma\Sigma
          \def\tau{\mathrm{T}}%
          \let\upsilon\Upsilon
          \let\phi\Phi
          \let\varphi\Phi
          \def\chi{\mathrm{X}}%
          \let\psi\Psi
          \let\omega\Omega
        }
      }
      \providecommand{\mathscript}
      	       {\mathscr}

      \providecommand{\cD}{\ensuremath{\mathscript D}\xspace}

      \providecommand{\hA}{\ensuremath{\mathcalbf A}\xspace}

      \providecommand{\bbbold}{\mathbb}

      \providecommand{\rN}{\ensuremath{\bbbold N}\xspace}
      
      \providecommand{\rP}{\ensuremath{\bbbold P}\xspace}
      
      \providecommand{\rR}{\ensuremath{\bbbold R}\xspace}
      
      \providecommand{\rT}{\ensuremath{\bbbold T}\xspace}

      \providecommand{\rZ}{\ensuremath{\bbbold Z}\xspace}

      \providecommand{\Ae}[1][]{\ensuremath{\ifx|#1|{\ }\else{\:#1\text{-}}\fi\text{almost everywhere }}\xspace}
      \providecommand{\Aa}[1][]{\ensuremath{\text{ for }\ifx|#1|{}\else{\:#1\text{-}}\fi\text{almost all }}}
      \providecommand{\as}[1][]{\ensuremath{\ifx|#1|{\ }\else{#1\text{-}}\fi\text{almost surely}}\xspace}
      \providecommand{\aposteriori}{aposteriori\xspace}

      \providecommand{\apriori}{{apriori}\xspace}

       \providecommand{\naturals}{\ensuremath{\rN}}
       
       \providecommand{\NO}[1][]{\ensuremath{\naturals_0\ifx|#1|{}\else^{#1}\fi}}
       
       \providecommand{\integers}{\rZ}

       \providecommand{\reals}{\rR}

       \providecommand{\R}[1]{\reals^{#1}}

       \providecommand{\fieldmats}[3][F]{\csname#1\endcsname{#2\times#3}}
       
       \providecommand{\fieldtens}[3][F]{\csname#1\endcsname{{#2}_1\times\dotsb\times{#2}_{#3}}}
       \providecommand{\realmats}[2]{\fieldmats[R]{#1}{#2}}
       \providecommand{\realsqmats}[1]{\realmats{#1}{#1}}

       \providecommand{\RO}[1][]{{\reals_{0+}\ifx|#1|{}\else^{#1}\fi}}
       \providecommand{\RP}[1][]{{\reals_+\ifx|#1|\else^{#1}\fi}}

       \providecommand{\ring}[1][A]{\csname r#1\endcsname}
       \providecommand{\field}[1][K]{\csname r#1\endcsname}

       \providecommand{\torus}[1]{\rT\ifthenelse{\equal{#1}1}{}{^#1}}
      
       \providecommand{\one}{\ensuremath{\bbbold 1}}
       \providecommand{\zerofun}{\ensuremath{\bbbold 0}}
       \providecommand{\ones}[1][]{\one\ifx|#1|\else_{#1}\fi}
       \providecommand{\zeros}[1][]{\zerofun\ifx|#1|\else_{#1}\fi}

       \providecommand{\diracdelta}[1][]{\ensuremath{{\mathrm{\delta}}\ifx|#1|{}\else_{#1}\fi}}

       \providecommand{\pic}{\ensuremath{\mathrm\pi}}%
       \providecommand{\pifracl}[2][]{\fracl{\ifx|#1|\else#1\fi\pic}{#2}}
       \providecommand{\pifrac}[2][]{\frac{\ifx|#1|\else#1\fi\pic}{#2}}

       \providecommand{\take}{\smallsetminus}
       \providecommand{\takesetof}[1]{\take\setof{#1}}
       \providecommand{\takeset}\takesetof
       \providecommand{\takeel}\oldneg%

       \providecommand{\inner}{\cdot}
       \providecommand{\vecprod}{\times}
       \providecommand{\outerp}{\wedge}

       \providecommand{\frobinner}{\!:\!}

       \providecommand{\W}{\ensuremath{\varOmega}\xspace}

       \providecommand{\w}{\ensuremath{\omega}\xspace}

       \providecommand{\qp}[2][]{\ensuremath{\ifx|#1|\left(\else\csname#1\endcsname(\fi{#2}\ifx|#1|\right)\else\csname#1\endcsname)\fi}}

       \providecommand{\qpreg}[1]{\ensuremath{(#1)}}
       \providecommand{\qpbig}[1]{\qp[big]{#1}}%
       \providecommand{\qpBig}[1]{\ensuremath{\Big(#1\Big)}}
       \providecommand{\qpbigg}[1]{\ensuremath{\bigg(\!#1\!\bigg)}}
       \providecommand{\qpBigg}[1]{\ensuremath{\Bigg(\!#1\!\Bigg)}}
       \providecommand{\qb}[2][]{\ifx|#1|\left[\else\csname#1\endcsname[\fi{#2}\ifx|#1|\right]\else\csname#1\endcsname]\fi}
       \providecommand{\qc}[2][]{\ensuremath{\ifx|#1|\left\{\else\csname#1\endcsname\{\fi{#2}\ifx|#1|\right\}\else\csname#1\endcsname\}\fi}}

       \providecommand{\qa}[2][]{\ifx|#1|\left\langle\else\csname#1\endcsname\langle\fi{#2}\ifx|#1|\right\rangle\else\csname#1\endcsname\langle\fi}%
       \providecommand{\qareg}[1]{\ensuremath{\langle#1\rangle}}
       \providecommand{\qabig}[1]{\ensuremath{\big\langle#1\big\rangle}}
       \providecommand{\qaBig}[1]{\ensuremath{\Big\langle#1\Big\rangle}}
       \providecommand{\qabigg}[1]{\ensuremath{\bigg\langle#1\bigg\rangle}}
       \providecommand{\qaBigg}[1]{\ensuremath{\Bigg\langle#1\Bigg\rangle}}
       \providecommand{\opinter}[2]{\ensuremath{\left(#1,#2\right)}\xspace}

       \providecommand{\opclinter}[2]{\ensuremath{\left(#1,#2\right]}\xspace}
       \providecommand{\clopinter}[2]{\ensuremath{\left[#1,#2\right)}\xspace}   
       \providecommand{\opintertopinfty}[1]{\opinter{#1}\infty}
       \providecommand\optyinter\opintertopinfty
       \providecommand{\opinterbotinfty}[1]{\opinter{-\infty}{#1}}
       \providecommand\tyopinter\opinterbotinfty
       \providecommand{\clintertopinfty}[1]{\clopinter{#1}\infty}
       \providecommand{\cltyinter}\clintertopinfty
       \providecommand{\clinterbotinfty}[1]{\opclinter{-\infty}{#1}}
       \providecommand{\tyclinter}\clinterbotinfty
       \providecommand{\expp}[1]{\ensuremath{\e^{#1}}}
       
       \providecommand{\compowqp}[2]{\ensuremath{\qp{\!#2\!\!}^{\kern -.4em #1}\!}}
       
       \providecommand{\powqpreg}[2]{\ensuremath{%
           \qpreg{#2}^{\kern 0em\lower .1ex\hbox{\scriptsize $#1$}}\kern-.3em}}
       \providecommand{\powqpbig}[2]{\ensuremath{%
           \qpbig{#2}^{\kern -.2em\lower .3ex\hbox{\scriptsize $#1$}}\kern-.3em}}
       \providecommand{\powqpBig}[2]{\ensuremath{%
           \qpBig{#2}^{\kern -.2em\lower .3ex\hbox{\scriptsize $#1$}}\kern-.3em}}
       \providecommand{\powqpbigg}[2]{\ensuremath{%
           \qpbigg{#2}^{\kern -.2em\lower .3ex\hbox{\scriptsize $#1$}}\kern-.3em}}
       \providecommand{\powqpBigg}[2]{\ensuremath{%
           \qpBigg{#2}^{\kern -.2em\lower .3ex\hbox{\scriptsize $#1$}}}}
       \providecommand{\powp}[3][]{{#3}\ifx|#1|^{#2}\else{#1}^{#2}\fi}%
       \providecommand{\pow}[2][]{\ifx|#1|\operatorname{pow}^{#2}\else\powp{#2}{#1}\fi}%
       \providecommand{\ppow}[3][]{\powp[#1]{#3}{#2}}

       \providecommand{\norm}[2][]{\ifx|#1|\left|\else\csname#1\endcsname|\fi#2\ifx|#1|\right|\else\csname#1\endcsname|\fi}
       \providecommand{\normon}[3][]{\norm[#1]{#2}_{#3}}

       \providecommand{\abs}[2][]{\ensuremath{\ifx|#1|{\left|#2\right|}\else{\csname#1\endcsname|{#2}\csname#1\endcsname|}\fi}}

       \providecommand{\Norm}[2][]{\ifx|#1|\left\|\else\csname#1\endcsname\|\fi{#2}\ifx|#1|\right\|\else\csname#1\endcsname\|\fi}

       \providecommand{\Normon}[3][]{\Norm[#1]{#2}_{#3}}

       \providecommand{\normonsob}[5][]{\normon[#1]{#2}{\sob{#3}{#4}\if|#5|{}\else(#5)\fi}}
       \providecommand{\Normonsob}[5][]{\Normon[#1]{#2}{\sob{#3}{#4}\if|#5|{}\else(#5)\fi}}
       \providecommand{\normonsobh}[4][]{\normon[#1]{#2}{\sobh{#3}\if|#4|{}\else(#4)\fi}}
       \providecommand{\normonsobhz}[4][]{\normon[#1]{#2}{\sobhz{#3}\if|#4|{}\else(#4)\fi}}
       \providecommand{\Normonsobh}[4][]{\Normon[#1]{#2}{\sobh{#3}\if|#4|{}\else(#4)\fi}}
       \providecommand{\Normonsobhz}[4][]{\Normon[#1]{#2}{\sobhz{#3}\if|#4|{}\else(#4)\fi}}
       \providecommand{\Normonleb}[4][]{\Normon[#1]{#2}{\leb{#3}\if|#4|\else(#4)\fi}}

       \providecommand{\ltwop}[3][]{\ensuremath{\qa{#2,#3}\ifx|#1|\else_{#1}\fi}}
       \providecommand{\ltwopreg}[2]{\ensuremath{\qareg{#1,#2}\ifx|#1|\else_{#1}\fi}}
       \providecommand{\ltwopbig}[2]{\ensuremath{\qabig{#1,#2}\ifx|#1|\else_{#1}\fi}}
       \providecommand{\ltwopBig}[2]{\ensuremath{\qaBig{#1,#2}\ifx|#1|\else_{#1}\fi}}
       \providecommand{\ltwopbigg}[2]{\ensuremath{\qabigg{#1,#2}\ifx|#1|\else_{#1}\fi}}
       \providecommand{\ltwopBigg}[2]{\ensuremath{\qaBigg{#1,#2}\ifx|#1|\else_{#1}\fi}}

       \providecommand{\average}[2][]{{\qa{#2}\ifx|#1|\else_{#1}\fi}}

       \providecommand{\ensemble}[2]{\ensuremath{\left\{ #1:\;#2 \right\}}}
       \providecommand{\setofsuch}{\ensemble}%

       \providecommand{\ceil}[1]{\ensuremath{\left\lceil{#1}\right\rceil}}

       \providecommand{\setof}[1]{{\qc{#1}}}
       \providecommand{\pair}[2]{\qp{#1,#2}}
       
       \providecommand{\triple}[3]{\qp{#1,#2,#3}}

       \providecommand{\conditionalto}[1]{{\left|{#1}\right.}}

      \providecommand{\measure}[1]{\ensuremath{\mathcalbf{\MakeUppercase{#1}}}}
      
      \providecommand{\probmeasure}[2][]{{\measure{#2}}\ifx|#1|\else_{#1}\fi}
      \providecommand{\Prob}{}
      \renewcommand{\Prob}[1][]{\probmeasure[{#1}]{p}}

      \providecommand{\randvars}[1][\Prob]{\operatorname{RV}\ifx|#1|{}\else{(#1)}\fi}
      \providecommand{\discrandvars}[1][\Prob]{\operatorname{DRV}\ifx|#1|{}\else{({#1)}\fi}} 
      \providecommand{\contrandvars}[1][\Prob]{\ensuremath{\operatorname{CDRV}\ifx|#1|{}\else(#1)\fi}} 
       \def\env@matrix{\hskip -\arraycolsep
        \let\@ifnextchar\new@ifnextchar
        \array{*\c@MaxMatrixCols c}}
       \renewcommand*\env@matrix[1][c]{\hskip -\arraycolsep
         \let\@ifnextchar\new@ifnextchar
         \array{*\c@MaxMatrixCols #1}}
       \providecommand{\irow}[2]{#1_{#2}}%
       \providecommand{\icol}[2]{#1^{#2}}%

       \providecommand{\ijrowcol}[3]{\icol{\irow{#1}{#2}}{#3}}
       \providecommand{\entry}[1]{\qb{#1}}
       \providecommand{\vecentry}[2]{\irow{#1}{#2}}
       
       \providecommand{\colvecentry}\vecentry
       \providecommand{\covecentry}[2]{\icol{#1}{#2}}
       \providecommand\rowvecentry\covecentry

       \providecommand{\rowof}[1]{\qb{#1}}

       \providecommand{\getentryi}[2]{\irow{\entry{#1}}{#2}}
       \providecommand{\getcolentry}\getentryi
       
       \providecommand{\getvecentry}[2]{\getentryi{\vec #1}{#2}}

       \providecommand{\discolvecitwo}[1]{\discolvectwo{\vecentry{#1}1}{\vecentry{#1}2}}
       
       \providecommand{\discolvecintwo}\discolvecitwo%

       \providecommand{\dismatof}[2][r]{\begin{bmatrix}[#1]#2\end{bmatrix}}

       \providecommand{\matentry}[3]{\ijrowcol{#1}{#2}{#3}}

       \providecommand{\block}[5]{\ijrowcol{#1}{\ifx#2#3{\rowof{#2}}\else\rowof{{#2}\dotsc{#3}}\fi}{\ifx#4#5{\rowof{#4}}\else\rowof{{#4}\dotsc{#5}}\fi}}
       \providecommand{\colblock}[3]{\getvecentry{#1}{\ifx#2#3{#2}\else\fromto{#2}{#3}\fi}}

       \providecommand{\dismatskeldots}[4]{
         \dismatof[c]{
           #1&\dotsc&#3
           \\
           \vdots & \ddots &\vdots
           \\
           #2&\dotsc&#4
         }
       }
       \providecommand{\dismatcommfromtofromto}[5]{
         \dismatskeldots{#1#2#4}{#1#3#4}{#1#2#5}{#1#3#5}
       }
       \providecommand{\dismatcustfromtofromto}[6][matentry]{
         \dismatcommfromtofromto{\csname#1\endcsname{#2}}#3#4#5#6
       }
       \providecommand{\dismatcustfromtofromto}[6][matentry]{
         \dismatskeldots{%
           \csname#1\endcsname{#2}{#3}{#4}%
         }{%
           \csname#1\endcsname{#2}{#3}{#6}%
         }{%
           \csname#1\endcsname{#2}{#5}{#4}%
         }{%
           \csname#1\endcsname{#2}{#5}{#6}%
         }%
       }%
       \providecommand{\dismatcustfromtofromto}[6][matentry]{
         \dismatof{
           \csname#1\endcsname{#2}{#3}{#4}&\dotsc&\csname#1\endcsname{#2}{#3}{#6}
           \\
           \vdots & \ddots &\vdots
           \\
           \csname#1\endcsname{#2}{#5}{#4}&\dotsc&\csname#1\endcsname{#2}{#5}{#6}
         }
       }

       \providecommand{\dissysaxbdotsnm}[5]{\begin{matrix}[r]%
           \matentry{#1}11\vecentry{#2}1&+\dotsb&+\matentry{#1}1{#5}\vecentry{#2}{#5}
           &
           =
           \ifx|#3|0\else{\vecentry {#3}1}\fi
           \\
           \dotsb
           \\
           \matentry{#1}{#4}1\vecentry{#2}1&+\dotsb&+\matentry{#1}{#4}{#5}\vecentry{#2}{#5}
           &
           =
           \ifx|#3|0\else{\vecentry {#3}{#4}}\fi
       \end{matrix}}
       \providecommand{\seqof}[1]{\qp{#1}}%
       \providecommand{\seqs}[2]{\seqof{#1}_{#2}}
       \providecommand{\sets}[2]{\setof{#1}_{#2}}%
       \providecommand{\seqi}[3][]{\seqs{#2_{#3}}{\ifx|#1|{#3}\else{{#3}\in{#1}}\fi}}%
       \providecommand{\sequ}[3][]{\seqs{#2^{#3}}{\ifx|#1|{#3}\else{{#3}\in{#1}}\fi}}%
       \providecommand{\subseqi}[4][]{\seqs{#2_{{#3}_{#4}}}{\ifx|#1|{#4}\else{{#4}\in{#1}}\fi}}%
       
       \providecommand{\seqsinat}[2]{\seqsi{#1}{#2}{\naturals}}

       \providecommand{\seti}[3][]{\sets{#2_{#3}}{\ifx|#1|_{#3}\else_{{#3}\in{#1}}\fi}}%
       \providecommand{\setu}[3][]{\sets{#2^{#3}}{\ifx|#1|{#3}\else{{#3}\in{#1}}\fi}}%
       \providecommand{\seqsi}[3]{\seqi[#3]{#1}{#2}}

       \let\liminf\relax
       \DeclareMathOperator*{\liminf}{liminf}
       \let\limsup\relax
       \DeclareMathOperator*{\limsup}{limsup}
       \providecommand{\limofat}[3][]{\ensuremath{\lim_{\ifx|#1|{}\else{#1\ni}\fi#3}{#2}}}
       \providecommand{\limsupofat}[3][]{\ensuremath{\limsup_{\ifx|#1|{}\else{#1\ni}\fi#3}{#2}}}
       \providecommand{\liminfofat}[3][]{\ensuremath{\liminf_{\ifx|#1|{}\else{#1\ni}\fi#3}{#2}}}

       \providecommand{\stringdotsfrom}[3][]{\ensuremath{#2\ifx|#1|\else#1\fi\,#3\ifx|#1|\else#1\fi\,\dotsc}}
       \providecommand{\listdotsfrom}[3][]{\ensuremath{#2\ifx|#1|\else#1\fi,#3\ifx|#1|\else#1\fi,\dotsc}}
       \providecommand{\stringdotsfromto}[3][]{\ensuremath{#2\ifx|#1|\else#1\fi\,\dotsc\,#3\ifx|#1|\else#1\fi}}
       \providecommand{\listdotsfromto}[3][]{\ensuremath{#2\ifx|#1|\else#1\fi,\dotsc,#3\ifx|#1|\else#1\fi}}
       \providecommand{\listifromto}[5][]{\ensuremath{{#2}_{#3}\ifx|#1|\else#1\fi},\text{ for }\ensuremath{\rangefromto{#3}{#4}{#5}}\xspace}
       \providecommand{\listufromto}[5][]{\ensuremath{{#2}^{#3}\ifx|#1|\else#1\fi},\text{ for }\ensuremath{\rangefromto{#3}{#4}{#5}}\xspace}
       \providecommand{\listitwo}[2][]{\ensuremath{#2_1\ifx|#1|\else#1\fi,#2_2\ifx|#1|\else#1\fi}}
       \providecommand{\listutwo}[2][]{\ensuremath{#2^1\ifx|#1|\else#1\fi,#2^2\ifx|#1|\else#1\fi}}
       \providecommand{\listithree}[2][]{\ensuremath{#2_1\ifx|#1|\else#1\fi,#2_2\ifx|#1|\else#1\fi,#2_3\ifx|#1|\else#1\fi}}
       \providecommand{\listithreez}[2][]{\ensuremath{#2_0\ifx|#1|\else#1\fi,#2_1\ifx|#1|\else#1\fi,#2_2\ifx|#1|\else#1\fi}}
       \providecommand{\listifourz}[2][]{\ensuremath{#2_0\ifx|#1|\else#1\fi,#2_1\ifx|#1|\else#1\fi,#2_2\ifx|#1|\else#1\fi,#2_3\ifx|#1|\else#1\fi}}
       \providecommand{\listuthree}[2][]{\ensuremath{#2^1\ifx|#1|\else#1\fi,#2^2\ifx|#1|\else#1\fi,#2^3\ifx|#1|\else#1\fi}}

       \providecommand{\sums}[2]{\ensuremath{\sum_{#1\in#2}}}

       \providecommand{\jump}[2][]{\ensuremath{\left\llbracket #2\right\rrbracket\ifx|#1|{}\else_{#1}\fi}}

       \providecommand{\fromto}[2]{\ensuremath{\setof{#1\dotsc#2}}}%

       \providecommand{\integerbetween}[2]{\ensuremath{={#1},\dotsc,{#2}}}

       \providecommand{\rangefromto}[3]{\ensuremath{#1\integerbetween{#2}{#3}}}

       \providecommand{\e}{\ensuremath{\operatorname{e}\!}\xspace}
       \providecommand{\ic}{\ensuremath{\operatorname{i}}\xspace}

       \providecommand{\d}{}
       \renewcommand{\d}[1][]{\ensuremath{\operatorname{d}\!\ifx|#1|\else{_{#1}}\fi}}
       
       \providecommand{\ds}[1][]{\d{\measure S}}
       \providecommand{\D}[1][]{\ensuremath{\operatorname{D}\!\ifx|#1|\else{_{#1}}\fi}}

      \providecommand{\registered}%
      {\ensuremath{^\text{\textregistered}}}

      \providecommand{\tand}{\ensuremath{\text{ and }}}

      \providecommand{\ballname}{\operatorname{B}}              %
      \providecommand{\Ball}[3][]{\ballname_{{#3}}^{#1}\!\qp{{#2}}}
      \providecommand{\ball}{\Ball}

      \providecommand{\card}{\ensuremath{\#}}

      \providecommand{\constant}[1]{\ensuremath{C_{#1}}}
      \providecommand{\constext}[2][]{\constant{\textup{#2}{\ifx|#1|{}\else{,\ensuremath{#1}}\fi}}}            %
      \providecommand{\constref}[2][]{\ensuremath{\constant{\textup{\ref{#2}{\ifx|#1|{}\else{,\ensuremath{#1}}\fi}}}}}
      \providecommand{\constdef}[2][]{\label{#2}\ensuremath{\constant{\textup{\ref{#2}{\ifx|#1|{}\else{,\ensuremath{#1}}\fi}}}}}

      \providecommand{\funkref}[3][]{\ensuremath{{#3}_{\textup{\ref{#2}{\ifx|#1|{}\else{,\ensuremath{#1}}\fi}}}}}

      \providecommand{\diam}{\operatorname{diam}}
      
      \providecommand{\curl}{\operatorname{curl}}
      \renewcommand{\curl}[1][]{\nabla\ifx|#1|{}\else\kern-2pt_{#1}\fi\kern-2pt\vecprod}
      \renewcommand{\div}[1][]{\nabla\ifx|#1|{}\else\kern-2pt_{#1}\fi\kern-1pt\inner}
      \providecommand{\divof}[2][]{\div[#1]\ifx|#2|{}\else\qb{#2}\fi}%
      \providecommand{\divideabyb}[2]{\operatorname{div}(a,b)}
      \providecommand{\grad}{}
      \renewcommand{\grad}[1][]{\nabla\ifx|#1|\else_{#1}\fi}

      \providecommand{\rot}[1][]{\nabla\ifx|#1|\else_{#1}\fi\outerp}
      
      \providecommand{\rowdiv}[1][]{\D\ifx|#1|{}\else\kern-1pt_{#1}\kern-2pt\fi\cdot}
      \providecommand{\rowdivof}[2][]{\rowdiv[#1]\ifx|#2|{}\else\qb{#2}\fi}

      \providecommand{\inv}[1][]{\operatorname{inv}\ifx|#1|\else^{#1}\fi}
      \providecommand{\ivt}[1]{\operatorname{ivt}\ifx|#1|\else^{#1}\fi}
      \providecommand\tensorinvariant\ivt
      
      \providecommand{\inverse}[2][]{\powp[#1]{-1}{#2}}
      \providecommand{\inverseqp}[1]{\inverse{\qp{#1}}}
      \providecommand{\inverseof}[1]{\inverseqp{#1}}

      \providecommand{\mod}{}
      \renewcommand{\mod}[1][]{\operatorname{mod}\ifx|#1|\else\kern-1pt_{#1}\fi}
      
      \let\oldfrac\frac
      \renewcommand{\frac}[3][]{\ifx|#1|\oldfrac{#2}{#3}\else\begin{array}{#1}{#2}\\\hline{#3}\end{array}\fi}
      \providecommand{\fracl}[3][]{\ifx|#1|\nicefrac{#2}{#3}\else{#2}#1/{#3}\fi}

      \providecommand{\qpfracl}[3][]{\qp{\ifx|#1|\fracl{#2}{#3}\else{#2}#1/{#3}\fi}}
      \providecommand{\qpfrac}[3][]{\qp{\ifx|#1|\frac{#2}{#3}\else{#2}#1/{#3}\fi}}
      
      \providecommand{\absfracl}[3][]{\abs{\ifx|#1|\fracl{#2}{#3}\else{#2}#1/{#3}\fi}}
      \providecommand{\absfrac}[3][]{\abs{\ifx|#1|\frac{#2}{#3}\else{#2}#1/{#3}\fi}}
      \providecommand{\fraclff}[3][]{\ifx|#1|{#2}/{#3}\else{#1}\fracl{#2}{#3}\fi}

      \providecommand{\eye}[1][]{\vec{\mathrm I}\ifx|#1|{}\else_{#1}\fi}%
      \providecommand{\numeye}[1][]{\boldsymbol{\mathsf{I}}\ifx|#1|{}\else_{#1}\fi}%
      \providecommand{\Eye}[1]{
        \begin{bmatrix}
        \ifthenelse{#1>1}{
          \ifthenelse{#1>2}{
            \ifthenelse{#1>3}{
              \ifthenelse{#1>4}{
                1&\zeroentry&\dotso&\zeroentry
                \\
                \zeroentry&1&\dotso&\zeroentry
                \\
                \vdots&\vdots&\ddots&\vdots
                \\
                \zeroentry&\zeroentry&\dotso&1
              }{        
                1&\zeroentry&\zeroentry&\zeroentry
                \\
                \zeroentry&1&\zeroentry&\zeroentry
                \\
                \zeroentry&\zeroentry&1&\zeroentry
                \\
                \zeroentry&\zeroentry&\zeroentry&1
              }
            }{
              1&\zeroentry&\zeroentry
              \\
              \zeroentry&1&\zeroentry
              \\
              \zeroentry&\zeroentry&1
            }
          }{
            1&\zeroentry
            \\
            \zeroentry&1
          }
        }{
          1
        }
        \end{bmatrix}
      }

      \providecommand{\lebmeas}[1][]{\measure L^{#1}}     %
      \providecommand{\lebmeasof}[2][]{\ifx|#1|\left|#2\right|\else\lebmeas[#1]\qp{#2}\fi}         %
      \providecommand{\meshsize}[1][]{h\ifx|#1|\else_{#1}\fi}
      \providecommand{\maxi}[2]{#1\vee#2}                       %

      \providecommand{\mini}[2]{#1\wedge#2}                     %

      \providecommand{\argmin}{\operatorname{argmin}\nolimits}
      
      \providecommand{\argmax}{\operatorname{argmax}\nolimits}
      \providecommand{\Argmax}{\operatorname{Argmax}\nolimits}
      \let\oldneg\neg
      \renewcommand{\neg}[1]{\left[#1\right]_-}
      \providecommand{\dash}[1][']{\ifthenelse{\equal{#1}{'}\OR\equal{#1}{''}}{#1}{^{(#1)}}}
      
      \providecommand{\pdfrac}[2][]{\ensuremath{\frac{\partial\ifx|#1|\phantom{#2}\else{#1}\fi}{\partial{#2}}}} %
      \providecommand{\pdfracpow}[3][]{\ensuremath{\frac{\partial^{#3}\ifx|#1|\phantom{#2}\else{#1}\fi}{\partial{#2}^{#3}}}} %

      \providecommand{\pd}[2][]{\ensuremath{\partial_{#2}}{\ifx|#1|{}\else{\qb{#1}}\fi}} %

      \providecommand{\dd}[2][]{\ensuremath{\ifx|#1|\frac{\d}{\d{#2}}\else\frac[l]{\d{#1}}{\d{#2}}\fi}}    %

      \renewcommand{\Im}{\operatorname{im}}                 %
      \renewcommand{\Re}{\operatorname{re}}                 %
      \providecommand{\imaginpart}[1][]{\Im{\ifx|#1|{}\else\qp{#1}\fi}} %
      \providecommand{\realpart}[1][]{\Re{\ifx|#1|{}\else\qp{#1}\fi}} %
      \providecommand\determinant\det
      
      \providecommand{\trace}{\operatorname{tra}}             %
      \providecommand{\traceof}[1]{\trace\qp{#1}}             %

      \providecommand{\transpose}{\intercal}%

      \providecommand{\Transpose}[1]{\ensuremath{{#1}^{\transpose}}}

      \providecommand{\Transposemat}[1]{\Transpose{\mat{#1}}}

      \providecommand{\transposemat}[1]{\Transposemat{#1}}

      \providecommand{\orthogonalto}[1][]{\ensuremath{\perp\ifx|#1|{}\else{_{#1}}\fi}}
      
      \providecommand{\rowof}[1]{\ensuremath{\vecof{#1}}}

      \providecommand{\discolvec}[2][r]{\ensuremath{\begin{bmatrix}[#1]#2\end{bmatrix}}}

      \providecommand{\discolvectwo}[3][r]{\ensuremath{\discolvec[#1]{#2\\#3}}}

      \providecommand{\discolvecitwo}[1]{\discolvectwo{\vecentry{#1}1}{\vecentry{#1}2}}

      \provideboolean{showzeroentries}%
      \setboolean{showzeroentries}{true}%
      \providecommand{\zeroentry}{\ifthenelse{\boolean{showzeroentries}}{{0}}{\phantom0}}
      \providecommand{\zeroentrywarning}{\ifthenelse{\boolean{showzeroentries}}{}{%
          \ensuremath{\text{($0$ entries omitted)\xspace}}}}

      \providecommand{\smint}{\ensuremath{{\text{\textbf{/}}}\kern-.75em\smallint}}
      \renewcommand{\smint}[1][]{\lower12.3pt\hbox{\begin{tikzpicture}\draw[line width=.75pt] (-3pt,-0.5)--(1pt,-0.5) node[pos=0.6]{$\int$};\path (3pt,-24pt)node {\scriptsize $#1$};\end{tikzpicture}}}

      \providecommand{\lap}{\ensuremath{\Delta}}
      \providecommand{\lapin}[1][]{\lap\ifx|#1|\else_{#1}\fi}
      \providecommand{\normalsymbol}{\operatorname{\mathbf{n}}}
      \renewcommand{\normalsymbol}{\operatorname{\mathbf{n}}}
      \providecommand{\normal}[1][]{\normalsymbol\ifx|#1|\else_{#1}\fi}%
      \providecommand{\normalto}[2][]{\ensuremath{\normal[#2]\ifx|#1|\else\qp{#1}\fi}}
      \providecommand{\normalder}[1][]{\ensuremath{\normal\ifx|#1|\else\qp{#1}\fi{\inner\grad}}}

      \providecommand{\tangentialsymbol}{\operatorname{\textbf{t}}}
      \providecommand{\tangentialto}[2][]{\tangentialsymbol\ifx|#1|\else^{#1}\fi\ifx|#2|\else_{#2}\fi}

      \providecommand{\intersected}{\ensuremath{\cap}}
      
      \providecommand{\meet}{\intersected}

      \providecommand{\union}[1]{\ensuremath{\bigcup\nolimits_{#1}}}

      \providecommand{\unions}[3][]{\union{#2\in{#3}\ifx|#1|\else:#1\fi}}

      \providecommand{\powersetof}[1]{\ensuremath{2^{#1}}}

      \ifthenelse{\boolean{usesvjour}}{}{
        \let\vec\undefined
        \providecommand{\vec}[1]{\ensuremath{\boldsymbol{#1}}}
        \renewcommand{\vec}[1]{\ensuremath{\boldsymbol{#1}}}
      }

      \providecommand{\tildevec}[1]{\tilde{\vec{#1}}}
      \providecommand{\hatmat}[1]{\hat{\mat{#1}}}

      \providecommand{\geomat}[1]{\vec{\UCmath{#1}}}

      \providecommand{\tildevec}[1]{\ensuremath{\tilde{\vec{#1}}}}
      
      \providecommand{\mat}[1]{\geomat{#1}} %
      \providecommand{\Prob}[1][]{\ensuremath{\operatorname{Prob}\ifx|#1|{}\else_{#1}\fi}}

      \providecommand{\pdf}[2][]{\ensuremath{\operatorname{pdf}_{#2\ifx|#1|{}\else{\conditionalto{#1}}\fi}}\xspace}

      \providecommand{\expectation}{\ensuremath{\operatorname{E}}}
      \providecommand{\EX}[1][]{\ensuremath{\expectation\ifx|#1|{}\else_{#1}\fi}}

      \providecommand{\gausskernel}[3][x]{%
        \ensuremath{
          \exp\frac{-\if#20{#1}\else(#1-\mu)\fi^2}{%
            2\if#31{}\else\powp2{#3}\fi}%
        }%
      }
      \providecommand{\gaussdistribution}[3][x]{%
        \ensuremath{\frac1{\sqrt{2\pic}\if#31{}\else#3\fi}%
          \gausskernel[#1]{#2}{#3}
        }%
      }%

      \providecommand{\boundary}{\partial}

      \providecommand{\sogroup}[1]{\operatorname{SO}(#1)}

      \providecommand{\PD}[1]{\operatorname{PD}\qpreg{#1}}
      \providecommand{\pdspace}[1]{\PD{\linspace v}}
      \providecommand{\pdmats}[2][F]{\PD{\csname#1\endcsname{#2}}}
      
      \providecommand{\SPD}{\operatorname{SPD}}
      \providecommand{\spdmats}[2][F]{\SPD(\csname#1\endcsname{#2})}

       \providecommand{\Continuous}{\ensuremath{\operatorname C}\xspace}%
       \providecommand{\Hspace}{\ensuremath{\operatorname H}\xspace}
       \providecommand{\Lebesgue}{\ensuremath{\operatorname L}\xspace}
       \providecommand{\Besovspace}{\ensuremath{\operatorname B}\xspace}

       \providecommand{\Weaklyder}{\ensuremath{\operatorname W}\xspace}
       
       \providecommand{\dual}[1]{\ensuremath{{#1}'}}
       
       \providecommand{\dualspace}[2][]{\dual{\linspace{#2}\ifx|#1|\else{_{#1}}\fi}}
       \providecommand{\bidual}[1]{\ensuremath{{#1}''}}
       \providecommand{\bidualspace}[2][]{\bidual{\linspace{#2}\ifx|#1|\else{_{#1}}\fi}}

       \providecommand{\cont}[1]{\ensuremath{\Continuous^{#1}}}

       \providecommand{\diff}[2][]{\ensuremath{\cD\ifx|#1|\else^{#1}\fi(#2)}}
       
       \providecommand{\BV}[1]{\ensuremath{\operatorname{BV}}}

       \providecommand{\leb}[1]{\ensuremath{\Lebesgue_{#1}}}

       \providecommand{\lebloc}[1]{\ensuremath{{{\Lebesgue}^{\kern-.20em\lower .1ex\hbox{\tiny\textrm{\textup{loc}}}}_{#1}}}}
       \providecommand{\lebnorm}[3][]{\ensuremath{\Norm{#2}_{\leb{#3}\ifx|#1|{}\else(#1)\fi}}}
       \providecommand{\bes}[3][]{\ensuremath{\Besovspace^{#2}_{#3\ifx|#1|\else,#1\fi}}}

       \providecommand{\sob}[2]{\ensuremath{{\smash\Weaklyder}^{#1}_{#2}}}

       \providecommand{\sobh}[1]{\ensuremath{\Hspace^{#1}}}
       \providecommand{\vecsobh}[1]{\ensuremath{\vec\Hspace^{#1}}}
       \ProvideDocumentCommand{\hdiv}{ O{} O{}}{\vecsobh{\operatorname{div}}\ifx+#1+\else_{0|#1}\fi\ifx|#2|\else(#2)\fi}
       \providecommand{\hcurl}[1][]{\vecsobh{\operatorname{curl}}\ifx|#1|\else(#1)\fi}
       
       \providecommand{\sobhz}[2][]{\sobh{#2}_{0\ifx+#1+\else|#1\fi}}

       \providecommand{\Lip}[1][]{\ensuremath{\operatorname{Lip}}\ifx|#1|{}\else{\qp{#1}}\fi}

       \ProvideDocumentCommand{\polyring}{ O{X} O{A} }{\ring[#2][#1]}
       \ProvideDocumentCommand{\polyfield}{ O{X} O{} }{\field[#2][#1]}
       \providecommand{\polyreals}[1][]{\polyfield[][R]\ifx|#1|\else^{#1}\fi}
       \providecommand{\poly}[2][]{\ensuremath{\rP\ifx#1\else_{#1}\fi^{#2}}}

       \providecommand{\Symmatrices}[2][R]{\ensuremath{\operatorname{Sym}{(\csname#1\endcsname{#2})}}}
       
       \providecommand{\SAmatrices}[2][F]{\ensuremath{\operatorname{SA}{(\csname#1\endcsname{#2})}}}
       \providecommand{\mesh}[2][]{\ensuremath{\mathcalbf{\MakeUppercase{#2}}\ifx|#1|\else_{#1}\fi}}

      \providecommand{\crouzeixraviart}[1][1]{\operatorname{CR}\ifx|#1|{}\else{^{#1}}\fi}

      \providecommand{\linspace}[1]{\mathscript{\MakeUppercase{#1}}}

      \providecommand{\linop}[1]{\mathcalbf{\MakeUppercase{#1}}}

      \providecommand{\nlop}[1]{\mathcalbf{\MakeUppercase{#1}}}

      \providecommand{\Lin}{\operatorname{Lin}}
      \providecommand{\CL}{\operatorname{CL}}
      \providecommand{\linops}[3][]{\ensuremath{\Lin\ifx|#1|\else^{#1}\fi\qp{{#2}\to{#3}}}}
      \providecommand{\linopss}[3][]{\linops[#1]{\linspace{#2}}{\linspace{#3}}}

      \providecommand{\clinops}[3][]{\ensuremath{\CL\ifx|#1|\else^{#1}\fi\qp{{#2}\to{#3}}}}
      
      \providecommand{\fepartition}[2][]{\mathscript{\MakeUppercase{#2}}\ifx|#1|{}\else_{#1}\fi}
      \providecommand{\fespace}[2][]{\mathbb{\MakeUppercase{#2}}\ifx|#1|{}\else_{#1}\fi}
      \providecommand{\hatfespace}[2][]{\widehat{\mathbb{\MakeUppercase{#2}}}\ifx|#1|{}\else_{#1}\fi}

      \providecommand{\vespace}[1][]{\fespace v\ifx|#1|\else_{#1}\fi}
      \providecommand{\hatvespace}[1][]{\hatfespace v\ifx|#1|\else_{#1}\fi}

      \providecommand{\fe}[2][]{\ensuremath{\UCmath{#2}\ifx|#1|\else_{#1}\fi}}%

      \providecommand{\vecfe}[2][]{\ensuremath{\vec{\fe{#2}}\ifx|#1|{}\else{_{#1}}\fi}}%
      
      \providecommand{\matfe}[2][]{\ensuremath{\mat{\fe{#2}}\ifx|#1|{}\else{_{#1}}\fi}}%
      
      \providecommand{\hatmatfe}[2][]{\ensuremath{\hatmat{\UCmath{#2}}\ifx|#1|{}\else{_{#1}}\fi}}%
      
      \providecommand{\EOC}{\ensuremath{\operatorname{EOC}}\xspace}
      \providecommand{\tol}{\ensuremath{\operatorname{tol}}\xspace}

      \providecommand{\Forall}{\:\forall\:}
      \providecommand{\Exists}{\:\exists\:}
      
      \providecommand{\Foreach}{\text{ for each }}%
      
      \providecommand{\emptyset}{\varnothing}
      \renewcommand{\emptyset}{\varnothing}

      \providecommand{\gets}{\mapsfrom}
      \renewcommand{\gets}{\mapsfrom}
      \renewcommand{\gets}{\leftarrow}
      \RequirePackage{amscd}

      \providecommand{\funk}[3]{\ensuremath{#1:#2\to#3}}

      \providecommand{\isomorphicto}{\leftrightarrows}
      \providecommand\isomorphic\isomorphicto

      \providecommand{\implies}{\ensuremath{\:\Rightarrow\:}\xspace}
      \renewcommand{\implies}{\ensuremath{\:\Rightarrow\:}\xspace}

      \providecommand{\nty}[1][n]{\ensuremath{#1\to\infty}}

      \providecommand{\restriction}[2]{\left.#1\right|_{#2}}
      \renewcommand{\restriction}[2]{\left.#1\right|_{#2}}

      \providecommand{\evalat}[3][]{\qb{#2}_{\ifx|#1|{}\else#1=\fi#3}}
      \providecommand{\evaldiff}[4][]{\qb{#2}^{\ifx|#1|{}\else#1=\fi#3}_{\ifx|#1|{}\else#1=\fi#4}}

      \providecommand{\aka}[1]{(also known as {#1})\xspace}
      
      \providecommand{\akaindexemph}[2][]{\aka{\indexemph[#1]{#2}}}

      \providecommand{\bs}{\char '134}   %

      \providecommand{\Program}[1]{\textsf{#1}\xspace}

      \providecommand{\matlabplot}[2][]{%
        \begin{center}
          \includegraphics[width=0.9375\linewidth,trim=64 200 64 200,clip]{#2}%
          \ifx|#1|\else\\#1\fi
        \end{center}%
      }

      \providecommand{\texcommand}[1]{\texttt{\bs{\nolinkurl{#1}}}\xspace}
      
      \providecommand{\codename}[1]{\nolinkurl{#1}\xspace}
      \providecommand{\colorvarname}[2][a]{\colorvar[#1]{\Verb{#2}}}
      
      \providecommand{\codevarname}[1]{\colorvarname[a]{#1}}
      
      \providecommand\olco\codevarname

      \ifthenelse{\boolean{useutopia}}{%
        
      }{%
        
      }

      \providecommand{\matlab}{{\small\Program{MATLAB}}\xspace}%
      \providecommand\MATLAB\matlab

      \ProvideDocumentCommand{\codesnip}{ O{.} O{1.0} m}{%
        \newline
        \begin{minipage}{#2\linewidth}
          \lstinputlisting{#1/#3}
        \end{minipage}
      }
      \providecommand{\codeprint}[2][.]{
        \ \newline
        \begin{minipage}{\linewidth}
          \lstinputlisting{#1/#2}
          \framebox{Contents of file %
            \ifthenelse{\isundefined\pickuppath}{%
             \codename{#2}%
            }{%
              \providecommand{\fullpickuppath}{}%
              \renewcommand{\fullpickuppath}{\pickuppath/\ifx|#1|\else#1/\fi#2}%
              \href{\fullpickuppath}{\codename{#2}}%
          }}
        \end{minipage}
      }
      \providecommand{\codenoprint}[2][.]{
              \providecommand{\fullpickuppath}{}%
              \renewcommand{\fullpickuppath}{\pickuppath/\ifx|#1|\else#1/\fi#2}%
              \href{\fullpickuppath}{\codename{#2}}%
      }

      \providecommand{\indexen}[2][]{{\ifthenelse{\boolean{shownotes}}{\color b}{}#2\ifx|#1|\index{#2}\else\index{#1}\fi}}
      \providecommand{\indexemph}[2][]{\emph{\indexen[#1]{#2}}}
      \providecommand{\indexma}[2][]{{\ifthenelse{\boolean{shownotes}}{\color b}{}#2\ifx|#1|\index{\(#2\)}\else\index{<#1@\(#2\)}\fi}}

      \providecommand{\ListParameters}{}
      \renewcommand{\ListParameters}%
      {
      	 \setlength{\topsep}{0pt}
      	 \setlength{\leftmargin}{0pt}
               \setlength{\itemsep}{0pt}
      	 \setlength{\parsep}{0pt}
      	 \setlength{\parskip}{0pt}
               \setlength{\labelsep}{0pt}
      	 \setlength{\itemindent}{0pt}
      }
      {%
        \begin{list}%
          {}%
          {\ListParameters%
          
      }}%
      {\end{list}}
      \newcounter{tmpcounter}
      \newcounter{LetterListItem}
      \renewcommand{\theLetterListItem}{(\alph{LetterListItem})}

      \newcounter{CapitalListItem}
      \renewcommand{\theCapitalListItem}{\Alph{CapitalListItem}.}

      \newcounter{NumberListItem}
      \renewcommand{\theNumberListItem}{\arabic{NumberListItem}}
      {
      	\begin{list}%
      	{\theNumberListItem.\ }%
      	{\usecounter{NumberListItem}%
      	 \ListParameters
      	}
      }%
      {\end{list}}
      \newcounter{QuestionListItem}
      \renewcommand{\theQuestionListItem}{\textbf{Question \arabic{QuestionListItem}}}
      {
      	\begin{list}%
      	{\theQuestionListItem.\ }%
      	{\usecounter{QuestionListItem}%
      	 \ListParameters
      	}
      }%
      {\end{list}}
      \newcounter{RomanListItem}
      \renewcommand{\theRomanListItem}{(\roman{RomanListItem})}
      {
      	\begin{list}%
      	{\theRomanListItem\ }%
      	{\usecounter{RomanListItem}
      	 \ListParameters
      	}
      }%
      {\end{list}}
      \newcounter{StepsItem}
      {
      	\begin{list}%
      	{Step \theStepsItem.\ }%
      	{\usecounter{StepsItem}%
      	 \ListParameters
      	}
      }%
      {\end{list}}
      \newcounter{CasesListItem}
      \renewcommand{\theCasesListItem}{\Alph{CasesListItem}}
      {
      	\begin{list}%
      	{\emph{Case \theCasesListItem.}\ }%
      	{\usecounter{CasesListItem}%
      	 \ListParameters
      	}
      }%
      {\end{list}}
      \newcounter{QAListItem}
      \renewcommand{\theQAListItem}{Q\arabic{QAListItem}:}
      {
      	\begin{list}%
      	{\theQAListItem}%
      	{\usecounter{QAListItem}
      	 \ListParameters
      	}
      }%
      {\end{list}}

      \ifthenelse{\boolean{isthesis}}{%
        \setcounter{secnumdepth}{1}%
      }{
        \setcounter{secnumdepth}{2} %
      }
      \providecommand{\ListParameters}{}
      \renewcommand{\ListParameters}
      {
      	 \setlength{\topsep}{0em}
      	 \setlength{\leftmargin}{0em}
               \setlength{\itemsep}{0ex}
      	 \setlength{\parsep}{.5ex}
      	 \setlength{\itemindent}{\labelsep}
      	 \addtolength{\itemindent}{\labelwidth}
      }

        \providecommand{\ObsName}{Remark}%
        \providecommand{\RemName}{Remark}%
        \providecommand{\NotName}{Notation}%
        \providecommand{\BFNName}{Big~Fat~Note}%
        \providecommand{\DefName}{Definition}%
        \providecommand{\ExaName}{Example}%
        \providecommand{\TheName}{Theorem}%
        \providecommand{\LemName}{Lemma}%
        \providecommand{\ProName}{Proposition}%
        \providecommand{\CorName}{Corollary}%
        \providecommand{\PbmName}{Problem}%
        \providecommand{\HypName}{Hypothesis}%
        \providecommand{\AlgName}{Algorithm}%
        \providecommand{\ExeName}{Exercise}%
        \providecommand{\SolName}{Solution}%
        \providecommand{\ClaName}{Claim}%
        \providecommand{\EsyName}{Essay}%
        \providecommand{\Proofname}{Proof}%
        \providecommand{\Derivename}{Derivation}%
      
      \ifthenelse{\boolean{isthesis}}{%
        \providecommand{\Thecounter}{The}
      }{%
        \providecommand{\Thecounter}{subsection}
      }
      \newcommand{\oltikzgetxy}[3]{%
        \tikz@scan@one@point\pgfutil@firstofone#1\relax
        \edef#2{\the\pgf@x}%
        \edef#3{\the\pgf@y}%
      }
      \providecommand{\pdfformat}[1]{
         \provideboolean{pdfoutput}
         \setboolean{pdfoutput}{#1}%
        \ifthenelse{\boolean{pdfoutput}}{
          \typeout{using pdf}
\makeatletter
\usepackage{pdfsync}
          \providecommand{\graphext}{pdf}
          \renewcommand{\graphext}{pdf}
          \providecommand{\graphextex}{pdf_t}
          \renewcommand{\graphextex}{pdf_t}
        }{
          \typeout{using eps}
          \RequirePackage[dvips]{graphicx,xcolor}
          \providecommand{\graphext}{eps}
          \renewcommand{\graphext}{eps}
          \providecommand{\graphextex}{eps_t}
          \renewcommand{\graphextex}{eps_t}
        }
        \RequirePackage{epsfig}
        \RequirePackage{tikz}
        \RequirePackage{rotating}
\makeatletter
        \RequirePackage{graphicx}
        \RequirePackage{xcolor}
        \provideboolean{darkcolortheme}
        \definecolor{SussexFlint}{rgb}{.00,.19,.21}
        \definecolor{SussexGrey}{rgb}{.51,.58,.49}
        \definecolor{SussexOrange}{rgb}{.94,.29,.00}
        \definecolor{SussexYellow}{rgb}{1.00,.73,.00}
        \definecolor{SussexRed}{rgb}{.94,.01,.49}
        \definecolor{SussexPurple}{rgb}{.48,.06,.44}
        \definecolor{SussexGreen}{rgb}{.00,.58,.46}
        \definecolor{OmarGreen}{rgb}{.00,.68,.36}
        \definecolor{SussexBlue}{rgb}{.00,.58,.65}
        \definecolor{OmarBlue}{rgb}{.00,.38,.65}
        \colorlet{a}{OmarBlue}%
        \colorlet{b}{SussexOrange}
        \colorlet{c}{SussexGreen}
        \colorlet{d}{SussexPurple}%
        \colorlet{e}{SussexRed}
        \colorlet{f}{SussexYellow}
        \colorlet{g}{white}%
        \colorlet{h}{SussexGrey}%
        \colorlet{i}{black}%
        \colorlet{j}{SussexFlint}
        \colorlet{colora}{a}
        \colorlet{colorb}{b}
        \colorlet{colorc}{c}
        \colorlet{colord}{d}
        \colorlet{colore}{e}
        \colorlet{colorf}{f}
        \colorlet{colorg}{g}
        \colorlet{colorh}{h}
        \colorlet{colori}{i}
        \colorlet{colorj}{j}
        \newcommand{\mausDarkColorTheme}{
          \colorlet{a}{SussexYellow!50!yellow}
          \colorlet{b}{SussexBlue}%
          \colorlet{c}{SussexRed!50!red}
          \colorlet{d}{SussexOrange!50!yellow}
          \colorlet{e}{SussexGreen!50!green}
          \colorlet{f}{SussexPurple!50!magenta}
          \colorlet{g}{black}%
          \colorlet{h}{SussexFlint!50!black}
          \colorlet{i}{white}%
          \colorlet{j}{SussexGrey}
        }
        \ifthenelse{\boolean{darkcolortheme}}{\mausDarkColorTheme}{}
\makeatletter
      }
      \providecommand{\solution}{\textbf{\SolName.}\xspace}

      \newcounter{phantomedinput}
      \newcounter{phantombox}
      \counterwithin*{phantombox}{phantomedinput}%
      \provideboolean{showphantoms}
      \renewcommand{\thephantombox}{\Alph{phantombox}}%
      \providecommand{\phantombox}[1]{\stepcounter{phantombox}%
        \ensuremath{\boxed{%
            {\ifthenelse{\boolean{showphantoms}}{#1}{\phantom{#1}}}%
            {\texttt{\tiny\ \colorbox{i!50}{\color g\thephantombox}}
            }%
          }%
        }%
      }

      \provideboolean{hidesolution}
      \newcommand{\consolution}[2][]{
        \ifthenelse{\boolean{hidesolution}}{#1\setboolean{showphantoms}{false}}{%
          {\setboolean{showphantoms}{true}\color{i!50}\par \small {\solution}\ #2\par\ \\[5pt]}}
      }
      \provideboolean{showmarks}
      \providecommand{\showmarks}[1]{%
        \ifthenelse{%
          \boolean{showmarks}}{%
          \marginpar{%
            \tiny [$#1$ mark\ifthenelse{\equal{#1}1}{\phantom{s}}s]}%
        }{}}%

      \newcommand{\condibreak}{\ifthenelse{\boolean{hidesolution}}{\clearpage}{}}
      
      \newcommand{\solutibreak}{\ifthenelse{\boolean{hidesolution}}{}{\clearpage}}
      \newcommand{\questionly}[1]{\ifthenelse{\boolean{hidesolution}}{#1}{}}
      \newcommand{\solutionly}[1]{\ifthenelse{\boolean{hidesolution}}{}{#1}}

       \providecommand{\qeyword}[1]{\index{#1}\ifthenelse{\boolean{shownotes}}{{\tiny\color e\colorbox{e!6.25}{#1}}}{}}
       \providecommand{\pathwordbase}{}
       \providecommand{\pathword}[2][]{%
         \label{#2}%
         \ifthenelse{\boolean{shownotes}}{%
           \ \\\index{#2@\tiny\codevarname{#2}}{\ensuremath{\tiny\color f\href{\pathwordbase/#2}{\colorvarname[f]{#2}}}}\\
         }{}%
       }
       \providecommand{\targword}[2][]{%
         \label{#2}%
         \ifthenelse{\boolean{shownotes}}{%
           \index{#2@\tiny\codevarname{#2}}{\ensuremath{\tiny\color f\href{\pathwordbase/#2}{\colorvarname[e]{#2}\ifx|#1|\else\colorvarname[e]{[#1]}\fi}}}\\
         }{}%
       }
       \providecommand{\sourceurl}[2][]{%
         \ifthenelse{\boolean{shownotes}}{{\ \\\tiny\colorbox{d!6.25}{\color d\texttt{source: \ifx|#1|\href{#2}{#1}\else\url{#2}\fi}}}}}
       \providecommand{\sourcecite}[2][]{\ifthenelse{\boolean{shownotes}}{{\ \\\tiny\colorbox{d!6.25}{\color d\texttt{source: \citet[#1]{#2}}}}}{%
       }}
       \providecommand{\conword}[2][]{\ifthenelse{\boolean{shownotes}}{#2}{#1}}
       \providecommand{\solword}[2][]{\ifthenelse{\boolean{hidesolution}}{#1}{#2}}
       \providecommand{\solghost}[1]{\ifthenelse{\boolean{showphantoms}}{#1}{\phantom{#1}}}

      \RequirePackage{lineno}
      \ifthenelse{\boolean{showchanges}}{
        \newcommand{\llabel}[1]{\hypertarget{llineno:#1}{\linelabel{#1}}}
        \newcommand{\lref}[1]{\hyperlink{llineno:#1}{\ref*{#1}}}
      }{
        \newcommand\llabel[1]{}
        \newcommand\lref[1]{}
      }
      \provideboolean{includeresponses}
      \setboolean{includeresponses}{false}
      \providecommand{\mailto}[1]{\href{mailto:#1}{\nolinkurl{#1}}}
      \provideboolean{showoldetails}
      \setboolean{showoldetails}{true}
      \providecommand{\oldetails}[2]{\ifthenelse{\boolean{showoldetails}}{#1}{#2}}
      \RequirePackage{hyphenat}
      \hyphenation{Ba-na-ch}
      \hyphenation{Cac-ciop-po-li}
      \hyphenation{ar-chi-m-e-dean}
      \hyphenation{op-ti-ma-li-ty}
      \hyphenation{qua-si-op-ti-ma-li-ty}

   \ifthenelse{\boolean{nohyperref}}{
     \PackageWarning{omarstyle}{package hyperref not loaded, please load manually if needed}
   }{
     \RequirePackage[obeyspaces]{url}
     \RequirePackage[bookmarks,hypertexnames=false,debug,pdfpagelabels]{hyperref}
     \RequirePackage{bookmark}
   }

   \newtheoremstyle{plain}%
     {}%
     {}%
     {\mdseries\slshape}%
     {\parindent}%
     {\bfseries}%
     {.}%
     {.5em}%
     {}%
   
   \newtheoremstyle{note}%
     {}%
     {}%
     {}%
     {\parindent}%
     {\bfseries}%
     {.}%
     {.5em}%
     {}%
   
   \newtheoremstyle{claim}%
     {}%
     {}%
     {\mdseries\slshape}%
     {}%
     {\bfseries}%
     {}%
     {.5em}%
     {}%
   
   \newtheoremstyle{exercise}%
     {}%
     {}%
     {}%
     {}%
     {\bfseries}%
     {.}%
     {1em}%
     {}%
   
   \newtheoremstyle{break}%
     {}%
     {}%
     {}%
     {}%
     {\bfseries}%
     {.}%
     {\newline}%
     {}%
   
   \swapnumbers{
     \theoremstyle{plain}
     \ifthenelse
         {\boolean{isthesis}}
         {\newtheorem{The}{\TheName}[section]}%
         {
         }%
   {
      \theoremstyle{plain}
   
      \renewcommand{\Thecounter}{subsection}

      \newtheorem*{The*}{\TheName}
      \newtheorem*{Lem*}{\LemName}
      \newtheorem*{Pro*}{\ProName}
      \newtheorem*{Cor*}{\CorName}
      \newtheorem*{Pbm*}{\PbmName}
      \newtheorem*{Hyp*}{\HypName}
      \newtheorem*{Exe*}{\ExeName}
      \newtheorem*{Txx*}{\ExeName} %
      \newtheorem*{Con*}{Conclusion}
      \newtheorem*{Sum*}{Summary}
    }
    {
      \theoremstyle{claim}

    }
    {
      \theoremstyle{note}

      \newtheorem*{Obs*}{\ObsName}

      \newtheorem*{Def*}{\DefName}
      \newtheorem*{Exa*}{\ExaName}
      \newtheorem*{Alg*}{\AlgName}
    }
   
    {
      \theoremstyle{break}
    }
   }

   \ifthenelse{\boolean{usecleveref}}{
     \usepackage{aliascnt}
     \usepackage{cleveref}
       \newtheorem{The}[subsection]{Theorem}
     \newaliascnt{Lem}{subsection}%
     \newtheorem{Lem}[Lem]{Lemma}
     \newaliascnt{Pro}{subsection}
     \newtheorem{Pro}[Pro]{Proposition}
     \aliascntresetthe{Lem}
   xo  \aliascntresetthe{Pro}
     \crefname{Lem}{lemma}{lemmata}
     \crefname{Pro}{proposition}{propositions}
   }{
     \newenvironment{The}[1][]{%
       \ifx&#1&%
       \subsection{\TheName\xspace}%
       \else%
       \subsection[\MakeUppercase#1 theorem]{\TheName\ (#1)}%
       \fi%
       \slshape}{%
       \upshape}
     \newenvironment{Pro}[1][]{\subsection{\ProName\xspace{\ifx&#1&{}\else{ (#1)}\fi}}\slshape}{\upshape}
     \newenvironment{Lem}[1][]{\subsection{\LemName\xspace{\ifx&#1&{}\else{ (#1)}\fi}}\slshape}{\upshape}
     \newenvironment{Cor}[1][]{\subsection{\CorName\xspace{\ifx&#1&{}\else{ (#1)}\fi}}\slshape}{\upshape}

     \newenvironment{Def}[1][]{\subsection{\DefName\xspace{\ifx&#1&{}\else{ of \indexen{#1}}\fi}}}{}
     \newenvironment{Obs}[1][]{\subsection{\ObsName\xspace{\ifx&#1&{}\else{ (#1)}\fi}}}{}
     
     \newenvironment{Alg}[1][]{\subsection{\AlgName\xspace{\ifx&#1&{}\else{ (#1)}\fi}}}{}
   }
   \providecommand{\qed}{\vrule height 5pt depth 0pt width 3pt}
   \providecommand{\qqed}{{\raggedright{\ \hfill\qed}}}
   
   \newcounter{passo}

   \newenvironment{Proof}[1][]%
   {\par\noindent{\bf \Proofname\ifx|#1|.\ \else\ #1.\ \fi}\setcounter{passo}{0}}%
   {\qqed\par}
   {\par\noindent{\bf \Derivename\ #1}\setcounter{passo}{0}}%
   {\qqed\par}
   \newenvironment{Proof*}[1][{}]%
   {\subsection{\Proofname\ #1}\setcounter{passo}{0}}
   {\qqed\par}

\makeindex
\synctex=1
\usepackage{natbib}
\usepackage{subfig}
\usepackage{hyperref}%
\usepackage{cleveref}
\usepackage{algpseudocode}
\makeatletter
\hyphenation{semi-norm non-li-near}
\renewcommand{\leq}{\leqslant}
\renewcommand{\geq}{\geqslant}
\renewcommand{\rot}{\nabla\!\times\!}%
\providecommand{\codevarname}[1]{\colorvarname{#1}}
\renewcommand{\codevarname}[1]{\colorvarname{#1}}
\providecommand{\tol}{\codevarname{tol}}
\renewcommand{\tol}{\codevarname{tol}}
\providecommand{\res}{\codevarname{res}}
\renewcommand{\res}{\codevarname{res}}
\providecommand{\maxiter}{\codevarname{maxiter}}
\renewcommand{\linop}[1]{\mathcal{\MakeUppercase{#1}}}

\ifthenelse{\boolean{useutopia}}{
  
}{
  
}
\providecommand{\feop}[2][h]{\mathcal{\MakeUppercase{#2}}\ifx|#1|\else_{#1}\fi}

\ifthenelse{\boolean{useutopia}}{
  \renewcommand{\fe}[2][]{\ensuremath{\mathsfit{#2}\ifx|#1|\else_{#1}\fi}}%
  \renewcommand{\vecfe}[2][]{\ensuremath{\vec{\mathsfit{#2}}\ifx|#1|{}\else{_{#1}}\fi}}%
  \renewcommand{\matfe}[2][]{\ensuremath{\mat{\mathsfit{\MakeUppercase{#2}}}\ifx|#1|{}\else{_{#1}}\fi}}%
}{
  \renewcommand{\fe}[2][]{\ensuremath{\mathsf{#2}\ifx|#1|\else_{#1}\fi}}%
  \renewcommand{\vecfe}[2][]{\ensuremath{\vec{\mathsf{#2}}\ifx|#1|{}\else{_{#1}}\fi}}%
  \renewcommand{\matfe}[2][]{\ensuremath{\mat{\mathsf{\MakeUppercase{#2}}}\ifx|#1|{}\else{_{#1}}\fi}}%
}
\providecommand{\tildefespace}[2][]{\tilde{\fespace[#1]{#2}}}
\providecommand{\DNewton}{\operatorname{\mathfrak D}\!}

\providecommand{\aposteriori}{a posteriori\xspace}

\providecommand{\apriori}{{a priori}\xspace}

\renewcommand{\aposteriori}{a posteriori\xspace}

\renewcommand{\apriori}{{a priori}\xspace}

\provideboolean{includeresponses}
\setboolean{includeresponses}{false}
\renewcommand{\colorvarname}[2][i]{\texttt{#2}}%
\renewcommand{\Argmax}[1][]{\operatorname{Argmax}\ifx|#1|\else\limits_{#1}\fi}
\renewcommand{\argmax}[1][]{\operatorname{argmax}\nolimits}
\providecommand{\normonsob}[4][\W]{\normon{#2}{\sob{#3}{#4}\if|#1|{}\else(#1)\fi}}
\providecommand{\Normonsob}[4][\W]{\Normon{#2}{\sob{#3}{#4}\if|#1|{}\else(#1)\fi}}
\providecommand{\normonsobh}[3][\W]{\normon{#2}{\sobh{#3}\if|#1|{}\else(#1)\fi}}
\providecommand{\Normonsobh}[3][\W]{\Normon{#2}{\sobh{#3}\if|#1|{}\else(#1)\fi}}
\providecommand{\Normonleb}[3][\W]{\Normon{#2}{\leb{#3}\if|#1|{}\else(#1)\fi}}

\renewcommand{\normonsob}[4][\W]{\normon{#2}{\sob{#3}{#4}\if|#1|{}\else(#1)\fi}}
\renewcommand{\Normonsob}[4][\W]{\Normon{#2}{\sob{#3}{#4}\if|#1|{}\else(#1)\fi}}
\renewcommand{\normonsobh}[3][\W]{\normon{#2}{\sobh{#3}\if|#1|{}\else(#1)\fi}}
\renewcommand{\Normonsobh}[3][\W]{\Normon{#2}{\sobh{#3}\if|#1|{}\else(#1)\fi}}
\renewcommand{\Normonleb}[3][\W]{\Normon{#2}{\leb{#3}\if|#1|{}\else(#1)\fi}}

\numberwithin{equation}{section}
\setboolean{showchanges}{false}%
\setboolean{shownotes}{false}%
\ifthenelse{\boolean{shownotes}}{
  \usepackage{backref}
  \hypersetup{backref=true}%
}{}
\providecommand{\authoromar}{Omar Lakkis}

\providecommand{\authoramireh}{Amireh Mousavi}

\providecommand{\ourshorttitle}{%
  Least-squares Galerkin to gradient recovery for HJB equation}
\providecommand{\ourtitle}{%
  A least-squares Galerkin approach to gradient recovery for
  Hamilton-Jacobi-Bellman equation with Cordes coefficients}
\providecommand{\ourabstract}{%
  We propose a conforming finite element method to approximate the strong solution of the second-order Hamilton-Jacobi-Bellman equation with Dirichlet boundary conditions and coefficients that satisfy the Cordes condition.
  We show the convergence of the continuum semismooth Newton method for the fully nonlinear Hamilton-Jacobi-Bellman equation.
  Using this linearization approach for the equation yields a recursive sequence of linear elliptic boundary value problems (BVPs) in nondivergence form.  
  We numerically solve these BVPs using the least-squares gradient recovery method proposed by \citet{LakkisMousavi:21:article:A-least-squares}. 
  We offer an optimal-rate \apriori and \aposteriori error bounds for the approximation.
  The \aposteriori error estimators are used to drive an adaptive refinement procedure. 
  We close with computer experiments on both uniform and adaptive meshes to reconcile the theoretical findings.
}
\providecommand{\ourkeywords}{%
  Hamilton--Jacobi--Bellman equations, Cordes coefficients, semismooth
  Newton linearization, least-squares approach, gradient recovery,
  optimal \apriori error bound, \aposteriori error bound, adaptive
  refinement}

\title[\ourshorttitle]{\ourtitle}
\author{\authoromar}
\author{\authoramireh}
\renewcommand{\d}{\operatorname d\!}
\makeatother
\begin{document}
\ifthenelse{\boolean{shownotes}}{
  \tableofcontents
  \listoffigures
  \setcounter{page}0
  \clearpage
}
\maketitle
\begin{abstract}
  \ourabstract
\end{abstract}
\subsection*{Keywords:}\ %
\ourkeywords
\subsection*{AMS subject classifications.} 65N15, 65N30, 65N50, 35D35, 47J25
\section{Introduction}
We develop a Galerkin least-squares numerical method to approximate a function $\funk u\W\reals$ that satisfies the following elliptic Dirichlet boundary value problem (BVP) associated to the \indexen{Hamilton--Jacobi--Bellman (HJB)} partial differential equation (PDE)
\begin{equation}
  \label{eq:HJB0-inhomogeneous}
  \sup_{\alpha\in\mathcal{A}}
  \qp{\linop L^\alpha u - f^\alpha}=0\text{ in } \W
  \tand
  \restriction u{\boundary\W}=r.
\end{equation}
Here and throughout $\W$ is a bounded convex domain in $\R d$,
$d\in\naturals$ (typically $d=2,3$), $\mathcal{A}$ is a compact metric
space called the \indexemph{(admissible) control set},
$r\in\sobh{3/2}(\boundary\W)$. For each admissible
\indexemph{control} $\alpha$ in $\mathcal{A}$, the corresponding forcing term
$f^\alpha$ is a member of $\leb2(\W)$ and the elliptic operator $\linop L^\alpha$ is
defined by a triple of functions $\funk{\triple{\mat a^\alpha}{\vec b^\alpha}{c^\alpha}}\W{\realsqmats d\times\R d\times\reals}$ and
\begin{equation}
  \label{op:HJB}
  \linop L^\alpha v
  :=
  \mat A^\alpha\frobinner\D^2 v
  +
  \vec b^\alpha\inner \nabla v
  -
  c^\alpha v
  \Foreach v\in\sobh2(\W)
  .
\end{equation}
Here, $\grad\phi$ and $\D^2\phi$ respectively indicate the gradient
and Hessian of a function $\phi$, and $\mat M\frobinner\mat
N:=\traceof{\transposemat M\mat N}$ defines the \indexen{Frobenius
  product} for two equally sized matrices $\mat M,\mat N\in\realmats
mn$, $m,n\in\naturals$. Additionally, if $m=n$, $\trace\mat{M}$
denotes $\mat{M}$'s trace, i.e., the sum of its (possibly repeated)
eigenvalues.

Excepting special situations, e.g., when the control set $\mathcal{A}$
is a singleton, equation (\ref{eq:HJB0-inhomogeneous}) is not linear
in $u$ and both its mathematical and computational analysis must be
approached as a fully nonlinear elliptic equation.

The HJB equation (\ref{op:HJB}) was introduced in the context of
\indexemph{dynamic programming} developed by
\citet{Bellman:57:book:Dynamic,Bellman&Kalaba} for optimal control
models. The second order equation considered here appears in optimal
control problems with stochastic processes; we refer to
\citet{Fleming&Sner} and the references therein as a source of
information on the modeling and analysis.

Fully nonlinear PDEs, including the HJB equations, of which a special
case is a reformulation of the Monge--Ampère equation
\cite{Lions:84:inproceedings:Hamilton-Jacobi-Bellman,
  krylov2001nonlinear}, play an essential role in many fields of
natural and social sciences as well as technology.  The practical
relevance motivates further the search for practical numerical
methods. The nonvariational
nature of fully nonlinear PDEs forms a challenge to their numerical
approximation via standard Galerkin methods, making finite differences
a natural first resort; but the flexibility
that Galerkin methods offer in terms of geometry approximation
and powerful adaptive mesh refinement techniques makes seeking
such methods worthwhile.

One of the main difficulties in seeking solutions to general fully
nonlinear PDEs, is the lack of classical solutions, which necessitates
often the investigation of solutions in a weak or generalized sense.
Since the structure of such PDEs precludes a natural variational
formulation the direct application of a weak solution in $\sobh1(\W)$
is not practical.  The natural approach to weak solutions, instead,
traced by \citet{Crandall&Lions} for first order equations and
extended to second order equations by \citet{Lions1,Lions2}, relies on
the concept of \indexemph{viscosity solution}, rooted in the theory of
vanishing viscosity methods in fluid dynamics intiated by
\citet{Hopf:50:article:The-partial}. Independent developments to weak
solutions for fully nonlinear eqautions and HJB can be traced also to
\citet{Aleksandrov:61:article:Investigation-on-the-Maximum} who
underscored the importance of the \indexemph{maximum principle} for
viscosity solutions and subsequent work connecting to stochastic
control by
\citet{Krylov:72:article:Control,Krylov:79:article:On-the-maximum-principle}.
This avenue has since led to a flurry in the theory of fully nonlinear
problems, often in connection with the theory of stochastic control
\citep{CaffarelliCabre:95:book:Fully,FlemingSoner:06:book:Controlled,Krylov:09:book:Controlled,Krylov:18:book:Sobolev}.

From a numerical perspective, a milestone was reached with the seminal
result of \citet{Barles&Souganidis}, demonstrating that consistency,
stability, and monotonicity of an approximating scheme guarantee the
convergence of the approximate solution to the exact viscosity
solution.  The importance of the Barles--Souganidis theorem
is vindicated by the extensive literature built upon it
\citep{Barles&Jakobsen,Oberman,Debrabant&Jakobsen,FengJensen:17:article:Convergent}.
The finite difference, including semi-Lagrangian or carefully designed
wide-stencil methods are well suited in approximating a viscosity
solution, as they preserve the maximum principle and consistency
\citep{Motzkin&Wasow,%
  KuoTrudinger:92:article:Discrete,%
  Bonnans&Zidani,%
  FroeseOberman:11:article:Convergent}.
Moreover, certain Galerkin methods, such as $\poly1$ on meshes that
satisfy the maximum mrinciple, have been shown to converge to the
viscosity solution, as demonstrated by \citep{Jensen&Smears,
  NochettoZhang:18:article:Discrete, Salgado&Zhang}.

\margnote[Omar]{continue from here}
While besides stability and consistency, monotonicity-based
discretization methods are guaranteed to converge, providing this
property in discretizations is not always immediately obvious and may
preserve only in specific cases.

To ease the strict requirement of discretization monotonicity,
\citet{Sme&Sul-2} explored the HJB equation under the Cordes condition
on the coefficients of the ellipic operators $\linop L^\alpha$
\citep{Cordes:59:article:Vereinfachter}. They introduced and analyzed
a finite element approximation that converges without the need to
maintain monotonicity.
In this approach, the underlying PDE is reformulated into a
\indexemph{second-order binonlinear form}, i.e., a form in two
variables which is nonlinear in at least one of its arguments, which
they then discretize employing a discontinuous Galerkin finite element
scheme.
 
The provided convergence analysis relies on the strong monotonicity of
the binonlinear form, considered as a generalization of coercivity for
nonlinear operators, but as a functional rather than in
discretization.

In \citet{Gal&Sul}, a similar convergence argument is employed, but
with a distinction: the HJB equation is treated as a variational
binonlinear form problem, utilizing a conforming finite element
method. The fundamental analytical tool in this work is also the
strong monotonicity of the binonlinear form.

There are two strategies to deal with nonlinear PDEs such as the HJB
equation. In the first strategy, the nonlinear problem is discretized
and the resulting nonlinear finite-dimensional system is linearized
with a nonlinear solver such as Newton's method. Many computational
methods for approximating the solution of the HJB equation follow this
approach \citep{Oberman, Sme&Sul-2, Gal&Sul}.
In this scenario, if either the discretization is monotone or, in the
functional setting, the binonlinear form is strongly monotone, the
error analysis of the method becomes possible. Furthermore, strong
monotonicity establishes the framework for applying the Browder-Minty
theorem to demonstrate the well-posedness of solutions to nonlinear
equations.
It is however not always possible to achieve a monotone discretization
or a strongly monotone binonlinear form for a fully nonlinear
problem. To the best of our knowledge, all methods that provide an
approximate solution and offer a satisfactory convergence analysis for
the HJB equation follow this strategy by somehow enforcing
monotonicity.

In a second strategy, which we pursue here, the nonlinear PDE is first
linearized, for example, by using Newton's method in the appropriate
infinite-dimensional space, into an iterative sequence of linear PDEs
in nondivergence form at the continuum level. Subsequently, these
linear PDEs are discretized.  By following this strategy, the
convergence rate of the nonlinear solver is independent of the
discrete space parameters, such as the meshsize or polynomial degree
in Galerkin methods. Furthermore any numerical method applicable to
linear problems in nondivergence form can be extended to fully
nonlinear problems, such as HJB equation we study herein. An instance
of the second strategy is
\citet{LakkisPryer:11:article:A-finite,Lakkis&Pryer}, who derived a
computational method for linear elliptic problems in nondivergence
form used to solve the linearized iteration for a class of fully
nonlinear problems in which the nonlinearity is \indexemph[algebraic
  nonlinearity]{algebraic}, i.e., a nonlinearity that can be written
without resorting to $\sup_\alpha$ operations with infinitely many
$\alpha$s (e.g., the Monge--Ampère).

In this paper, we adopt the strategy of
first-linearize-then-discretize.  In the linearization step we
establish the Newton differentiability of the HJB operator from
$\sobh2(\W)$ to $\leb2(\W)$, subject to a uniform convergence of controls
Assumption~\ref{assum-uniformly-convergent-coefficients}.

While \cite{Sme&Sul-2} have demonstrated this concept for the operator
from $\sob 2s(\W;\mesh T)$ to $\leb{l}(\W)$, where $1\leq l < s \leq
\infty$, these spaces are mesh-dependent and encompass only finite
element spaces (not $\sobh2(\W)\rightarrow \leb2(\W)$).

The Newton differentiability of the HJB operator allows us to extend
the classical Newton linearization to the HJB operator, even when it
is not necessarily Fréchet differentiable. This extended method is
referred to as the \emph{semismooth Newton} linearization
\citep{Hintermuller:10:booklet:Semismooth, Ito&Kunisch}. With this
approach, the solution of the nonlinear HJB equation is realized as
the limit of a recursive sequence of solutions to linear problems in
nondivergence form. We complete the linearization theory by
demonstrating the superlinear convergence rate of this recursive
method.

We also introduce the algorithmic form of the resulting recursive
method, known as \emph{Howard's algorithm} or \emph{policy iteration},
which, to the best of our knowledge, is presented here for the first
time in the infinite dimensional setting for HJB PDEs.

At each iteration, this algorithm updates a space-dependent function,
denoted as $q$ and taking values in the control set $\mathcal A$. This
function $q$ determines the linear operator $\linop l^{q_{n}}$ as seen
in (\ref{eq:recursive-practical-a-e}). Consequently, Howard's
algorithm yields a sequence of control-seeking parts and nondivergence
form PDE solvers, essentially constituting a parametrized
linearization procedure.

In the context of connecting fully nonlinear and linear problems, the
bridge is formed by nondivergence form linear operators. To discretize
PDEs in this scenario, we employ a least-squares Galerkin gradient
recovery method. This approach offers a straightforward means of
dealing with linear equations in nondivergence form. It can be
interpreted as a mixed finite element technique and provides a
convenient framework for deriving \aposteriori error estimates and
corresponding adaptive methods, as demonstrated in
\citet{LakkisMousavi:21:article:A-least-squares}.

The least-squares approach enables us to replace any constraint required
to ensure the problem is well-posed with an additional term in the quadratic
(least-squares) form. This flexibility to work in a general space is a
significant advantage, as constructing finite element approximations
that exactly satisfy conditions like rotational-free or vanishing
tangential trace can be challenging.

The rest of our article is arranged as follows: in \S\ref{sec:set-up},
we clarify the problem assumptions and the existence theory concerning
to the well-posedness of the strong solution for the HJB equation.

In \S\ref{sec:smismooth-Newton-method}, we introduce the concept of
Newton differentiability for operators and illustrate its application
to the HJB operator in the continuous setting, particularly from
$\sobh2(\W)$ to $\leb2(\W)$.  This is done under the assumption of
uniform convergence on the policy operators, as outlined in
Assumption~\ref{assum-uniformly-convergent-coefficients}. We also
discuss the use of the semismooth Newton method for linearizing the
HJB equation. Subsequently, we present the algorithmic expression of
the linearization procedure in the form of Howard's algorithm, which
involves solving a second-order elliptic equation in nondivergence
form.  In \S\ref{sec:Variational-linear}, we provide a review of the
least-squares Galerkin approach, which includes gradient recovery as
proposed in \cite{LakkisMousavi:21:article:A-least-squares}. This
method is employed to solve linear elliptic PDEs in nondivergence
form, and we also revisit the key error bounds associated with it.  In
\S\ref{sec:Discretization}, we return to the nonlinear HJB equation
and discuss the error analysis of the approximation, which includes
both the associated \apriori and \aposteriori error estimates. Owing
to the \aposteriori error bound, we design error indicators to be
employed in an adaptive refinement strategy. Additionally, we outline
the algorithms used during the implementation phase to approximate the
solution of the HJB equation.  In \S\ref{sec:numerical-experiment} we
present two numerical tests that validate the theoretical results.

\section{Set-up and notation}
\label{sec:set-up} 
We now briefly review ellipticty and Cordes conditions in
\S\ref{sec:assumptions-on-data}, refornulate the HJB-Dirichlet problem
\eqref{eq:HJB0-inhomogeneous} into a homogeneous Dirichlet problem in
\S\ref{sec:homogeneous-Dirichlet}, and recall that it, and thus the
original heterogeneous version (\ref{eq:HJB0-inhomogeneous}), admits a
unique strong solution under $\leb\infty(\W;\cont0(\mathcal A))$ data
assumptions in
\S\ref{The:well-posedness-strong-solution}--\S\ref{rem:relaxing-continuity-uniformly-bounded}.
\subsection{Assumptions on the data}
\label{sec:assumptions-on-data}

Throughout this paper, concerning the BVP (\ref{eq:HJB0-inhomogeneous}), we suppose that the coefficients $\mat A^\alpha$ satisfy the \indexemph{uniform ellipticity condition}
\begin{equation}
  \label{def:uniformly-elliptic}
  \constref[\hA,\flat]{def:uniformly-elliptic}
  \eye
  \leq
  \mat A^\alpha
  \leq
  \constref[\hA,\sharp]{def:uniformly-elliptic}
  \eye,
  \text{ a.e. in }\W,\Foreach \alpha \in \mathcal{A}
\end{equation}
for some positive constants $\constref[\hA,\flat]{def:uniformly-elliptic}$ and $\constref[\hA,\sharp]{def:uniformly-elliptic}$ independent of $\alpha\in\mathcal A$, while the tensor-valued $\mat A^\alpha$, vector-valued $\vec b^\alpha$, nonnegative scalar-valued $c^\alpha$ satisfy satisfy one of the Cordes conditions as outlined by \cite{Sme&Sul-2}. These conditions are as follows
\begin{enumerate}[(a)\ ]
\item
  there exists 
  $\lambda >0$ and $\varepsilon\in (0,1)$ such that for $\alpha \in \mathcal{A}$, 
  \begin{equation}
    \label{def:general-Cordes-condition}
    \dfrac{\norm{\mat A^\alpha}^2
      + \fraclff{\norm{\vec b^\alpha}^2}{2\lambda}
      + (\fraclff{c^\alpha}\lambda)^2}{
      (\trace\mat A^\alpha
      + \fracl {c^\alpha}\lambda ) ^2
    }
    \leq
    \dfrac{1}{d+\varepsilon}
    \text{ a.e. in } \W,
  \end{equation}
  where for a tensor/matrix $\mat M\in\realsqmats n$, $\norm{\mat M}$ signifies its
  Frobenius norm, $(\mat M\frobinner\mat M)^{\fracl{1}{2}}$.
\item
  or in the case of a \emph{homogeneously second-order} $\linop l^\alpha$, i.e., $\vec b^\alpha =0$ and
  $c^\alpha=0$, take $\lambda=0$ and replace
  \begin{equation}
    \label{def:special-Cordes-condition}
    \dfrac{\norm{ \mat A^\alpha }^2}{(\trace \mat A^\alpha)^2}
    \leq
    \dfrac{1}{d-1+\varepsilon}
    \text{ a.e. in }\W
  \end{equation}
  for some $\varepsilon\in (0,1)$.
\end{enumerate}
For more details on the Cordes condition see \citet[\S2.2]{LakkisMousavi:21:article:A-least-squares}.
We take the admissible \indexemph[admissible control set]{control} \akaindexemph[admissible policies set]{policy} \emph{set} $\mathcal A$ to be a compact metric space with distance $d_\mathcal A$.
The ensuing topology is used in defining spaces $\cont0(\mathcal A;X)$ with $X=\reals,\R d,$ or $\realsqmats d$. 
Most often $\mathcal A$ is in fact subset (e.g., a Lie group or a subspace) of the matrix algebra $\realsqmats{d}$; we give examples in \S\ref{test:non-homogeneous-boundary} and \S\ref{test:adaptive-disk-domain}.
We denote by $\ball[\mathcal  A]{\alpha}{\rho}$ the open ball of center $\alpha$ and radius $\rho\geq0$ with $d_{\mathcal A}$ scale.
\subsection{Homogeneous Dirichlet problem reformulation of (\ref{eq:HJB0-inhomogeneous})}
\label{sec:homogeneous-Dirichlet}
Since $r\in\sobh{3/2}(\boundary\W)$, it admits an extension to $\tilde{r}\in\sobh2(\W)$ with boundary trace $r$ ($\restriction {\tilde{r}} {\boundary\W}= r$) satisfying
\begin{equation}
  \label{eqn:trace-inequality}
  \Norm{\tilde{r}}_{\sobh2(\W)} 
  \leq
  \constref{eqn:trace-inequality}
  \Norm{r}_{\sobh{3/2}(\boundary\W)}
\end{equation}
for some $\constref{eqn:trace-inequality}>0$ depending only on $\W$. 
By setting $v= u-\tilde{r}$, we rewrite (\ref{eq:HJB0-inhomogeneous}) as the homogeneous problem of finding $v$ such that
\begin{equation}
  \label{eq:HJB0-homogeneous}
  \sup_{\alpha \in \mathcal{A}}
  \left(  \linop L^\alpha v - f^\alpha + \linop L^\alpha \tilde{r}\right) = 0  ~ \text{ in } \W
  \tand
  \restriction v{\boundary\W}=0.
\end{equation}
Therefore, we consider the homogeneous problem of finding $u$ satisfying
\begin{equation}
  \label{eq:HJB-homogeneous}
  \sup_{\alpha \in \mathcal{A}}
  \left(  \linop L^\alpha u - f^\alpha \right) = 0 ~ \text{ in } \W
  \tand
  \restriction u{\boundary\W}=0.
\end{equation}
We shorten notation by introducing the \indexemph{HJB operator}
${\nlop F}:\sobh2(\W) \meet \sobhz1(\W)\rightarrow \leb2(\W)$ by 
\begin{equation}
  \label{def:op:HJB}
  \nlop F[v](\vec x)
  =
  \sup_{\alpha \in \mathcal{A} }
  \left(  
  \linop L^\alpha v(\vec x) - f^\alpha(\vec x)
  \right) 
  ~
  \text{ for }
  \vec x\in\W.
\end{equation}
%
%
%
%
%
%
%
\begin{The}[existence and uniqueness of a strong solution \citep{Sme&Sul-2}]
  \label{The:well-posedness-strong-solution} 
  Suppose that $\W$ is a bounded convex domain in $\R d$,
  $\mathcal{A}$ is a compact metric space under $d_{\mathcal A}$,
  and that ${\mat A,\vec b,c,f}\in\cont0({\W\times\mathcal{A}};X)$,
  for $X=\Symmatrices d,\R d,\reals,\reals$, satisfy
  (\ref{def:uniformly-elliptic}) and 
  (\ref{def:general-Cordes-condition}) with $\lambda>0$, or
  (\ref{def:special-Cordes-condition}) with $\lambda=0$ when $\vec b
  = 0$ and $c = 0$ hold.
  Then there exists a unique function $u\in\sobh2(\W)\meet\sobhz1(\W)$
  that satisfies the HJB equation (\ref{eq:HJB-homogeneous}) almost
  everywhere in $\W$.
\end{The}

\begin{Obs}[strong solution of the nonhomogeneous equation]
  \label{rem:well-posedness- non-homogeneous}
  Since (\ref{eq:HJB0-homogeneous}) and (\ref{eq:HJB0-inhomogeneous}) are equivalent, 
  under the assumptions of Theorem~\ref{The:well-posedness-strong-solution} and for 
  $r \in \sobh{3/2}(\boundary \W)$, there exists a unique strong solution $u \in \sobh2(\W)$ to 
  nonhomogeneous HJB equation (\ref{eq:HJB0-inhomogeneous}).
\end{Obs}
\begin{Obs}[$\leb\infty(\W)$ vs $\cont0(\W)$ data]
  \label{rem:relaxing-continuity-uniformly-bounded}
  The sufficient requirement on data to be $\cont0(\W)$ in Theorem~\ref{The:well-posedness-strong-solution} can be relaxed to be just $\leb{\infty}(\W)$. This allows us to consider $\mat A$, $\vec b ,c$ in  $\leb\infty(\W;\cont0(\mathcal{A};X))$ for $X= \Symmatrices d$, $\R d$, $\reals$ respectively, while for $X = \reals$ we omit it in the notation. We also assume $f\in\leb2(\W;\cont0(\mathcal{A}))$.
\end{Obs}
\section{Semismooth Newton method}
\label{sec:smismooth-Newton-method}
Due to the nonalgebraic nonlinearity of the HJB equation, applying a
linearization methods like classical Newton is not trivial
\changes{
  even when the operator $\nlop F$ defined in \eqref{def:op:HJB} has
  everywhere an invertible derivative: the problem is that such a
  derivative cannot be explicity found, while it can be realised by
  finding the appropriate $\alpha=q(\vec x)$ for an apporpriate
  function $\funk q\W\mathcal A$.
}%
In this section, we describe the semismooth Newton method for
linearizing the fully nonlinear problem~(\ref{eq:HJB0-inhomogeneous}),
which involves nonsmooth nonlinear operators.
The basic tool here is the concept \indexemph{Newton derivative} of
$\nlop F$, a set-valued operator $\DNewton\nlop F$, introduced by
\citet{Ito&Kunisch}, which we define and discuss the associated
elements in
\S\ref{def:N-set}--\S\ref{def:Newton-differentiable-operaor}.
While \citet{Sme&Sul-2} demonstrated the Newton differentiability of
$\nlop F$ from $\sob2s(\W; \mesh T)$ to $\leb{l}(\W)$ with $1\leq l <
s \leq \infty$, Theorem~\ref{the:semi-smooth:HJB-op} extends this
result for $\nlop F$ from $\sobh2(\W)$ to $\leb2(\W)$ under an
assumption on the control set $\mathcal{A}$ described in
\S\ref{sec:assum-uniformly-convergent-coefficients}.
We then outline in
\S\ref{lem:boundedness-inverse-DF}%
some properties of the $\DNewton {\nlop F}[v]$'s members that will ensure
the superlinear convergence of the semismooth Newton method in 
\S\ref{the:superlinear:convergence}--\S\ref{cor:superlinear:convergence-HJB}.
We close in \S\ref{sec:Howard-algorithm} with a algorithmic presentation of the semismooth Newton linearization known as Howard's algorithm or policy iteration.
\subsection{Policy and set-valued maps}
Introduce the \indexemph{policy map set} \akaindexemph{control map set}
\begin{equation}
  \mathcal{Q}:= \left\lbrace \funk q\W{\mathcal{A}} \left\vert ~ q \text{ is measurable}\right\rbrace \right.,
\end{equation}
and the set-valued \indexemph{state-to-policy} operator
$\mathcal{N} :\sobh2(\W)\meet\sobhz1(\W) \rightrightarrows \mathcal{Q}$
by
\begin{equation}
  \label{def:N-set}
  \mathcal{N}[v]
:=
\setofsuch{
  q\in\mathcal{Q}
}{
  q(\vec x) \in \underset{\substack{\alpha \in \mathcal{A}}}
  \Argmax\qp{\linop L^\alpha v(\vec x) - f^\alpha(\vec x)}
  \text{ for a.e. }\vec x\in\W
}

  ,
\end{equation}
where by the notation $\mathcal M:X \rightrightarrows Y$, we mean that $\mathcal M$ is a powerset-valued, $\powersetof Y$-valued, map, meaning that for every $x \in X$, $M(x)$ is a subset of $Y$.
\begin{Lem}[state-to-policy operator is well defined and continuous]
  \label{lem:lim-alpha-j}
  For any $v \in \sobh2(\W) \meet \sobhz1(\W)$, the maps set $\mathcal{N}[v]$ is nonempty. 
  Moreover, if $\seqsinat vj$
  be a sequence such that $ v_j \rightarrow v$ in $\sobh2(\W)$ and 
  $\seqsinat qj $ 
  be a sequence in which $q_j \in \mathcal{N}[ v_j]$, then
  \begin{equation}
    \label{lim:alpha-j-a-e}
    \lim_{j \rightarrow \infty}  
    \underset
        {\substack{q \in  \mathcal{N}[v]
        }}
        \inf d_{\mathcal{A}}(q_j, q) =0 
        \text{ a.e. in }\W,
  \end{equation}
  with reminder that $d_{\mathcal{A}}$ is a metric on $\mathcal{A}$.
\end{Lem}
\begin{Proof}
  To prove that $\mathcal{N}[v]$ is nonempty, we refer to Theorem 10 of \cite{Sme&Sul-2}. 
  For almost every $\vec x \in\W$, we prove (\ref{lim:alpha-j-a-e}) by contradiction. 
  Suppose that there exist a sequence $\seqsinat vj$ convergent to $v$
  in $\sobh2(\W)\meet\sobhz1(\W)$, a subsequence 
  $\seqsinat qj$, $q_j\in\mathcal{N}[v_j]$ 
  and a real $\bar{\varepsilon}>0$ such that
  \begin{equation}
    d_{\mathcal{A}}(q_j(\vec x), q(\vec x))>\bar{\varepsilon} \quad 
    \text{for all } j\in \naturals \text{ and for all } q \in \mathcal{N}[ v].
  \end{equation}  
  It means that 
  $q_j(\vec x) \in \mathcal{A} \smallsetminus \underset{\substack{q \in \mathcal{N}[v]}} \bigcup \ball[\mathcal A]{q(\vec x)}{\bar\varepsilon}$
  where 
  \begin{equation}
    \ball[\mathcal A]{q(\vec x)}{\bar\varepsilon} := \left\lbrace \alpha \in \mathcal{A} 
    \left\vert ~ d_{\mathcal{A}}( \alpha, q(\vec x)) <\bar{\varepsilon}\right\rbrace \right. .
  \end{equation}
  Since
  $\mathcal{A} \smallsetminus \underset{\substack{q \in \mathcal{N}[v]}} \bigcup \ball[\mathcal A]{q(\vec x)}{\bar\varepsilon}$ 
  is compact, there exists a subsequence, which we pass without changing notation, 
  such that 
  $q_j(\vec x) \rightarrow \tilde{\alpha} \in \mathcal{A}
  \smallsetminus \underset{\substack{q \in \mathcal{N}[v]}} \bigcup
  \ball[\mathcal A]{q(\vec x)}{\bar\varepsilon}$.
  On the other hand, $q_j \in \mathcal{N}[v_j]$ implies that
  \begin{equation}
    \left( \linop L^{q_j} v_j - f^{q_j} \right)(\vec x) 
    \geq 
    \qp{\linop L^{q}  v_j - f^q}(\vec x)
    \Forall j
    \in
    \naturals.
  \end{equation}
  Since $ v_j \rightarrow v$ in $\sobh2(\W)$, again we pass to a 
  subsequence without changing notation such that 
  $v_j$, $\nabla v_j$, $\D^2v_j$ tend to $v$, $\nabla v$, $\D^2v$ respectively 
  pointwise almost everywhere in $\W$. Taking 
  the limit of the above inequality as $j \rightarrow \infty$ shows that 
  $ \tilde{\alpha} \in \mathcal{N}[v](\vec x)$. This is a contradiction with 
  $\tilde{\alpha} \in \mathcal{A} \smallsetminus \underset{\substack{q \in  \mathcal{N}[v]}} \bigcup \ball[\mathcal A]{q(\vec x)}{\bar\varepsilon}$. 
\end{Proof}
\begin{Obs}[approximating policy sequence selection]
  \label{rem:existence-sequence}
  Arguing by contradiction we can also show that if the sequence of functions
  $\seqsinat vj \subset \sobh2(\W) \meet \sobhz1(\W)$ converges 
  to $v$ in $\sobh2(\W)$, then for any $q\in\mathcal{N}[v]$ there exists a sequence
  of policies $\seqsinat qj$
  such that $q_j\in\mathcal{N}[v_j]$ and $q_j$ tends 
  to $q$ almost everywhere in $\W$.
\end{Obs}
\begin{Pro}[policy-to-coefficient lower semicontinuity]
  \label{cor:lim-coefficient}
  Recall that $\mat A$, $\vec b$, $c$ are in
  $\leb\infty(\W;\cont0(\mathcal{A};X))$ for $X=\Symmatrices{d}$,
  $\R{d}$, $\reals$ respectively.  Under the assumptions of
  Lemma~\ref{lem:lim-alpha-j}, we have
  \begin{equation}
    \label{lim:coefficient}
    \lim_{j \rightarrow \infty}  
    \inf_{q \in \mathcal{N}[v]}
    \left( 
    \norm{\mat A^{q_j} - \mat A^{q}} 
    +
    \norm{\vec b^{q_j}-\vec b^{q}}
    +
    \norm{c^{q_j} - c^q}
    \right) 
    =0
    \text{ a.e. in }\W.
  \end{equation}
\end{Pro}
\begin{proof}
  We want to prove that for almost all $\vec x\in\W$
  \begin{equation}
    \label{def:limit}
    \Forall \varrho>0 ~~ \Exists N \in \naturals :
    \forall j\geq N
    \inf_{q \in \mathcal{N}[v]}
    \evalat{
      \norm{\mat A^{q_j} - \mat A^{q}} 
      +
      \norm{\vec b^{q_j}-\vec b^{q}}
      +
      \norm{c^{q_j} - c^q}
    }{\vec x}<\varrho.
  \end{equation}
  Since $\mat A, \vec b, c$ are continuous on $\mathcal{A}$, so for 
  almost all $\vec x \in \W$ we deduce that for any $\varrho>0$ there 
  is some $\delta>0$ such that 
  \begin{multline}
    \label{eq:proof:policy-to-coeff-lower-semicontinuity}
    \inf_{q \in \mathcal{N}[v]}d_{\mathcal{A}}(q_j(\vec x), q(\vec x))<\delta 
    \implies
    \inf_{q \in \mathcal{N}[v]} 
    \evalat{
      \norm{\mat A^{q_j} - \mat A^{q}} 
      +
      \norm{\vec b^{q_j}-\vec b^{q}}
      +
      \norm{c^{q_j} - c^q}
    }{\vec x}  
    <\varrho.
  \end{multline}
  Relation (\ref{lim:alpha-j-a-e}) guarantees that for large enough $N\in\naturals$
  \begin{equation}
    \label{ineq:policy-limit}
    \forall j \geq N\implies \inf_{q \in \mathcal{N}[v]}d_{\mathcal{A}}(q_j(\vec x), q(\vec x))<\delta .
  \end{equation}
  The assertion now follows from (\ref{eq:proof:policy-to-coeff-lower-semicontinuity}) 
  and
  (\ref{ineq:policy-limit}).  
\end{proof}
\subsection{Assumption on convergence of policy sequence}
\label{sec:assum-uniformly-convergent-coefficients}
\margnote[OL]{To check this section.}
While we do not use the results of Lemma~\ref{lem:lim-alpha-j} and
Proposition~\ref{cor:lim-coefficient} explicitly in this paper, we
have presented them to establish the foundation for the following
assumption.  Although in Lemma~\ref{lem:lim-alpha-j}, we have
demonstrated that the sequence in (\ref{lim:alpha-j-a-e}) converges
pointwisely, in the following assumption, we assume that it converges
uniformly. This assumption implies uniform convergence of the sequence
in~(\ref{lim:coefficient}) as well and this uniform convergence
ensures the Newton differentiability of the HJB operator, thereby
motivating the the following assumption.
\subsection{Assumption (Uniform convergence of policy sequences)}
\margnote[OL]{To check this section.}
\label{assum-uniformly-convergent-coefficients}
For $\seqsinat vj$ as a sequence such that $ v_j \rightarrow v$ in $\sobh2(\W)$ and 
$\seqsinat qj $ as a sequence in which $q_j \in \mathcal{N}[ v_j]$, we suppose that 
\begin{equation}
  \label{lim:alpha-j-uniformly}
  \lim_{j \rightarrow \infty} 
  \Norm{ 
    \underset
	{\substack{q \in  \mathcal{N}[v]
	}}
	\inf d_{\mathcal{A}}(q_j, q) 
  }_{\leb{\infty}(\W)}
  =0 .
\end{equation}
Under this assumption, by employing a reasoning process analogous to the proof presented in Proposition~(\ref{cor:lim-coefficient}), we can deduce 
the following stronger convergence 
\begin{equation}
  \label{lim-uniformly-convergent-coefficients}
  \lim_{j \rightarrow \infty }
  \Norm{
    \inf_{q \in \mathcal{N}[v]} 
    \left( \norm{\mat A^{q_j} - \mat A^q}
    +
    \norm{\vec b^{q_j} - \vec b^q}
    +
    \norm{c^{q_j} - c^q}
    \right)
  }_{\leb{\infty}(\W)} = 0.
\end{equation}
In accordance with \citet[Remark 1]{Sme&Sul-2}, the Newton differentiability of the HJB operator, denoted as $F:\sobh{s}(\W)\rightarrow \leb{l}(\W)$ and defined in (\ref{def:op:HJB}), generally cannot be guaranteed unless $s>l$. To overcome this limitation for case $s=l=2$, we make  Assumption~\ref{assum-uniformly-convergent-coefficients}, allowing us to ensure Newton differentiability. 
We would like to emphasize that Example 8.14 in \cite{Ito&Kunisch}, which is as an illustration of non-Newton differentiability in the case of $s\leq l$, does not satisfy Assumption~\ref{assum-uniformly-convergent-coefficients}. 
This is primarily a result of nonzero differences between the values of the sequence policy operators and their limit, which prevent uniform convergence.
\begin{Def}[Newton differentiable operator]
  \label{def:Newton-differentiable-operaor}
  Following \cite{Ito&Kunisch}, for $\linspace X$ and $\linspace Z$
  Banach spaces and $\emptyset\subsetneq B\subset\linspace X$, a map
  $F:B\subset \linspace X \rightarrow \linspace Z$ is called
  \indexemph{Newton differentiable} at $x\in B$ if and only if there
  exists an open neighborhood $N(x)\subset B$ and set-valued map with
  nonempty image $\DNewton F[x]: N(x) \rightrightarrows\linopss XZ$
  such that
  \begin{equation}
    \label{def:semi-smooth}
    \lim_{\Norm{e}_{\linspace X} \rightarrow 0 }
    \dfrac{1}{\Norm{e}_{\linspace X}}
    \sup_{D \in \DNewton  F[x+e]}
    \Norm{
      F[x+e] - F[x] -De
    }_{\linspace Z} 
    = 0,
  \end{equation}
  while $\linopss XZ$ denotes the set of all bounded linear
  operators from $\linspace X$ to $\linspace Z$.  The operator set
  $\DNewton{F}[x]$ is called the \indexemph{Newton derivative} of $F$
  at $x$ and any member of $\DNewton F[x]$ is \indexemph{a Newton
    derivative} of $F$ at $x$.  The operator $F$ is called
  \indexemph{Newton differentiable on} $B$ with Newton derivative
  $\DNewton{F}:B\rightrightarrows\linopss{X}Z$ if $F$ is Newton
  differentiable at each $x\in{B}$.

  The Newton derivative may not be unique, but when $\D F[x]$, the Fréchet derivative of $F$ at $x$, exists we have $\DNewton{F}[x]=\setof{\D F[x]}$.
\end{Def}
\subsection{The Newton derivative of the HJB operator}
We now introduce a suit-able candidate for the Newton derivative of the operator ${\nlop F}$ in (\ref{def:op:HJB}) at $v$ by
\begin{equation}
  \label{eq:semi-smooth-HJB}
  \begin{gathered}
    \DNewton{\nlop{F}}:\sobh2(\W)\meet\sobhz1(\W)%
    \rightrightarrows%
    \linops{\sobh2(\W)\meet\sobhz1(\W)}{\leb{2}(\W)}	
    \\
    \DNewton {\nlop F}[v]:=
\setofsuch{
  \linop L^q :=
  (\mat A^q \frobinner\D^2  + \vec b^q \inner \nabla  - c ^q ) 
}{
  q \in \mathcal{N}[v]
}
\forall v \in \sobh2(\W) \meet \sobhz1(\W) 
,
  \end{gathered}
\end{equation}
where $\mathcal{N}[v]$ is defined by (\ref{def:N-set}). 
The term $\DNewton{\nlop{F}}[v]$ in (\ref{eq:semi-smooth-HJB}) 
is a Newton derivative candidate for (\ref{def:op:HJB}). We investigate 
this issue in Theorem \ref{the:semi-smooth:HJB-op}, where we extend
\citet[Thm. 13]{Sme&Sul-2} under the Assumption \ref{assum-uniformly-convergent-coefficients} to cover the case
$\sobh2(\W)\meet\sobhz1(\W)$.
\begin{The}[Newton differentiability of the HJB operator]
  \label{the:semi-smooth:HJB-op}
  Suppose that $\W$ is a bounded convex domain, $\mathcal{A}$ is a
  compact metric space, $\mat A, \vec b, c$ are in 
  $\leb\infty(\W;\cont0(\mathcal{A};X))$ for $X=\Symmatrices{d}$,
  $\R{d}$, $\reals$ respectively and $f\in\leb2(\W;\cont0(\mathcal{A}))$, which satisfy Assumption~\ref{assum-uniformly-convergent-coefficients}.
  Then, the HJB operator ${\nlop F}$ defined by (\ref{def:op:HJB}) is
  Newton differentiable with Newton derivative $\DNewton{\nlop{F}}$,
  defined by (\ref{eq:semi-smooth-HJB}), on $\sobh2(\W)\meet\sobhz1(\W)$.
\end{The}
\begin{Proof}
  Suppose that $\seqsinat ej \subset \sobh2(\W) \meet \sobhz1(\W)$ be a sequence with $\Norm{e_j}_ {\sobh2(\W)} \rightarrow 0$ and $v \in \sobh2(\W) \meet \sobhz1(\W)$. 
  Let $q_j \in \mathcal{N}[v+e_j]$ and $q \in \mathcal{N}[v]$.
  We have 
  \begin{equation}
    \label{ineq:monotone:smaller:zero}
          {\nlop F}[v + e_j] - {\nlop F}[v] - \linop L^{q_j} e_j = \linop L^{q_j} v - f^{q_j} - {\nlop F}[v] \leq 0 
          \text{ a.e. in } \W.
  \end{equation}
  On the other hand, we get
  \begin{equation}
    \label{ineq:monotone:bigger:zero}
    \begin{aligned}
      0 &\leq {\nlop F}[v+e_j] - ( \linop L^{q} (v+ e_j) - f^{q}) = {\nlop F}[v+e_j] - {\nlop F}[v] -  \linop L^q e_j
      \\
      & = {\nlop F}[v+e_j] - {\nlop F}[v] -  \linop L^{q_j} e_j + \linop L^{q_j} e_j - \linop L^q e_j
      \text{ a.e. in } \W.
    \end{aligned}
  \end{equation}
  (\ref{ineq:monotone:bigger:zero}) implies the first and (\ref{ineq:monotone:smaller:zero}) 
  implies the second inequality of the following
  \begin{equation}
    \label{ineq:smaller:zero}
    \linop L^{q} e_j - \linop L^{q_j} e_j \leq {\nlop F}[v+e_j] - {\nlop F}[v] - \linop L^{q_j} e_j
    \leq 0  \text{ a.e. in } \W.
  \end{equation}
  By applying the absolute value to both sides, we get
  \begin{multline}
    \label{ineq:norm}
    \norm{{\nlop F}[v+e_j] - {\nlop F}[v] - \linop L^{q_j} e_j}
    \leq
    \norm{\linop L^{q} e_j - \linop L^{q_j} e_j}
    \\
    =
    \norm{(\mat A^{q_j} - \mat A^q) \frobinner\D^2 e_j 
      +
      (\vec b^{q_j} - \vec b^\alpha)\inner \nabla e_j
      -
      (c^{q_j} - c^q)e_j
    } 
    \text{ a.e. in } \W.
  \end{multline}
  There exists $\constref{ineq:norm-1} > 0$ depending on the dimension $d$, 
  such that 
  \begin{multline}
    \label{ineq:norm-1}
    \norm{{\nlop F}[v+e_j] - {\nlop F}[v] - \linop L^{q_j} e_j}
    \\
    \leq
    \constref{ineq:norm-1}
    \left( \norm{\mat A^{q_j} - \mat A^q}
    +
    \norm{\vec b^{q_j} - \vec b^q}
    +
    \norm{c^{q_j} - c^q}
    \right)
    \left(
    \norm{\D^2 e_j} + \norm{\grad e_j} + \norm{e_j}
    \right)
    \text{ a.e. in } \W.
  \end{multline}
  Therefore, we have
  \begin{multline}
    \label{ineq:norm-inf}
    \norm{{\nlop F}[v+e_j] - {\nlop F}[v] - \linop L^{q_j} e_j}
    \\
    \leq
    \constref{ineq:norm-1}
    \inf_{q \in \mathcal{N}[v]} 
    \left( \norm{\mat A^{q_j} - \mat A^q}
    +
    \norm{\vec b^{q_j} - \vec b^q}
    +
    \norm{c^{q_j} - c^q}
    \right)
    \left(
    \norm{\D^2 e_j} + \norm{\grad e_j} + \norm{e_j}
    \right)
    \\
    \text{ a.e. in } \W.
  \end{multline}
  By taking $\leb{2}(\W)$-norm of the both sides and then using the generalized 
  H\"older %
  inequality on the right hand side, we arrive at %
  \begin{multline}
    \label{ineq:Norm}
    \Norm{{\nlop F}[v+e_j] - {\nlop F}[v] - \linop L^{q_j} e_j}_{\leb{2}(\W)}
    \\
    \leq
    \constref{ineq:norm-1}
    \Norm{
      \inf_{q \in \mathcal{N}[v]} 
      \left( \norm{\mat A^{q_j} - \mat A^q}
      +
      \norm{\vec b^{q_j} - \vec b^q}
      +
      \norm{c^{q_j} - c^q}
      \right)
    }_{\leb{\infty}(\W)}
    \Norm{e_j}_{\sobh2(\W)}.
  \end{multline}
  After dividing the both sides by $\Norm{e_j}_{\sobh2(\W)}$, 
  the limit (\ref{lim-uniformly-convergent-coefficients}) implies that 
  \begin{equation}
    \label{lim:semi-smooth-2-zeta-norm}
    \lim_{\Norm{e_j}_{\sobh2(\W)}\to0}
    \dfrac{1}{\Norm{e_j}_{\sobh2(\W)}}
    \Norm{{\nlop F}[v+e_j] - {\nlop F}[v] - \linop L^{q_j} e_j}_{\leb{2-\zeta}(\W)}
    = 0
    .
  \end{equation}
  The assertion now follows from that (\ref{lim:semi-smooth-2-zeta-norm}) %
  holds for any $q_j \in \mathcal{N}[v_j]$.
\end{Proof}
\begin{Obs}[relaxing assumptions of Theorem~\ref{the:semi-smooth:HJB-op}]
  \label{rem:relaxing-assumptions-of-Newton-differentiability}
  We note two points relating to Theorem~\ref{the:semi-smooth:HJB-op}.
  \begin{itemize}
  \item
    Apart from Assumption~(\ref{assum-uniformly-convergent-coefficients}) we did not require the ellipticity nor the Cordes condition to
    show Newton differentiability of $\nlop F$ of (\ref{def:op:HJB})
    on the infinite-dimensional space. One may employ the semismooth
    Newton method to HJB equations under weaker conditions.
  \item
    It is easy to check that we can also show this concept for the HJB
    operator on general space $\sobh2(\W)$, which allows more general
    Dirichlet boundary conditions and possibly oblique (second)
    boundary conditions.
  \end{itemize}
\end{Obs}
\begin{Lem}[bounded invertibility of the Newton derivatives]
  \label{lem:boundedness-inverse-DF}
  Under the assumptions of Theorem~\ref{the:semi-smooth:HJB-op} on
  $\W$, $\mat{A},\vec{b},c$ and $f$ and assuming that $\mat{A},\vec{b},c$
  satisfy
  (\ref{def:uniformly-elliptic}) and either
  (\ref{def:general-Cordes-condition}) with $\lambda >0$ or
  (\ref{def:special-Cordes-condition}) with $\lambda =0$ when $\vec b = 0$, $c = 0$, we have that for any $v \in \sobh2(\W) \meet \sobhz1(\W)$ the all members $\linop L^q \in \DNewton {\nlop F}[v]$, which
  $\DNewton {\nlop F}[v]$ is defined by (\ref{eq:semi-smooth-HJB}), are
  nonsingular and $\Norm{{\linop L^q}^{-1}}$ are bounded.
\end{Lem}
\begin{Proof}
  For any $v \in \sobh2(\W) \meet \sobhz1(\W)$ and $q \in \mathcal{N}[v]$,
  thanks to Theorem~\ref{The:well-posedness-strong-solution} for a fixed control map 
  $\alpha = q(\cdot)$, we deduce that $\left\lbrace \linop L^q\right\rbrace $ are invertible and also bijective. 
  It follows from $\mat A, \vec b, c \in \leb\infty(\W;\cont0(\mathcal{A};X))$ 
  for $X= \Symmatrices d, \R d, \reals $ respectively that $\left\lbrace \linop L^q\right\rbrace $ are bounded. 
  Therefore, Banach's inverse mapping theorem implies that $\left\lbrace {\linop L^q}^{-1}\right\rbrace $ are 
  also bounded. 
\end{Proof} 
\begin{The}[superlinear convergence of the semismooth Newton method]
  \label{the:superlinear:convergence}
  Suppose that the operator $ F$ is Newton differentiable with Newton
  derivative $\DNewton F$
  in an open neighborhood $U$ of $u$, solution of $ F[u]=0$. If
  for any $\tilde{u} \in U$, all $\D \in \DNewton F[\tilde{u}]$
  are invertible and 
  $\Norm{\inverseof{D}}$ 
  are bounded, then the iteration 
  \begin{equation} 
    \label{eq:recursive}
    u_{n+1} = u_n - D^{-1}_n F[u_n], \quad D_n \in \DNewton  F[u_n]
  \end{equation}
  converges superlinearly to $u$ provided $u_0$ is sufficiently close to $u$.
\end{The}
For a proof see \citet[Thm.~8.16]{Ito&Kunisch}.
\begin{Cor}[superlinear convergence of the linearized HJB equation]
  \label{cor:superlinear:convergence-HJB}
  Let the operator ${\nlop F}$ as (\ref{def:op:HJB}). Lemma~\ref{lem:boundedness-inverse-DF} 
  and Theorem~\ref{the:superlinear:convergence}
  imply that by choosing an initial guess $u_0$ sufficiently close to the exact solution 
  $u \in \sobh2(\W) \meet \sobhz1(\W)$ of (\ref{eq:HJB-homogeneous}), the solution 
  of recursive problem~(\ref{eq:recursive}) converges superlinearly to $u$.
  By acting $D_n =\linop L^{q_n}$ on the both sides 
  of (\ref{eq:recursive}), it can be rewritten as 
  \begin{equation}
    \label{eq:recursive-practical-a-e}
    \linop L^{q_n} u_{n+1} = f^{q_n}
    \text{ a.e. in }\W.
  \end{equation}
\end{Cor}
\subsection{Howard's algorithm}
\label{sec:Howard-algorithm}
We now present the recursive problem~(\ref{eq:recursive-practical-a-e}) algorithmically, known as Howard's algorithm or policy iteration at the continuous level. Howard's algorithm alternately computes a control and a value function to generate a sequence that converges to the solution of (\ref{eq:HJB-homogeneous}). The process is repeated until the convergence criterion is satisfied. This algorithm can be interpreted as a Newton extension to the equation with a nonsmooth operator (\ref{eq:HJB-homogeneous}). Howard's algorithm for (\ref{eq:HJB-homogeneous}) with initial guess of $u_0$ is described as follows.
\begin{algorithmic}
  \State {Compute the control map $q_0$ such that
    \begin{equation}
      \label{eq:optimize0-Howard's-algorithm}
      q_{0}(\vec x) \in \underset
      {\substack{\alpha \in \mathcal{A}
      }}
      \Argmax \left(  \linop L^\alpha u_0 - f^\alpha \right)(\vec x) \text{ for a.e. } \vec x \text{ in } \W;
    \end{equation}
  }
  \State{
    Solve the linear PDE
    \begin{equation}
      \label{eq:linear1-Howard's-algorithm} 
      \linop L^{q_0} u_{1} - f^{q_0}  = 0 ~ \text{ in } \W
      \tand
      \restriction{ u_{1}}{\boundary\W}=0;
    \end{equation}
  }
  \For{$n\geq 1$}
  \State{
    Update the control map $q_n$ such that %
    \begin{equation}
      \label{eq:optimize-Howard's-algorithm}
      q_{n}(\vec  x) \in \underset
      {\substack{\alpha \in \mathcal{A}
      }}
      \Argmax \left(  \linop L^\alpha u_n - f^\alpha \right)(\vec x) \text{ for a.e. } \vec x \text{ in } \W;
    \end{equation}
  }
  \State{
    Solve the linear PDE
    \begin{equation}
      \label{eq:linear-Howard's-algorithm} 
      \linop L^{q_n} u_{n+1} - f^{q_n}  = 0 ~ \text{ in } \W
      \tand
      \restriction{ u_{n+1}}{\boundary\W}=0;
    \end{equation}
  }
  \EndFor 
\end{algorithmic} 

To follow Howard's algorithm, we need to solve the optimization problem (\ref{eq:optimize-Howard's-algorithm}) and the linear problem in nondivergence form (\ref{eq:linear-Howard's-algorithm}) in each iteration. Lemma~\ref{lem:lim-alpha-j} implies that for any $u_n \in \sobh2(\W) \meet \sobhz1(\W)$, $\mathcal{N}[u_n]$ is nonempty, ensuring the existence of a solution to (\ref{eq:optimize-Howard's-algorithm}) in each iteration. 
For Howard's algorithm to proceed, we also require the existence of a solution to (\ref{eq:linear-Howard's-algorithm}) in each iteration. Since for each $n \in \naturals$, $\quad \mat A^{q_n}$, $\vec b^{q_n}$, $c^{q_n}$ satisfy the Cordes condition either (\ref{def:general-Cordes-condition}) or (\ref{def:special-Cordes-condition}), Theorem~\ref{The:well-posedness-strong-solution} ensures well-posedness of the strong solution $u_n \in \sobh 2(\W) \cap \sobhz 1(\W)$ to (\ref{eq:linear-Howard's-algorithm}) for fixed $\alpha = q_n(\cdot)$. Indeed, the strong solution $u$ represents a fixed-point iteration of the algorithm. 
\\
Remark~(\ref{rem:well-posedness- non-homogeneous}) and the Newton differentiability of the HJB operator on $\sobh2(\W)$, as pointed out in Remark~\ref{rem:relaxing-assumptions-of-Newton-differentiability}, imply that we can directly adapt Howard's algorithm to the nonhomogeneous HJB equation (\ref{eq:HJB0-inhomogeneous}). In this case, we need to replace the boundary condition $\restriction{u_1,u_{n+1}}{\boundary\W}=r$ in (\ref{eq:linear1-Howard's-algorithm}) and (\ref{eq:linear-Howard's-algorithm}).
\section{A least-squares Galerkin gradient recovery for linear
  nondivergence form elliptic PDEs}
\label{sec:Variational-linear}
We review here the
least-squares Galerkin gradient recovery method for linear
equations in nondivergence form as proposed by
\citet{LakkisMousavi:21:article:A-least-squares} and improve it for applying to solve (\ref{eq:linear1-Howard's-algorithm}) and (\ref{eq:linear-Howard's-algorithm}) in the context of Howard's algorithm.
Our improvement involves imposing a tangential trace constraint on the cost functional instead of the function space, which relies on a \indexemph{generalized Maxwell's inequality} presented in Lemma~\ref{lem:general-maxwell-inequality}, extending its analogue in \citet{LakkisMousavi:21:article:A-least-squares}.
This enables us to derive a \indexemph{Miranda--Talenti type estimate} in Theorem~\ref{the:modified-M-T}, which is crucial in proving Theorem~\ref{the:Coercivity-continuity}, that the problem is well-posed using a Lax--Milgram argument. This sets the foundation for a finite element discretization in \Cref{subsec:finite-element-discrete}, for which we derive \apriori error estimates in Theorem~\ref{the:convergence-rate-inhomogen} and \aposteriori estimates in Theorem~\ref{the:residual-error-bound}.
\subsection{Least-squares Galerkin for linear nondivergence form PDE}  
Consider the second order elliptic linear equation in nondivergence form
\begin{equation}
  \label{eq:non-divergence-homogeneous}
  \linop Lu
  :=
  \mat A\frobinner\D^2 u
  +
  \vec b\inner \nabla u
  -
  c u
  =
  f
  ~ \text{ in } \W
  \tand
  \restriction u{\boundary\W}=r
\end{equation}
where $\mat A, \vec b, c \in \leb{\infty}(\W; X)$ for $X= \Symmatrices d$, $\R d$, $\R {\geq0}$ respectively.
These elements satisfy the uniform ellipticity condition (\ref{def:uniformly-elliptic}) and the Cordes condition either (\ref{def:general-Cordes-condition}) with $\lambda>0$ or (\ref{def:special-Cordes-condition}) with $\lambda=0$ when $\vec b=0$, $c=0$, and $f \in \leb2(\W)$, $r \in \sobh{3/2}(\boundary\W)$.

Well-posedness of the strong solution to (\ref{eq:non-divergence-homogeneous}) is achieved via Theorem~\ref{The:well-posedness-strong-solution}, Remark~\ref{rem:well-posedness- non-homogeneous} and Remark~\ref{rem:relaxing-continuity-uniformly-bounded}.
But dealing with such sizeable regular function space ($\sobh2(\W)$) leads to complicated computations that are unpleasant. 
Because of this, we propose to consider an alternative equivalent problem in which its solution resides in a weaker space. To facilitate this, we introduce some notations.

We denote the outer normal to $\W$ by $\normalto\W(\vec x)$ for almost all $\vec x$ on $\boundary\W$. The tangential trace of $\vec \psi \in \sobh1(\W; \R d)$ is defined by
\begin{equation}
  \mathchanges{\tangentialto\W} \vec\psi :=
  \qp{\eye -\normalto\W \normalto\W \inner}\restriction{\vec\psi}{\boundary\W}. 
\end{equation}
We inset the function space
\index{$\linspace w$}
\begin{equation}
  \linspace W:=
  \left\lbrace 
  \vec{\psi}\in\sobh1(\W;\R d)
  \left\vert ~
  \tangentialto\W \vec\psi = 0
  \right\rbrace \right.
  ,
\end{equation}
equipped with the $\sobh1$-norm and consider the following norm for $\sobh1(\W) \times \sobh1\qp{\W;\R d}$,
\begin{equation}
  \Norm{(\varphi , \vec\psi) }_{\sobh1(\W)}^2
  :=
  \Norm{ \varphi  }_{\sobh1(\W)}^2
  +
  \Norm{ \vec\psi  }_{\sobh1(\W)}^2
  \Foreach
  (\varphi ,\vec\psi)
  \in
  \sobh1(\W) \times \sobh1\qp{\W;\R d}
  .
\end{equation}
We denote $\leb2(\W)$/ $\leb2(\boundary\W)$ inner product of two scalar/vector/tensor-valued functions by
\begin{equation}
  \ltwop{\varphi}{\psi}:=
  \int_\W\varphi(\vec x)\star\psi(\vec x)\d\vec x,
  \qquad
  \ltwop{\varphi}{\psi}_{\boundary\W}:=
  \int_{\boundary\W}\varphi(\vec x)\star\psi(\vec x)\ds({\vec x})
\end{equation}
where $\star$ stands for one of the arithmetic, Euclidean-scalar, or Frobenius inner product in $\reals$, $\R d$, or $\realmats dd$ respectively. The notations $\mini{a}b$, $\maxi{a}b$ respectively indicate the minimum and maximum of two values $a, b$.

By considering
$0\leq\theta\leq1$, we define the linear operator
\begin{equation}
  \label{op:on-cross-space}
  \begin{gathered}
    \linop{M}_\theta : \sobh1(\W) \times \sobh1\qp{\W;\R d} \rightarrow \leb2(\W)
    \\
    (\varphi, \vec \psi)
    \mapsto
    \mat A \frobinner \D \vec \psi 
    + 
    \vec b
    \inner
    (
    \theta\vec \psi +(1-\theta) \nabla \varphi
    )
    -
    c \varphi
    =:\linop M_\theta(\varphi,\vec\psi).
  \end{gathered}
\end{equation}
Initially, we assume that the boundary condition is zero, i.e., $\restriction u{\boundary\W}=0$. 
In this regard, we introduce the following quadratic functional on $\sobhz1(\W) \times \sobh1(\W; \R d)$
\begin{equation} 
  \label{functional:minimize}
  E_\theta(\varphi,\vec\psi)
  :=
  \Norm{ \nabla\varphi-\vec\psi}_{\leb2(\W)} ^2 
  +
  \Norm{ \rot\vec\psi }_{\leb2(\W)}^2
  +
  \Norm{\tangentialto\W  \vec \psi }_{\leb2(\boundary\W)}^2
  +
  \Norm{ \linop{M}_\theta(\varphi, \vec\psi)  -f } _{\leb2(\W)}^2
\end{equation} 
where $ \rot\vec\psi$ indicates curl of $\vec \psi$. We then deal with the convex minimization problem of finding a unique pair of the form
\begin{equation}
  \label{eq:minimization}
  (u, \vec g)
  = 
  \underset
      {\substack{
          (\varphi, \vec\psi)\in \sobhz1(\W) \times \sobh1(\W; \R d)
      }}
      \argmin
      E_\theta(\varphi ,\vec\psi )
      .
\end{equation}
The key point is that the problem of finding the strong solution $u$ to (\ref{eq:non-divergence-homogeneous}) with $\restriction u{\boundary\W}=0$ and solving the minimization problem (\ref{eq:minimization}) are equivalent and $\vec g = \grad u$ holds as a consequence of (\ref{eq:minimization}). In the rest of the paper, $\vec g$ will be synonymous with $\grad u$.
The Euler--Lagrange equation of the minimization problem (\ref{eq:minimization}) includes in finding $(u, \vec g) \in \sobhz1(\W) \times \sobh1(\W; \R d)$ such that
\begin{multline}
  \label{eq:E-L-eq-2}
  \qa{ 
    \nabla u-\vec g
    ,
    \nabla \varphi - \vec \psi
  } 
  +
  \qa{
    \rot \vec g  , \rot \vec \psi 
  }
  +
  \qa{
    \tangentialto\W  \vec g 
    ,
    \tangentialto\W \vec \psi
  }_{\boundary\W}
  \\
  +
  \qa{
    \linop{M}_\theta(u , \vec g)
    , 
    \linop{M}_\theta(\varphi , \vec \psi)
  }
  =
  \qa{
    f , \linop{M}_\theta(\varphi , \vec \psi)
  }
  \Foreach (\varphi, \vec\psi)\in \sobhz1(\W) \times \sobh1(\W; \R d).
\end{multline}
Consistent with (\ref{eq:E-L-eq-2}), %
we define the symmetric 
bilinear form
\begin{equation*}
  \funk
      {a_\theta}{
        \ppow{
          \left(  
          \sobhz1(\W) \times \sobh1(\W; \R d) 
          \right) 
        }2
      }
      \reals
\end{equation*}
by
\begin{multline}
  a_\theta
  (\varphi ,\vec\psi
  ;
  \varphi' , \vec\psi')
  :=
  \ltwop{
    \grad \varphi-\vec\psi
  }{
    \grad \varphi'-\vec\psi'
  }
  +
  \ltwop{
    \rot\vec\psi
  }{
    \rot\vec\psi'
  }
  \\
  +
  \ltwop{ 
    \tangentialto\W \vec \psi
  }{ 
    \tangentialto\W \vec \psi'
  }_{\boundary\W}
  +
  \ltwop{
    \linop{M}_\theta(\varphi ,\vec\psi )
  }{
    \linop{M}_\theta(\varphi',\vec\psi')
  } .
\end{multline}
To establish the well-posedness of (\ref{eq:E-L-eq-2}) using the Lax--Milgram theorem, it suffices to demonstrate the coercivity and continuity of the bilinear form $a_\theta$. We address this matter through the following results.
\begin{Lem}[Maxwell's inequality]
  \label{lem:maxwell-inequality}
  For a convex $\W$, every function $\vec\psi\in\sobh1(\W;\R d)$ satisfies
  \begin{multline}
    \label{ineq:maxwell-general}
    \Norm{\D \vec \psi}_{\leb2(\W)}^2
    \leq
    \Norm{\div \vec \psi}_{\leb2(\W)}^2
    +
    \Norm{\rot\vec\psi}_{\leb2(\W)}^2
    \\
    +2\ltwop{
      \tangentialto\W \vec \psi
    }{
      \D\qb{\vec \psi\inner\normalto\W}-\partial_{\normalto\W}(\vec \psi\inner\normalto\W) \normalto\W
    }_{\boundary\W}.
  \end{multline}
\end{Lem}
See \citet[Lem. 2.2 \& Rem. 2.4]{Costabel} for more details. 
The bound (\ref{ineq:maxwell-general}) for $\vec \psi \in \linspace W$ known as
Maxwell's inequality, which is 
\begin{equation}
  \label{ineq:maxwell}
  \Norm{\D \vec \psi}_{\leb2(\W)}^2
  \leq
  \Norm{\div \vec \psi}_{\leb2(\W)}^2
  +
  \Norm{\rot\vec\psi}_{\leb2(\W)}^2.
\end{equation}
\begin{Lem}[generalized Maxwell's inequality]
  \label{lem:general-maxwell-inequality}
  Let $\W$ be a convex domain. There exists $C_{\ref{ineq:general-maxwell}}>0$ 
  such that for any $\vec \psi \in \sobh1(\W; \R d)$,
  \begin{equation}
    \label{ineq:general-maxwell}
    \Norm{\D \vec \psi}_{\leb2(\W)}^2
    \leq
    C_{\ref{ineq:general-maxwell}}
    \qp{
      \Norm{\div \vec \psi}_{\leb2(\W)}^2
      +
      \Norm{\rot\vec\psi}_{\leb2(\W)}^2
      +
      \Norm{\tangentialto\W \vec \psi}_{\leb2(\boundary \W)}^2
    }.
  \end{equation}
\end{Lem}
\begin{Proof}
  We argue by contradiction. Suppose inequality (\ref{ineq:general-maxwell}) fails, 
  then for any $n \in \naturals$ there is $\vec \psi_n \in \sobh1(\W; \R d)$
  such that
  \begin{equation}
    \label{ineq-contradiction-general-maxwell}
    \Norm{\div \vec \psi_n}_{\leb2(\W)}^2
    +
    \Norm{\rot\vec\psi_n}_{\leb2(\W)}^2
    +
    \Norm{\tangentialto\W \vec \psi_n}_{\leb2(\boundary \W)}^2
    <
    \dfrac{1}{n}
    \Norm{\D \vec \psi_n}_{\leb2(\W)}^2.
  \end{equation}
  Without loss of generality, assume $\Norm{\D \vec \psi_n}_{\leb2(\W)}= 1$
  for all $n \in \naturals$. 
  Then (\ref{ineq-contradiction-general-maxwell}) implies that 
  $\Norm{\div \vec \psi_n}_{\leb2(\W)}$, 
  $\Norm{\rot\vec\psi_n}_{\leb2(\W)}$ and 
  $\Norm{\tangentialto\W \vec \psi_n}_{\leb2(\boundary \W)}$
  all converge to zero as $\nty$.
  On the other hand, from (\ref{ineq:maxwell-general}) and by applying
  Cauchy-Schwarz inequality, for all $n \in \naturals$ we have
  \begin{multline}
    \Norm{\D \vec \psi_n}_{\leb2(\W)}^2
    \leq
    \Norm{\div \vec \psi_n}_{\leb2(\W)}^2
    +
    \Norm{\rot\vec\psi_n}_{\leb2(\W)}^2
    \\
    +2\ltwop{
      \tangentialto\W \vec \psi_n
    }{
      \D (\vec \psi_n \inner \normalto\W) - \partial_{\normalto\W}(\vec \psi_n \inner \normalto\W) \normalto\W
    }_{\boundary\W}
    \leq
    \Norm{\div \vec \psi_n}_{\leb2(\W)}^2
    \\
    +
    \Norm{\rot\vec\psi_n}_{\leb2(\W)}^2
    +
    2
    \Norm{\tangentialto\W \vec \psi_n}_{\leb2(\boundary\W)}
    \Norm{\D\qb{\vec \psi_n \inner \normalto\W}
      -
      \partial_{\normalto\W}(\vec \psi_n \inner \normalto\W) \normalto\W}_{\leb2(\boundary\W)}
    .
  \end{multline}
  By taking the limit of the both sides of the above inequality, when $n\rightarrow \infty$, we get
  \begin{multline}
    1= \lim_{n \rightarrow \infty} \Norm{\D \vec \psi_n}_{\leb2(\W)}^2
    \leq
    \lim_{n \rightarrow \infty} \Big(
    \Norm{\div \vec \psi_n}_{\leb2(\W)}^2
    +
    \Norm{\rot\vec\psi_n}_{\leb2(\W)}^2
    \\
    +2
    \Norm{\tangentialto\W \vec \psi_n}_{\leb2(\boundary\W)}
    \Norm{\D (\vec \psi_n \inner \normalto\W)
      -
      \partial_{\normalto\W}\qb{\vec \psi_n \inner \normalto\W} \normalto\W}_{\leb2(\boundary\W)}
    \Big)
    =0,
  \end{multline}
  which it is a contradiction.
\end{Proof}
\begin{Obs}[extending \cite{LakkisMousavi:21:article:A-least-squares}]
  \label{rem:modified-theorems}
  Using the generalized Maxwell inequality (\ref{ineq:general-maxwell}) instead of
  its special case (\ref{ineq:maxwell}) in the proof arguments of Lemma 3.2,
  Theorem 3.6 and Theorem 3.7 of \cite{LakkisMousavi:21:article:A-least-squares} yields the following lemma and theorems.
\end{Obs}
\begin{Lem}[a Miranda-Talenti estimate]
  \label{lem-M-T-estimate}
  Let $\mat A$ satisfies the Cordes condition with $\lambda = 0$ (\ref{def:special-Cordes-condition}),
  then there exists $C_{\ref{ineq:M-T-estimate}}>0$ such that for any $\vec \psi \in \sobh1(\W; \R d)$,
  \begin{equation}
    \label{ineq:M-T-estimate}
    \Norm{\rot\vec\psi}_{\leb2(\W)}^2 
    +
    \Norm{\tangentialto\W \vec \psi}_{\leb2(\boundary\W)}^2
    +
    \Norm{\mat A \frobinner \D \vec \psi}_{\leb2(\W)}^2
    \geq
    C_{\ref{ineq:M-T-estimate}}
    \Norm{\D \vec \psi}_{\leb2(\W)}^2.
  \end{equation}
\end{Lem}
\begin{The}[a modified Miranda-Talenti-estimate]
  \label{the:modified-M-T}
  Let $\W$ be a bounded, open convex subset of $\R d$, 
  $0<\rho<\dfrac{1}{\maxi{C_{\ref{ineq:M-T-estimate}}}1}$ and 
  $0\leq\theta\leq1$. Then for any $(\varphi,\vec \psi) \in \sobhz1(\W) \times \sobh1(\W; \R d)$, 
  we have
  \begin{multline}
    \label{ineq:modified-M-T}
    \qp{\dfrac{1}{\maxi{C_{\ref{ineq:M-T-estimate}}}1} - \fracl{\rho}{2}}
    \norm{(\varphi, \vec \psi)}_{\lambda, \theta}^2
    \leq
    \Norm{\rot\vec\psi}_{\leb2(\W)}^2 
    +
    \Norm{\tangentialto\W \vec \psi}_{\leb2(\boundary\W)}^2
    \\
    +
    \Norm{D_{\lambda}(\varphi, \vec \psi)}_{\leb2(\W)}^2 
    +
    \qp{\theta^2 + (1-\theta)^2}
    \fracl{\lambda}{\rho}
    \Norm{\grad \varphi - \vec \psi}_{\leb2(\W)}^2,
  \end{multline}
  where 
  $\norm{(\varphi, \vec \psi)}_{\lambda, \theta}^2:= \Norm{D \vec \psi}_{L^2(\W)}^2 
  + 2 \lambda \Norm{\theta \vec \psi + (1- \theta)\grad \varphi}_{L^2(\W)}^2
  + \lambda^2 \Norm{\varphi}_{L^2(\W)}^2$
  and 
  $D_{\lambda}(\varphi, \vec \psi):= \div \vec \psi - \lambda \varphi$.
\end{The}
\begin{The}[coercivity and continuity of $a_\theta$]
  \label{the:Coercivity-continuity}
  Let $\W \subset \R d$ be a bounded convex open domain and the coefficients 
  $\mat A, \vec b, c$ satisfy (\ref{def:uniformly-elliptic}) and either 
  (\ref{def:general-Cordes-condition}) with 
  $\lambda>0$ or (\ref{def:special-Cordes-condition}) with $\lambda =0$ when 
  $b=0, c=0$ holds. Then $a_\theta$ is coercive and continuous, i.e., there exist 
  $\constref{ineq:coercivity}, \constref{ineq:continuity-linear}>0$ such that for any 
  $(\varphi, \vec \psi), (\varphi', \vec \psi') \in \sobhz1(\W) \times \sobh1(\W; \R d)$
  \begin{multline}
    \label{ineq:coercivity}
    a_\theta(\varphi, \vec \psi ; \varphi, \vec \psi) 
    :=
    \Norm{ \nabla\varphi-\vec\psi}_{\leb2(\W)} ^2 
    +
    \Norm{ \rot\vec\psi }_{\leb2(\W)}^2
    \\
    +
    \Norm{\tangentialto\W \vec \psi}_{\leb2(\boundary\W)}^2
    +
    \Norm{ \linop{M}_\theta(\varphi, \vec\psi) } _{\leb2(\W)}^2
    \geq 
    \constref{ineq:coercivity}
    \Norm{(\varphi, \vec \psi)}_{\sobh1(\W)}^2,
  \end{multline} 
  \begin{equation}
    \label{ineq:continuity-linear}
    \norm{a_\theta(\varphi, \vec \psi ; \varphi', \vec \psi')}
    \leq
    \constref{ineq:continuity-linear}
    \Norm{(\varphi, \vec \psi)}_{\sobh1(\W)} 
    \Norm{(\varphi', \vec \psi')}_{\sobh1(\W)}.
  \end{equation}
\end{The}
\begin{Proof}
  By using Lemma \ref{lem-M-T-estimate} and Theorem \ref{the:modified-M-T}, the proof 
  of coercivity is similar to the argument proof of Theorem 3.7 of \cite{LakkisMousavi:21:article:A-least-squares} and
  for the continuity, we refer to section~3.9 of \cite{LakkisMousavi:21:article:A-least-squares}.
\end{Proof}
\subsection{An equivalent problem to the equation with nonzero boundary}
\label{subsec:equivalent-problem-non-zero-boundary}
If the boundary condition of (\ref{eq:non-divergence-homogeneous}) is nonhomogeneous,
i.e., $\restriction u{\boundary\W}=r \neq 0$ for some $r \in \sobh{3/2}(\boundary \W)$, 
the functional $E_\theta$ is replaced by the extended functional $\tilde{E_\theta}$ as
\begin{equation}
  \begin{gathered}
    \funk{ \tilde{E_\theta}}{\sobh1(\W) \times \sobh1(\W; \R d)
    }\reals
    \\
    \begin{aligned}
      \tilde{E_\theta} (\varphi,\vec\psi)
      := 
      \Norm{ \grad\varphi-\vec\psi}_{\leb2(\W)} ^2
      &
      +
      \Norm{ \rot \vec \psi }_{\leb2(\W)} ^2 
      +
      \Norm{ \linop{M}_\theta(\varphi, \vec\psi)  -f } _{\leb2(\W)}^2
      \\
      +
      \Norm{ \varphi-r }_{\leb2(\boundary\W)}^2
      &
      +
      \Norm{ \tangentialto\W( \vec \psi - \grad r) }_{\leb2(\boundary\W)}^2
    \end{aligned}
  \end{gathered} 
\end{equation} 
and we then consider the Euler--Lagrange equation of the minimization problem 
\begin{equation}
  \label{eq:minimization-non-homogeneous}
  (u, \vec g)
  =\underset
  {\substack{
      (\varphi , \vec\psi) \in\sobh1(\W) \times \sobh1(\W; \R d)
  }}
  \argmin \tilde{E_\theta}(\varphi ,\vec\psi) .
\end{equation}
Indeed, we find $(u, \vec g) \in \sobh1(\W) \times \sobh1(\W; \R d)$ such that,
\begin{multline}
  \label{eq:E-L-eq-2-nonhogen}
  \qa{ 
    \nabla u-\vec g
    ,
    \nabla \varphi - \vec \psi
  } 
  +
  \qa{
    \rot \vec g  , \rot \vec \psi 
  }
  +
  \qa{
    \linop{M}_\theta(u , \vec g)
    , 
    \linop{M}_\theta(\varphi , \vec \psi)
  } 
  +
  \qa{u , \varphi}_{\boundary\W} 
  \\
  +
  \qa{
    \tangentialto\W \vec g 
    ,
    \tangentialto\W \vec \psi 
  }_{\boundary\W}
  =
  \qa{
    f , \linop{M}_\theta(\varphi , \vec \psi)
  }
  +
  \qa{r , \varphi}_{\boundary\W} 
  +
  \qa{
    \tangentialto\W \grad r 
    ,
    \tangentialto\W \vec \psi 
  }_{\boundary\W}
  \\
  \Foreach (\varphi, \vec\psi)\in \sobh1(\W) \times \sobh1(\W; \R d).
\end{multline}
The well-posedness of the resulting problem can be demonstrated using a similar argument to the proof of Theorem~\ref{the:Coercivity-continuity}. Accordingly, we define the bilinear form
\begin{equation}
  \begin{gathered}
    \funk
        {\tilde{a}_\theta}{
          \ppow{
            \sobh1(\W) \times \sobh1(\W; \R d) 
          }2
        }
        \reals
        \\
        \tilde{a}_\theta
        (\varphi ,\vec\psi
        ;
        \varphi' , \vec\psi')
        := 
        a_\theta
        (\varphi ,\vec\psi
        ;
        \varphi' , \vec\psi')
        +
        \ltwop{
          \varphi}{\varphi'}_{\boundary\W}.
  \end{gathered} 
\end{equation}  
\begin{The}[coercivity and continuity of $\tilde{a}_\theta$]
  \label{the:Coercivity-continuity-a-tilde}
  Under the assumption of Lemma~\ref{the:Coercivity-continuity}, there exist 
  $\constref{ineq:coercivity-inhomogen}, \constref{ineq:continuity-inhomogen} > 0$ 
  such that for any 
  $(\varphi, \vec \psi), (\varphi', \vec \psi') \in \sobh1(\W) \times \sobh1(\W; \R d)$
  \begin{equation}
    \label{ineq:coercivity-inhomogen}
    \tilde{a}_\theta(\varphi, \vec \psi ; \varphi, \vec \psi) 
    \geq 
    \constref{ineq:coercivity-inhomogen}
    \Norm{(\varphi, \vec \psi)}_{\sobh1(\W)}^2
  \end{equation}
  \begin{equation}
    \label{ineq:continuity-inhomogen}
    \norm{\tilde{a}_\theta(\varphi, \vec \psi ; \varphi', \vec \psi')}
    \leq
    \constref{ineq:continuity-inhomogen}
    \Norm{(\varphi, \vec \psi)}_{\sobh1(\W)} 
    \Norm{(\varphi', \vec \psi')}_{\sobh1(\W)}
  \end{equation}
\end{The}
\begin{Proof}
  Using Poincaré's inequality: there exists $C_{\ref{ineq:Poincare}}>0$ such that for any $(\varphi, \vec \psi) \in \sobh1(\W) \times \sobh1(\W; \R d)$ we have
  \begin{equation}%
    \begin{gathered}
      \label{ineq:Poincare}
      \Norm{\grad \varphi}_{\leb2(\W)}^2
      +
      \Norm{\varphi}_{\leb2(\boundary\W)}^2
      \geq
      C_{\ref{ineq:Poincare}}
      \Norm{\varphi}_{\sobh1(\W)}^2,
      \\
      \Norm{\D \vec \psi}_{\leb2(\W)}^2
      +
      \Norm{\tangentialto\W \vec \psi}_{\leb2(\boundary\W)}^2
      \geq
      C_{\ref{ineq:Poincare}}
      \Norm{\vec \psi}_{\sobh1(\W)}^2,
    \end{gathered}
  \end{equation}%
  the argument is analogous to what is presented in the proof of Theorem~\ref{the:Coercivity-continuity}.
\end{Proof}
\subsection{Galerkin finite element discretization}
\label{subsec:finite-element-discrete}
Consider $\mathfrak T$ as a collection of shape-regular conforming
\indexen{simplicial parititions} \aka{\indexen{triangulations}} of
$\W$ into \indexen{simplices}. Given $\mesh t\in\mathfrak T$, for each
$K\in\mesh T$, let $h_{K}:=\diam{K}$ and write $h:=\underset {\substack{ K \in\mesh T}}\max~ h_{K}$.
Having a curved boundary $\boundary\W$ prevents $ \underset {\substack{ K \in\mesh T}} \bigcup K $ of coinciding with $\W$. In this case, one can approximate sections of $\partial\W$, using line segments or simple curves. This approach results in simplices with curved sides, known as isoparametric elements.
Consider the following Galerkin finite 
element spaces
\begin{gather}
  \label{def:interpolation-space}
  \fespace u
  :=
  \poly{k}\qp{\mesh T}\meet\sobhz1(\W),
  \quad
  \tilde{\fespace u}
  :=
  \poly{k}\qp{\mesh T}\meet\sobh1(\W),
  \quad 
  \fespace g
  :=
  \poly{k}\qp{\mesh T;\R d}
  \meet
  \sobh1(\W; \R d)
  .
\end{gather}
By applying these Galerkin finite element spaces, the discrete problem corresponding to 
zero boundary condition finds $(\fe u_{\fespace u}, \vecfe g_{\fespace g}) \in \fespace u \times \fespace g$ 
such that
\begin{equation}
  \label{eq:discrete}
  a_\theta(\fe u_{\fespace u}, \vecfe g_{\fespace g}; \varphi, \vec \psi) 
  =
  \qa{
    f , \linop{M}_\theta(\varphi , \vec \psi)
  }
  \Foreach (\varphi, \vec\psi)\in \fespace u \times \fespace g
\end{equation}
and the discrete problem corresponding to 
nonzero boundary condition finds 
$(\fe u_{\tilde{\fespace u}}, \vecfe g_{\fespace g}) \in \tilde{\fespace u} \times \fespace g$ 
such that
\begin{equation}
  \label{eq:discrete-non-zero-boundary}
  \tilde{a}_\theta(\fe u_{\tilde{\fespace u}}, \vecfe g_{\fespace g}; \varphi, \vec \psi)  
  =
  \qa{f,  \linop{M}_\theta(\varphi, \vec \psi )} 
  +
  \qa{ r, \varphi}_{\boundary\W}
  +
  \qa{\tangentialto\W \grad r , \tangentialto\W \vec \psi}_{\boundary\W} 
  \Foreach (\varphi, \vec\psi)\in \tilde{\fespace u} \times \fespace g.
\end{equation}
Since the coercivity is inherited by subspaces,
Theorem~\ref{the:Coercivity-continuity} and
Theorem~\ref{the:Coercivity-continuity-a-tilde} imply that both
discrete problems~(\ref{eq:discrete}) and
(\ref{eq:discrete-non-zero-boundary}) are well-posed.
The solutions of both discrete problems satisfy the following
error estimate theorems.
\begin{The}[\apriori error estimate for linear equation in nondivergence form]
  \label{the:convergence-rate}
  Let $\mesh T$ be in a collection $\mathfrak{T}$ of shape-regular 
  conforming simplicial meshes on the polyhedral domain $\W\subseteq\R d$. 
  Moreover assume that the strong solution $u$ of 
  (\ref{eq:non-divergence-homogeneous}) with $\restriction u{\boundary\W}=0$ 
  satisfies $u \in \sobh{\rho+2}(\W)$,
  for some real $\rho > 0$.
  Let $(\fe u_{\fespace u}, \vecfe g_{\fespace g}) \in \fespace u \times \fespace g$ 
  be the finite element approximation of (\ref{eq:discrete}).
  Then for some 
  $\constref{eqn:convergence-rate}>0$, independent
  of $u$ and $h$ we have
  \begin{equation}
    \label{eqn:convergence-rate}
    \Norm{(u,\grad u) - (\fe u_{\fespace u}, \vecfe g_{\fespace g})}_{\sobh1(\W)} 
    \leq
    \constref{eqn:convergence-rate} h^{\mini{k}\rho}
    \Norm{u}_{\sobh{\rho+2}(\W)}.
  \end{equation}
\end{The}
\begin{Proof}
  We refer to Theorem~4.6 of \cite{LakkisMousavi:21:article:A-least-squares}.
\end{Proof}
With a same argument of the proof, we can also show the \apriori error bound of the 
discrete problem corresponding to nonzero boundary as the following theorem.
\begin{The}[\apriori error estimate for linear equation in nondivergence form with nonzero boundary]
  \label{the:convergence-rate-inhomogen}
  Suppose that the assumptions of Theorem~\ref{the:convergence-rate} on the domain 
  and simplicial meshes are hold. Assume that the strong solution $u$ of 
  (\ref{eq:non-divergence-homogeneous}) with $\restriction u{\boundary\W}=r \neq 0$ 
  for some $r \in \sobh{3/2}(\boundary \W)$ satisfies $u \in \sobh{\rho+2}(\W)$,
  for some real $\rho > 0$. Let 
  $(\fe u_{\tilde{\fespace u}}, \vecfe g_{\fespace g}) \in \tilde{\fespace u} \times \fespace g$ 
  be the finite element approximation of (\ref{eq:discrete-non-zero-boundary}).
  Then for some 
  $\constref{eqn:convergence-rate-inhomogen}>0$, independent
  of $u$ and $h$ we have
  \begin{equation}
    \label{eqn:convergence-rate-inhomogen}
    \Norm{(u,\grad u) - (\fe u_{\tilde{\fespace u}}, \vecfe g_{\fespace g})}_{\sobh1(\W)} 
    \leq
    \constref{eqn:convergence-rate-inhomogen} h^{\mini{k}\rho}
    \Norm{u}_{\sobh{\rho+2}(\W)}.
  \end{equation}
\end{The}
\begin{Obs}[domain with curved boundary]
  If $\W$ includes a curved boundary, the isoparametric finite element is used. Smooth 
  or piecewise smooth boundary($\boundary\W$) ensures that the error bound of 
  using isoparametric finite element similar to that of Theorem~\ref{the:convergence-rate} 
  and Theorem~\ref{the:convergence-rate-inhomogen}. 
  This is established in \cite{Ciarlet}.
\end{Obs}
\begin{The}[\aposteriori error-residual estimate for linear equation in nondivergence form]
  \label{the:residual-error-bound}
  Let $u$ be the strong solution of (\ref{eq:non-divergence-homogeneous}).
  \begin{enumerate}[(a)\ ]
  \item
    For zero boundary condition, let $(\fe u_{\fespace u}, \vecfe g_{\fespace g})$ be the unique 
    solution of (\ref{eq:discrete}).
    \begin{enumerate}[(i)\ ]
    \item
      The following a posteriori residual upper bounds holds
      \begin{equation}
        \label{eqn:aposteriori-error-bound}
        \Norm{ (u, \grad u)-(\fe u_{\fespace u},\vecfe g_{\fespace g})}_{\sobh1(\W)}^2
        \leq
        \constref{ineq:coercivity}^{-1} E_\theta (u_{\fespace u}, \vec g_{\fespace g}).
      \end{equation} 
    \item
      For any open subdomain $\w \subseteq \Omega$ we have 
      \begin{multline*}
        \label{eqn:lower-error-bound}
        \Norm{ \grad \fe u_{\fespace u}- \vecfe g_{\fespace g} }_{\leb2(\w)}^2 
        +
        \Norm{ \rot \vecfe g_{\fespace g} }_{\leb2(\w)}^2
        +
        \Norm{ \linop{M}_\theta(\fe u_{\fespace u}, \vecfe g_{\fespace g}) -f }_{\leb2(\w)}^2  
        +
        \Norm{ \tangentialto\W \vecfe g_{\fespace g} }_{\leb2(\boundary\w \meet \boundary\W)}^2 
        \\
        \leq
        \constref[\w]{ineq:continuity-linear}
        \Norm{ (u, \grad u)-(\fe u_{\fespace u},\vecfe g_{\fespace g})}_{\sobh1(\w)}^2,
      \end{multline*}
      where $\constref[\w]{ineq:continuity-linear}$ is the continuity constant of $a_\theta$ 
      on $\sobhz1(\w) \times \sobh1(\w; \R d)$.
    \end{enumerate}
  \item
    For nonzero boundary condition, let $(\fe u_{\tilde{\fespace u}}, \vecfe g_{\fespace g})$ be
    the unique solution of (\ref{eq:discrete-non-zero-boundary}).
    \begin{enumerate}[(i)\ ]
    \item
      The following a posteriori residual upper bounds holds
      \begin{equation}
        \label{eqn:aposteriori-error-bound-inhomogen}
        \Norm{ (u, \grad u)-(\fe u_{\tilde{\fespace u}},\vecfe g_{\fespace g})}_{\sobh1(\W)}^2
        \leq
	\constref{ineq:coercivity-inhomogen}^{-1}\tilde{ E_\theta} (\fe u_{\tilde{\fespace u}}, \vecfe g_{\fespace g}).
      \end{equation} 			
    \item
      For any open subdomain $\w \subseteq \Omega$ we have 
      \begin{multline*}
        \label{eqn:lower-error-bound-inhomogen}
        \Norm{ \grad \fe u_{\tilde{\fespace u}}- \vecfe g_{\fespace g} }_{\leb2(\w)}^2 
        +
        \Norm{ \rot \vecfe g_{\fespace g} }_{\leb2(\w)}^2
        +
        \Norm{ \linop{M}_\theta(\fe u_{\tilde{\fespace u}}, \vecfe g_{\fespace g}) -f }_{\leb2(\w)}^2 
        +
        \Norm{ \fe u_{\tilde{\fespace u}} - r }_{\leb2(\boundary\w \meet \boundary\W)}^2
        \\
        +
        \Norm{ \tangentialto\W ( \vecfe g_{\fespace g}  - \grad r) }_{\leb2(\boundary\w \meet \boundary\W)}^2 
        \leq
        \constref[\w]{ineq:continuity-inhomogen}
        \Norm{ (u, \grad u)-(\fe u_{\tilde{\fespace u}},\vecfe g_{\fespace g})}_{\sobh1(\w)}^2,
      \end{multline*}
      where $\constref[\w]{ineq:continuity-inhomogen}$ is the continuity 
      constant of $\tilde{a}_\theta$ on $\sobh1(\w) \times \sobh1(\w; \R d)$.
    \end{enumerate}
  \end{enumerate}
\end{The}
\begin{Proof}
  We refer to Theorem~4.3 of \cite{LakkisMousavi:21:article:A-least-squares}.
\end{Proof}
\section{Discretization}
\label{sec:Discretization}
In this section, we turn to the practical approximation of Howard's algorithm, as clarified in \S\ref{sec:Howard-algorithm}, by applying the method from \S\ref{sec:Variational-linear} to solve each iteration of problem~(\ref{eq:linear-Howard's-algorithm}). 
We discuss the resulting approximation error as well.
Hereupon, we  present an \apriori (Theorem~\ref{pro:convergence-rate-HJB}) and an a posteriori (Theorem \ref{eqn:aposteriori-error-bound-HJB}) error analysis of the approximation. %
In the spirit of the a posteriori error estimator, we specify the error indicators in an adaptive mesh  refinement algorithm. 
Afterwards, an approximation of the control problem is explained and finally, we close the section by offering Howard's and adaptive algorithms, which are used in the implementation.
\subsection{Discretization of recursive problem}
\label{sec:error-analysis}  
Let the simplicial mesh $\mesh T $ and the Galerkin finite element spaces 
${\fespace u}, \tilde{\fespace u}, {\fespace g}$
be as introduced in \S \ref{subsec:finite-element-discrete}.
For any fixed $\theta \in [0,1]$ and control map $q \in \mathcal{Q}$ %
we associate a linear operator $ \linop M_\theta^{q}$ acting on 
$(\varphi, \vec \psi) \in \sobh1(\W) \times \sobh1(\W; \R d)$ as follows 
\begin{equation}
  \linop M_\theta^{q}(\varphi,\vec\psi):= 
  \mat A^{q} \frobinner \D \vec \psi 
  + 
  \vec b ^{q}
  \inner
  (
  \theta\vec \psi +(1-\theta) \nabla \varphi
  )
  -
  c^{q} \varphi.
\end{equation}
Furthermore, we define the set-valued operator 
$\tilde{\mathcal N}: \sobh1(\W) \times \sobh1(\W; \R d) \rightrightarrows \mathcal{Q}$ by
\begin{equation}
  \label{def:N-tild-set}
  \tilde{\mathcal{N}}(\varphi, \vec \psi)
  :=
  \left\lbrace 
  q \in \mathcal{Q}
  \left\vert ~
  q(\vec x) \in  
  \underset{
    \substack{\alpha \in \mathcal{A}
  }}
  \Argmax
  \evalat{
    \linop M_\theta^{\alpha}(\varphi,\vec\psi) - f^\alpha
  }{\vec x},
  \text{ for a.e. }  \vec x \text{ in } \W
  \right\rbrace \right.
  ,
\end{equation} 
where $\theta$ is implicit in the notation
Recall the recursive problem~(\ref{eq:recursive-practical-a-e}) and let $({\fe u_{\fespace u}}_n, {\vecfe g_{\fespace g}}_n)$ be the least-squares 
Galerkin with gradient recovery finite element approximation of it %
at step $n$, approximated by setting 
$q_{n-1} \in \tilde{\mathcal{N}}({\fe u_{\fespace u}}_{n-1}, {\vecfe g_{\fespace g}}_{n-1})$. 
Indeed, $({\fe u_{\fespace u}}_n,{\vecfe g_{\fespace g}}_n)$ is the solution of (\ref{eq:discrete}) 
for $ \linop M_{\theta} = \linop M^{q_{n-1}}_{\theta} $ and $f = f^{q_{n-1}}$. 
Let 
\begin{equation}
  \label{def:limit-approximation}
  (\fe u_{\fespace u}, \vecfe g_{\fespace g})=\underset{\substack{n\rightarrow \infty}}
  \lim({\fe u_{\fespace u}}_n, {\vecfe g_{\fespace g}}_n).
\end{equation} 

We analyze errors by examining the limiting behavior of both the exact solution and the approximation of the recursive problem (\ref{eq:recursive-practical-a-e}). To report error bounds, we will consider an equivalent problem to \eqref{eq:HJB-homogeneous}, where $(u, \nabla u)$ represents its exact solution, and $(u_{\fespace u}, \vec g_{\fespace g})$ represents its approximation. In this regard, it is appropriate to consider a mixed formulation of the HJB problem. Corresponding to the approach used to deal with the linear equation at each iteration, determining such a nonlinear form of the equation is straightforward.
\subsection{Mixed formulation of HJB problem}
To introduce the mixed formulation corresponding to 
HJB problem (\ref{eq:HJB-homogeneous}),
we define the mixed HJB operator as following
\begin{equation}
  \begin{gathered}
    \funk{ \hat{\nlop F}}{\sobh1(\W) \times \sobh1(\W; \R d)
    }\leb2(\W)
    \\
    \begin{aligned}
      \hat{\nlop F}(\varphi, \vec\psi): = 	
      \linop{M}_\theta^q(\varphi, \vec\psi)  -f^q 
      \text{ such that } 
      q \in \tilde{\mathcal{N}}(\varphi, \vec \psi),
    \end{aligned}
  \end{gathered}
\end{equation}
where similar to (\ref{def:N-tild-set}), $\theta$ is muted in the notation.
Through adapting the proof of Theorem~\ref{the:semi-smooth:HJB-op} 
to operator $\hat{\nlop F}$, we argue that for any 
$(\varphi, \vec \psi) \in \sobh1(\W) \times \sobh1(\W; \R d)$
\begin{equation}
  \label{def:newton-derivative-mixed-HJB}
  \DNewton {\hat{\nlop F}}[(\varphi, \vec \psi)]:=
  \left\lbrace 		
  \linop M_\theta^{q} := 
  \mat A^{q} \frobinner \D ~ %
  + 
  \vec b ^{q}
  \inner
  (
  \theta
  +(1-\theta) \nabla ~ %
  )
  -
  c^{q} %
  \left\vert ~
  q \in \tilde{\mathcal{N}}(\varphi, \vec \psi)
  \right\rbrace \right.
\end{equation} 
is the Newton derivative of $\hat{\nlop F}$ at $(\varphi, \vec \psi)$.
\newline
Based on the least-squares idea applied for the linearized problem, we introduce the following quadratic functional
\begin{equation}
  \begin{gathered}
    \funk{ \hat{E_\theta}}{\sobhz1(\W) \times \sobh1(\W; \R d)
    }\reals
    \\
    \begin{aligned}
      \hat{E_\theta} (\varphi,\vec\psi)
      := 
      \Norm{ \grad\varphi-\vec\psi}_{\leb2(\W)} ^2
      +
      \Norm{ \rot \vec \psi }_{\leb2(\W)} ^2 
      +
      \Norm{ \tangentialto\W( \vec \psi) }_{\leb2(\boundary\W)}^2
      +
      \Norm{ \linop{M}_\theta^q(\varphi, \vec\psi)  -f^q } _{\leb2(\W)}^2
    \end{aligned}
  \end{gathered} 
\end{equation} 
in which $q \in  \tilde{\mathcal{N}}(\varphi, \vec \psi)$. 
We then consider the minimization problem of finding 
$(u, \vec g) \in \sobhz1(\W) \times \sobh1(\W; \R d)$ 
such that 
\begin{equation}
  \label{eq:minimization-mixed}
  (u, \vec g)
  =\underset
  {\substack{
      (\varphi , \vec\psi) \in\sobhz1(\W) \times \sobh1(\W; \R d)
  }}
  \argmin \hat{E_\theta}(\varphi ,\vec\psi) .
\end{equation} 
Obviously, $\vec g = \grad u$ and $u$ is the unique strong solution of (\ref{eq:HJB-homogeneous}).
By applying Newton derivative (\ref{def:newton-derivative-mixed-HJB}), 
the Euler-Lagrange equation of the minimization problem (\ref{eq:minimization-mixed}) 
finds $ (u, \vec g) \in \sobhz1(\W) \times \sobh1(\W; \R d)$ such that satisfies the binonlinear problem
\begin{multline}
  \label{eq:mixed-Eu-La}
  \ltwop{
    \grad u-\vec g
  }{
    \grad \varphi-\vec\psi
  }
  +
  \ltwop{
    \rot\vec g
  }{
    \rot\vec\psi
  }
  +
  \qa{
    \tangentialto\W \vec g
    ,
    \tangentialto\W \vec \psi 
  }_{\boundary\W}
  \\
  +
  \ltwop{
    \linop{M}_\theta^q( u ,\vec g ) - f^q
  }{
    \linop{M}_\theta^q(\varphi,\vec\psi)
  }   
  =0,
  \\ 
  \text{ for }
  q \in \tilde{\mathcal{N}}(u, \vec g)
  \text{ and each }  
  (\varphi, \vec\psi) \in \sobhz1(\W) \times \sobh1(\W; \R d).
\end{multline}
Respectively, the discrete nonlinear problem finds $(\fe u_{\fespace u}, \vecfe g_{\fespace g}) \in \fespace u \times \fespace g$ 
such that
\begin{multline}
  \label{eq:mixed-Eu-La-discrete}
  \ltwop{
    \grad \fe u_{\fespace u} - \vecfe g_{\fespace g}
  }{
    \grad \varphi-\vec\psi
  }
  +
  \ltwop{
    \rot \vecfe g_{\fespace g}
  }{
    \rot\vec\psi
  }
  +
  \qa{
    \tangentialto\W \vecfe g_{\fespace g}
    ,
    \tangentialto\W \vec \psi 
  }_{\boundary\W}
  \\
  +
  \ltwop{
    \linop{M}_\theta^q( u_{\fespace u} ,\vecfe g_{\fespace g} ) - f^q
  }{
    \linop{M}_\theta^q(\varphi,\vec\psi)
  }   
  =0,
  \\ 
  \text{ for }
  q \in \tilde{\mathcal{N}}(u_{\fespace u}, \vecfe g_{\fespace g})
  \text{ and each }  
  (\varphi, \vec\psi) \in \fespace u \times \fespace g.
\end{multline}
To achieve linearity, applying the semismooth Newton linearization to \eqref{eq:mixed-Eu-La-discrete} yields a recursive bilinear form equation, as in \eqref{eq:discrete}.
Hence, the solutions of \eqref{eq:mixed-Eu-La-discrete} and \eqref{def:limit-approximation} coincide, and we use the same notation for them, as explained in the following remark. Consequently, to provide error bounds, we consider problem \eqref{eq:mixed-Eu-La} and its discrete version \eqref{eq:mixed-Eu-La-discrete}.
\begin{Obs}
  When discretizing the HJB equation using the Galerkin finite element method, the steps of discretization and linearization are computationally commutative. This means that whether we apply the Galerkin method first to enter a finite-dimensional function space and then perform linearization, or if we linearize first and then use finite element discretization, the outcomes are equivalent.
\end{Obs}

\subsection{Supremum property}
We recall that for real numbers 
$\left\lbrace x^{\alpha} \right\rbrace_{\alpha} , \left\lbrace y^{\alpha} \right\rbrace_{\alpha}$, 
we have 
\begin{equation}
  \label{ineq:sup-property}
  \norm{\sup_{\alpha} x^{\alpha}  -  \sup_{\alpha} y^{\alpha}} 
  \leq
  \sup_{\alpha}
  \norm{x^{\alpha} - y^{\alpha}}.
\end{equation}
This property is used to show the following claims.
\begin{Lem}[weak monotonicity and Lipschitz continuity]
  \label{lem:semi-monotonicity-continuity}
  Let $\W$ be a convex domain. There exist 
  $C_{\ref{ineq:smi-monotonicity}} , C_{\ref{ineq:continuity}} > 0 $ 
  such that every 
  $(\varphi , \vec \psi), (\varphi' , \vec \psi'), (\varphi'' , \vec \psi'') , (\varphi''', \vec \psi''') 
  \in \sobhz1(\W) \times \sobh1(\W; \R d)$ 
  and $q \in \tilde{\mathcal{N}}(\varphi , \vec \psi) , q' \in \tilde{\mathcal{N}}(\varphi' , \vec \psi') , 
  q'' \in \tilde{\mathcal{N}}(\varphi'' , \vec \psi'') , q''' \in \tilde{\mathcal{N}}(\varphi''' , \vec \psi''')$ 
  satisfy
  \begin{multline}
    \label{ineq:smi-monotonicity}
    \Norm{ \grad(\varphi - \varphi')-(\vec\psi - \vec\psi')}_{\leb2(\W)} ^2
    +
    \Norm{ \rot (\vec \psi - \vec\psi') }_{\leb2(\W)} ^2 
    +
    \Norm{ \tangentialto\W( \vec \psi - \vec\psi') }_{\leb2(\boundary\W)}^2
    \\
    +
    \Norm{
      \qp{\linop{M}_\theta^{q}(\varphi , \vec \psi) - f^q}
      -
      \qp{\linop{M}_\theta^{q'}(\varphi' , \vec \psi') - f^{q'}}
    }_{\leb2(\W)}^2
    \geq C_{\ref{ineq:smi-monotonicity}}
    \Norm{(\varphi , \vec \psi) - (\varphi' , \vec \psi')}_{\sobh1(\W)}^2.
  \end{multline}
  \begin{multline}
    \label{ineq:continuity}
    \ltwop{
      \grad (\varphi - \varphi')- (\vec \psi - \vec \psi')
    }{
      \grad (\varphi'' - \varphi''')- (\vec\psi'' - \vec \psi''')
    }
    \\
    +
    \ltwop{
      \rot (\vec \psi - \vec \psi')
    }{
      \rot(\vec \psi'' - \vec \psi''')
    }
    +
    \qa{
      \tangentialto\W (\vec \psi - \vec \psi')
      ,
      \tangentialto\W (\vec \psi'' - \vec \psi''') 
    }_{\boundary\W}
    \\
    +
    \ltwop{
      \qp{\linop{M}_\theta^q( \varphi ,\vec \psi ) - f^q}
      -
      \qp{\linop{M}_\theta^{q'}( \varphi' ,\vec \psi' ) - f^{q'}}
    }{
      \qp{\linop{M}_\theta^{q''}( \varphi'' ,\vec \psi'' ) - f^{q''}}
      -
      \qp{\linop{M}_\theta^{q'''}( \varphi''' ,\vec \psi''' ) - f^{q'''}}
    }   
    \\
    \leq
    C_{\ref{ineq:continuity}}
    \Norm{\qp{ \varphi ,\vec \psi} - \qp{ \varphi' ,\vec \psi'}}_{\sobh1(\W)}
    \Norm{\qp{ \varphi'' ,\vec \psi''} - \qp{ \varphi''' ,\vec \psi'''}}_{\sobh1(\W)}.
  \end{multline} 
\end{Lem}
\begin{Proof}
  By tracking what is saying in the proof of Theorem \ref{the:Coercivity-continuity} 
  and using (\ref{ineq:sup-property}), the inequalities are achieved. 
\end{Proof}
\begin{Obs}[non strong monotonicity]
  We want to emphasize that (\ref{ineq:smi-monotonicity}) differs from the strong monotonicity of the binonlinear form in Equation (\ref{eq:mixed-Eu-La}). It can be verified that the corresponding binonlinear form is not strongly monotone.
\end{Obs}
\begin{The}[quasi-optimality]
  \label{the:quasi-optimality}
  Let $(\fe u_{\fespace u}, \vecfe g_{\fespace g})$ be the solution of the
  discrete problem (\ref{eq:mixed-Eu-La-discrete}). It satisfies the error bound
  \begin{equation}
    \label{ineq:quasi-optimality}
    \Norm{(u,\grad u) - (\fe u_{\fespace u}, \vecfe g_{\fespace g})}_{\sobh1(\W)} 
    \leq
    \frac{C_{\ref{ineq:continuity}}}{C_{\ref{ineq:smi-monotonicity}}}
    \inf_{(\varphi, \vec \psi) \in \fespace u \times \fespace g}
    \Norm{(u,\grad u) - (\varphi, \vec \psi)}_{\sobh1(\W)} .
  \end{equation}
\end{The}
\begin{Proof}
  Consider any arbitrary $(\varphi, \vec \psi) \in \fespace u \times \fespace g$. 
  Since for $q \in \tilde{\mathcal{N}}(u , \grad u)$, 
  $\linop{M}_\theta^{q}(u , \grad u) - f^{q} = 0$, 
  Lemma \ref{lem:semi-monotonicity-continuity} 
  and (\ref{eq:mixed-Eu-La-discrete}) imply that for 
  $q' \in \tilde{\mathcal{N}}(\fe u_{\fespace u}, \vecfe g_{\fespace g})$ we have
  \begin{multline}
    \label{ineq-proof-quasi-opt}
    C_{\ref{ineq:smi-monotonicity}}
    \Norm{(u,\grad u) - (\fe u_{\fespace u}, \vecfe g_{\fespace g})}_{\sobh1(\W)}^2 
    \\
    \leq
    \Norm{ \grad \fe u_{\fespace u} - \vecfe g_{\fespace g}}_{\leb2(\W)} ^2
    +
    \Norm{ \rot \vecfe g_{\fespace g} }_{\leb2(\W)} ^2 
    +
    \Norm{ \tangentialto\W \vecfe g_{\fespace g} }_{\leb2(\boundary\W)}^2
    +
    \Norm{ 
      \qp{ \linop{M}_{\theta}^{q'}(\fe u_{\fespace u} , \vecfe g_{\fespace g})  - f^{q'} }
    } _{\leb2(\W)}^2
    \\
    =
    \ltwop{
      \grad \fe u_{\fespace u} - \vecfe g_{\fespace g}
    }{
      \grad \varphi-\vec\psi
    }
    +
    \ltwop{
      \rot \vecfe g_{\fespace g}
    }{
      \rot\vec\psi
    }
    +
    \qa{
      \tangentialto\W \vecfe g_{\fespace g}
      ,
      \tangentialto\W \vec \psi 
    }_{\boundary\W}
    \\
    +
    \ltwop{
      \linop{M}_\theta^{q'}( u_{\fespace u} ,\vecfe g_{\fespace g} ) - f^{q'}
    }{
      \linop{M}_\theta^{q'}(\varphi,\vec\psi) - f^{q'}
    }   	
    \\
    = 
    \ltwop{
      \grad\qp{ u -  \fe u_{\fespace u}} - \qp{ \grad u - \vecfe g_{\fespace g}}
    }{
      \grad (u- \varphi) - ( \grad u -\vec\psi)
    }
    +
    \ltwop{
      \rot \qp{ \grad u - \vecfe  g_{\fespace g}}
    }{
      \rot \qp{\grad u - \vec\psi}
    }
    \\
    +
    \qa{
      \tangentialto\W \qp{ \grad u - \vecfe g_{\fespace g} }
      ,
      \tangentialto\W \qp{ \grad u - \vec \psi }
    }_{\boundary\W}
    \\
    +
    \ltwop{
      \linop{M}_\theta^{q}( u ,\grad u ) - f^{q}
      -
      \qp	{\linop{M}_\theta^{q'}( u_{\fespace u} ,\vecfe g_{\fespace g} ) - f^{q'} }
    }{
      \linop{M}_\theta^{q}( u ,\grad u ) - f^{q}
      -
      \qp{	\linop{M}_\theta^{q'}(\varphi,\vec\psi) - f^{q'} }
    }   
    \\
    \leq 
    C_{\ref{ineq:continuity}}
    \Norm{(u, \grad u) - (u_{\fespace u}, \vecfe g_{\fespace g}) }_{\sobh1(\W)}
    \Norm{(u, \grad u) - (\varphi,\vec\psi) }_{\sobh1(\W)}
  \end{multline}

\end{Proof}
\begin{Pro}[\apriori error estimate for HJB equation]
  \label{pro:convergence-rate-HJB}
  Suppose that the strong solution $u$ of 
  (\ref{eq:HJB-homogeneous}) satisfies $u \in \sobh{\rho+2}(\W)$,
  for some real $\rho > 0$. Then for some 
  $\constref{eqn:convergence-rate-HJB}>0$, independent 
  of $u$ and $h$ we have
  \begin{equation}
    \label{eqn:convergence-rate-HJB}
    \Norm{(u,\grad u) - (\fe u_{\fespace u}, \vecfe g_{\fespace g})}_{\sobh1(\W)} 
    \leq
    \constref{eqn:convergence-rate-HJB} h^{\mini{k}\rho}
    \Norm{u}_{\sobh{\rho+2}(\W)}.
  \end{equation}
\end{Pro}
\begin{Proof} 
  (\ref{ineq:quasi-optimality}) and the error bound of the interpolation 
  \citep[Th.4.4.20]{BrennerScott:08:book:The-mathematical} demonstrate the claim.
\end{Proof}
\begin{Pro}[error-residual \aposteriori estimates for HJB equation]
  \label{eqn:aposteriori-error-bound-HJB}
  Let $(\fe u_{\fespace u},\vecfe g_{\fespace g})$ be as considered 
  in (\ref{def:limit-approximation}) and 
  $q \in \tilde{\mathcal{N}}(\fe u_{\fespace u},\vecfe g_{\fespace g})$.
  \begin{itemize}[(i)\ ]
  \item
    The following a posteriori residual 
    upper bound holds 
    \begin{multline}
      \label{eqn:upper-error-bound-HJB}
      \Norm{ (u, \grad u)-(\fe u_{\fespace u},\vecfe g_{\fespace g})}_{\sobh1(\W)}^2
      \leq
      \constref{ineq:smi-monotonicity}^{-1}
      \Big(
      \Norm{ \grad \fe u_{\fespace u}- \vecfe g_{\fespace g} }_{\leb2(\W)}^2 
      +
      \Norm{ \rot \vecfe g_{\fespace g}}_{\leb2(\W)}^2
      \\
      +
      \Norm{ \tangentialto\W \vecfe g_{\fespace g} }_{\leb2(\boundary\W)}^2
      +
      \Norm{ \linop M^{q}_{\theta} (\fe u_{\fespace u}, \vecfe g_{\fespace g} ) - f^{q} }_{\leb2(\W)}^2 
      \Big).
    \end{multline}
  \item%
    For any open subdomain $\w \subseteq \Omega$ the following 
    a posteriori residual lower bound holds
    \begin{multline}
      \label{eqn:lower-error-bound-HJB}
      \Norm{ \grad \fe u_{\fespace u}- \vecfe g_{\fespace g} }_{\leb2(\w)}^2 
      +
      \Norm{ \rot \vecfe g_{\fespace g} }_{\leb2(\w)}^2
      +
      \Norm{ \tangentialto\W \vecfe g_{\fespace g} }_{\leb2(\boundary\w \meet \boundary\W)}^2
      \\
      +
      \Norm{ 
      	\linop M^{q}_\theta(\fe u_{\fespace u}, \vecfe g_{\fespace g}) -f^{q} 
      }_{\leb2(\w)}^2  
      \leq
      \constref[\w]{ineq:continuity}
      \Norm{ (u, \grad u)-(\fe u_{\fespace u},\vecfe g_{\fespace g})}_{\sobh1(\w)}^2,
    \end{multline}
    where \constref[\w]{ineq:continuity} is the constant 
    of (\ref{ineq:continuity}) for subdomain $\w \subseteq \W$.
  \end{itemize}
\end{Pro}
\begin{Proof}
  (\ref{ineq:smi-monotonicity}) and the fact that for $q \in \tilde{\mathcal{N}}(u , \grad u)$,
  $\linop M^q_\theta(u, \grad u) -f^q =0$ imply (\ref{eqn:upper-error-bound-HJB}). 
  \\
  One can easily check that (\ref{ineq:continuity}) holds for any subdomain $\w \subseteq \W$ as well, which implies (\ref{eqn:lower-error-bound-HJB}).
\end{Proof}
Lemma \ref{lem:semi-monotonicity-continuity} and consequently 
propositions \ref{pro:convergence-rate-HJB} and \ref{eqn:aposteriori-error-bound-HJB} 
can be easily adapted to the problem with non-zero boundary, $r\neq 0$.
As we see in propositions~\ref{pro:convergence-rate-HJB}, \ref{eqn:aposteriori-error-bound-HJB}, 
the error bounds of discretization for nonlinear problem are similar to the 
error bounds of discretization in linear case which have reported in 
theorems \ref{the:convergence-rate}--\ref{the:residual-error-bound}.
We use the a posteriori residual error bounds of Proposition \ref{eqn:aposteriori-error-bound-HJB} 
as an error indicator in the adaptive scheme. 
\subsection{Implementation and error indicators}
\label{subsec:Implementation}
Practically, corresponding to the problem with zero boundary, at each step of 
the recursive problem, for fixed control map $q$, we find 
$(\fe u_{\fespace u}, \vecfe g_{\fespace g}) \in \fespace u \times \fespace g$ 
such that
\begin{multline}
  \label{eq:discrete-homogen-practical}
  \ltwop{
    \grad \fe u_{\fespace u}-\vecfe g_{\fespace g}
  }{
    \grad \varphi-\vec\psi
  }
  +
  \ltwop{
    \rot\vecfe g_{\fespace g}
  }{
    \rot\vec\psi
  }
  +
  \ltwop{
    \linop{M}_\theta^q(\fe u_{\fespace u} ,\vecfe g_{\fespace g} )
  }{
    \linop{M}_\theta^q(\varphi,\vec\psi)
  }   
  \\ 
  +
  \qa{
    \tangentialto\W \vecfe g_{\fespace g}
    ,
    \tangentialto\W \vec \psi 
  }_{\boundary\W}
  =
  \qa{
    f^q , \linop{M}_\theta^{q}(\varphi , \vec \psi)
  }  
  \Foreach (\varphi, \vec\psi)\in \fespace u \times \fespace g,
\end{multline}
and corresponding to the problem with nonzero boundary ($r \neq 0$), due to 
(\ref{eq:discrete-non-zero-boundary}), we find 
$(\fe u_{\tilde{\fespace u}}, \vecfe g_{\fespace g}) \in \tilde{\fespace u} \times \fespace g$ 
such that
\begin{multline}
  \label{eq:discrete-nonhomogen-practical}
  \ltwop{
    \grad \fe u_{\tilde{\fespace u}} - \vecfe g_{\fespace g}
  }{
    \grad \varphi-\vec\psi
  }
  +
  \ltwop{
    \rot \vecfe g_{\fespace g}
  }{
    \rot\vec\psi
  }
  +
  \ltwop{
    \linop{M}_\theta^q(\fe u_{\tilde{\fespace u}} ,\vecfe g_{\fespace g} )
  }{
    \linop{M}_\theta^q(\varphi,\vec\psi)
  }  
  \\
  +
  \qa{\fe u_{\tilde{\fespace u}} , \varphi}_{\boundary \W} 
  +
  \qa{\tangentialto\W \vecfe g_{\fespace g} , \tangentialto\W \vec \psi}_{\boundary \W}
  =
  \qa{ r, \varphi}_{\boundary \W} 
  \\
  +
  \qa{ \tangentialto\W \grad r, \tangentialto\W \vec \psi}_{\boundary \W}
  +
  \qa{
    f^q, \linop{M}_\theta^{q}(\varphi , \vec \psi)
  }
  \Foreach (\varphi, \vec\psi)\in \tilde{\fespace u} \times \fespace g.
\end{multline}
Accordingly, for each $\mathit{K} \in \mesh T$, we consider the local error indicator by 
\begin{equation}
  \label{def:error-estimator-general}
  \eta^2(\mathit{K}):=
  \begin{cases}
    \Norm{ \grad \fe u_{\fespace u}- \vecfe g_{\fespace g} }_{\leb2(\mathit{K})}^2
    +
    \Norm{ \rot \vecfe g_{\fespace g} }_{\leb2(\mathit{K})}^2
    \\
    \quad +
    \Norm{ \linop{M}_\theta^{q}(\fe u_{\fespace u}, \vecfe g_{\fespace g}) -f^q }_{\leb2(\mathit{K})}^2  
    +
    \Norm{
      \tangentialto\W \vecfe g_{\fespace g} %
    }_{\leb2( \boundary \mathit{K} \meet \boundary \W) }^2    
    &\text{if } r=0	
    \vspace*{4mm}
    \\
    \Norm{ \grad \fe u_{\tilde{\fespace u}}- \vecfe g_{\fespace g} }_{\leb2(\mathit{K})}^2
    +
    \Norm{ \rot \vecfe g_{\fespace g} }_{\leb2(\mathit{K})}^2
    +
    \Norm{ \linop{M}_\theta^{q}(\fe u_{\tilde{\fespace u}}, \vecfe g_{\fespace g}) -f^q }_{\leb2(\mathit{K})}^2  
    \\
    \quad
    +
    \Norm{\fe u_{\tilde{\fespace u}} - r }_{\leb2( \boundary \mathit{K} \meet \boundary \W) }^2 
    +
    \Norm{\tangentialto\W \qp{\vecfe g_{\fespace g} - \grad r} 
    }_{\leb2( \boundary \mathit{K} \meet \boundary \W) }^2 
    &\text{if } r\neq0, 
  \end{cases}
\end{equation}
and the global error indicator by 
\begin{equation}
  \eta := \sqrt{\sum_{\mathit{K} \in \mesh T}{\eta^2(\mathit{K})}}.
\end{equation}
\subsection{Approximation of control problem}
As we see in Howard's algorithm, in addition to approximating a linear equation in 
nondivergence form, we also need to solve or approximate a control problem at each 
step of the iteration appropriately. 
In dealing with control problems, we approximate them elementwise on $\mesh T$. 
For each $K \in \mesh T$, and given $(u, \vec g)$, $q(K)$ is evaluated as a member of the set
\begin{equation}
	\underset{\substack{\alpha \in \mathcal{A}}}
	\Argmax  \int_K \left( \linop M_\theta^{\alpha}(u,\vec g) - f^\alpha \right)  .
\end{equation}
This makes $q$ belong to the finite element space of $\mathcal
A$-valued piecewise constants ${\fespace a}:= \poly{0}\qp{\mesh T; \mathcal{A}}$. 
We emphasize that the local control problem may not have a unique solution for each $K$ and thus the control $q$ is one of many possible choices.  But we do guarantee that the algorithm will approximate one of its solutions.

We summarize Howard's algorithm and the adaptive refinement algorithm in the next paragraphs.
\begin{Alg}[Howard least-squares gradient recovery Galerkin HJB solver]
  \label{alg:Howard-practical}
  The following pseudocode summarizes our code.%
  \begin{algorithmic}
    \Require{%
      data of the problem~(\ref{eq:HJB0-inhomogeneous}),
      parameter $\theta\in[0,1]$ to (\ref{op:on-cross-space})%
      mesh $\mesh t$ of domain $\W$, $k$ polynomial degree,
      initial guess $(\tilde u_0,\tildevec g_0)$, tolerance
      $\tol$ and maximum number of iterations $\maxiter$.}
    \Ensure{%
      approximate control map and approximate solution to
      (\ref{eq:HJB0-inhomogeneous}) either with
      $\Norm{(\fe[n+1] u,\vecfe[n+1] g)-(\fe[n] u, \vecfe[n] g)}_{\sobh1(\W)}\lesssim \tol$
      or after $\maxiter$ iterations.}
    \Procedure{Howard-GALS}{data of the problem~(\ref{eq:HJB0-inhomogeneous}), $\theta,\mesh t,k,u_0,\vec g_0,\tol,\maxiter$}
    \State{build the Galerkin spaces
      $\fespace u,\tildefespace u,\fespace g$ as (\ref{def:interpolation-space}) 
      and $\fespace a = \poly{0}\qp{\mesh T; \mathcal{A}}$
    }
    \State{\(\pair{\fe u_0}{\vecfe g_0}\gets\)
      projection of \((\tilde u_0,\tildevec g_0)\)
      onto \(\tildefespace u\times\fespace g\)}
    \State{$n\gets 1$}
    \State{$\res\gets \tol +1$}
    \While{$n<\maxiter$ and $\res>\tol$}
    \For{$K \in\mesh T$}\Comment{build piecewise constants control
      map $\fe[n]q\in\fespace a$}
    \State{%
      \textsf{\textbf{find}} $\fe q_n \in {\fespace a}$ such that
      \begin{equation*}
        \fe q_{n}(K) \in
        \Argmax[\alpha \in \mathcal A]
        \int_K
        \qb{
          \linop M_\theta^{\alpha}(\fe[n-1]u,\vecfe[n-1]g) - f^\alpha
        }
      \end{equation*}
    }
    \EndFor
    \State{%
      \textsf{\textbf{solve}} for
      $(\fe[n]u,\vecfe[n]g)%
      \gets
      (\fe u_{\tilde{\fespace u}},\vecfe g_{\fespace g})$
      satisfying problem~(\ref{eq:discrete-nonhomogen-practical})
      with $q\gets q_n$
    }
    \State{\({\res\gets
        \Normonsobh[\W]{
          \pair{\fe[n]u}{\vecfe[n]g}
          -
          \pair{\fe[n-1]u}{\vecfe[n-1]g}
        }1}\)}
    \State{$n\gets n+1$}
    \EndWhile
    \EndProcedure
  \end{algorithmic} 
\end{Alg}
\begin{Alg}[adaptive least-squares gradient recovery Galerkin HJB solver]
  \label{alg:adaptive}
  \begin{algorithmic}
    \Require
    data of the problem~(\ref{eq:HJB0-inhomogeneous}), %
    refinement fraction $\beta\in\opinter01$, tolerance
    $\tol_{\bf a}$,
    and maximum number of iterations 
    $\maxiter_{\bf a}$.
    \Ensure approximate solution to (\ref{eq:HJB0-inhomogeneous}) either with 
    \\ $\Norm{(u,\grad u)-(\fe[l] u,\vecfe[l] g)}_{\sobh1(\W)} \lesssim \tol_{\bf a}$ or 
    after $\maxiter_{\bf a}$ iterations
    \State construct an initial admissible partition $\mathcal{T}_0$
    \State{$\eta^2 \gets \tol_{\bf a} ^2 + 1$}
    \State{$l \gets 0$}
    \While{$l\leq\maxiter_{\bf a}$ and $\eta^2 >\tol_{\bf a}^2$}
    \State{%
      $(q_l,\fe[l] u,\vecfe[l] g)
      \gets$\Call{Howard-GALS}{data of the problem~(\ref{eq:HJB0-inhomogeneous}), $\theta,\mesh t,k,u_0,\vec g_0,\tol,\maxiter$}
    }
    \Comment{
      the output of the Algorithm~\ref{alg:Howard-practical} with $ \mathcal{T} \gets \mathcal{T}_l$ 
    }
    \For {${K} \in \mathcal{T}_l$}
    \State{%
      compute $\eta({K})^2$ via (\ref{def:error-estimator-general})
    }
    \EndFor
    \State{%
      \textsf{\textbf{estimate}} by computing
      $\eta^2 \gets \sums K{\mathcal{T}_l}\eta^2({K})$
    }%
    \State{%
      sort array
      $\seqs{\eta({K})^2}{K\in\mathcal{T}_l}$
      in decreasing order
    }
    \State{%
      \textsf{\textbf{mark}} the first $\ceil{\beta\card\mathcal{T}_l}$
      elements $K$ with the highest $\eta(K)^2$
    }
    \State{\textsf{\textbf{refine}}
      $\mathcal{T}_l$ ensuring split of all marked elements 
      \State{ $l \gets l+1$}
    } 
    \EndWhile 
  \end{algorithmic}
\end{Alg}
\section{Numerical experiments}
\label{sec:numerical-experiment}
In this section, to test the performance of our method, we report on two 
numerical experiments in a subset of $\R 2$ domains. In both experiments, 
as input Howard's algorithm~\ref{alg:Howard-practical}, we choose
$u_0 = 0$, $\vec g_0 =0$ and 
$\theta =0.5$. We set the stop criterion of the algorithm as either 
$\Norm{(\fe[n+1] u,\vecfe[n+1] g) - (\fe[n] u,\vecfe[n] g)}_{\sobh1(\W)}< 10^{-7}$ 
($\tol = 10^{-7}$) or maximum $8$ iterations ($\maxiter = 8$). 
All implementations were done by using the FEniCS package.

In the first experiment, we aim to test the convergence rate of the method. Thus 
the known solution is considered smooth enough, and the numerical results are 
presented on the uniform mesh. In the second experiment, hence we would like 
to test the performance of the adaptive scheme, the known solution is considered 
near singular. Through comparing the convergence rate via the adaptive with the 
uniform refinement, we observe the efficiency of the adaptive scheme. In both 
experiments, the results of the errors are plotted in logarithmic scale.
\subsection{Test problem with nonzero boundary condition}
\label{test:non-homogeneous-boundary}
We approximate (\ref{eq:HJB0-inhomogeneous})
with $\W= (-1,1) \times(-1,1)$ and
$\mathcal{A}:=\sogroup2$ (the special orthogonal group on $\R2$)
parametrized as
$\reals/2\pic\integers=\clopinter0{2\pic}\ni\alpha\mapsto\expp{\ic\alpha}$
(in matrix form), with the metric
\begin{align}
  d_{\mathcal
    A}(\alpha,\beta)
  &
  :=\norm{\expp{\ic\alpha}-\expp{\ic\beta}}=2(1-\cos(\alpha-\beta)),
  \\
  \label{coef:non-homogen}
  \mat{A}^\alpha(\vec x)
  &
  := 
  \begin{bmatrix}[r]
    \cos\alpha & \sin\alpha
    \\
    -\sin\alpha & \cos\alpha
  \end{bmatrix}
  \begin{bmatrix}
    2 & \fracl12
    \\
    \fracl12 & 1
  \end{bmatrix}
  \begin{bmatrix}[r]
    \cos\alpha & -\sin\alpha
    \\
    \sin\alpha & \cos\alpha
  \end{bmatrix},
  \\
  \label{data:non-homogen}
  \vec b^\alpha
  &
  :=\vec 0, 
  \quad
  c^\alpha
  := 2 - 0.5(\cos(2\alpha) + \sin(2\alpha)),
  \\ 
  f^\alpha
  &
  := \linop L^\alpha u  -(1- \cos(2\alpha - \pic(x_1+x_2))).
\end{align}
where the exact solution is
\begin{equation}
  \label{fun:exact-non-homogen}
  u(\vec x) = \sin(\pic x_1) \sin(\pic x_2) + \sin (\pic(x_1 + x_2)).
\end{equation}

$\mat A^\alpha$ together with $\vec b ^\alpha$, $c ^\alpha$ satisfy
the Cordes condition~(\ref{def:general-Cordes-condition}) with
$\lambda = 1$ and $\varepsilon = 0.45$.  We approximate a solution by
following Howard's algorithm~\ref{alg:Howard-practical}.  Different
measures of the error in the together linear($\poly{1}$) and
quadratic($\poly{2}$) finite element spaces are reported in
Fig.~\ref{fig:rate-convergence-non-homogen}. 
To benchmark this test we use the \indexemph{experimental orders of convergence} (\indexen{EOC}) associated with a numerical experiment
with errors $e_i$ and (uniform) meshsizes $h_i$, $i = 0, \cdots, I$, which is defined by
\begin{equation}
	\EOC := \frac{\log(e_{i+1}/e_i)}{\log(h_{i+1}/h_i)}.
\end{equation}
As we observe, even for a
nonhomogeneous problem, the numerical results confirm the analysis
convergence rate of Proposition~\ref{pro:convergence-rate-HJB}.
\begin{figure}[tbh]
  \begin{center}
    \subfloat[$\poly{1}$ elements]{%
      \includegraphics[scale=.45]{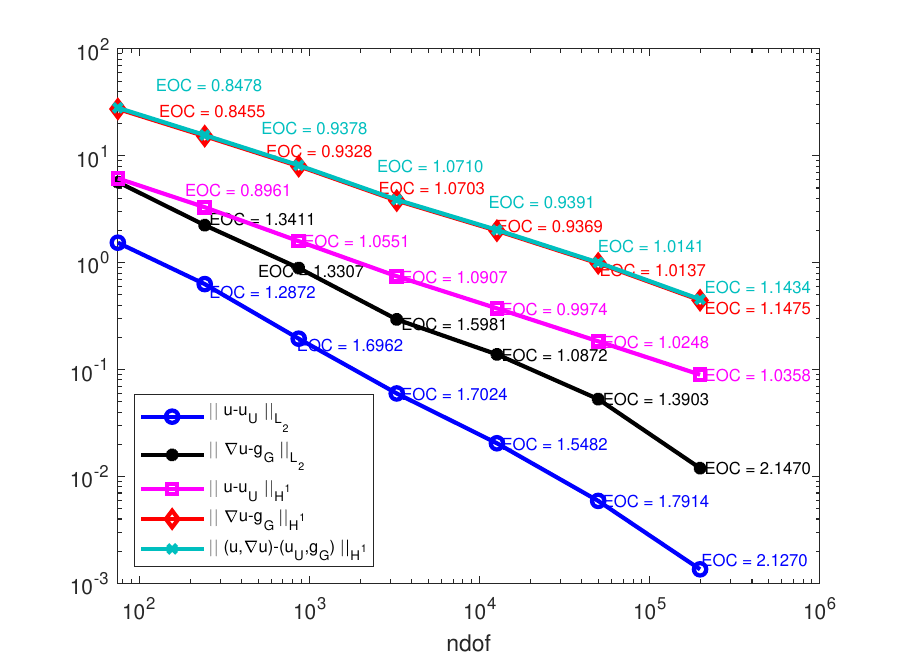}}
    \subfloat[$\poly{2}$ elements]{%
      \includegraphics[scale=.45]{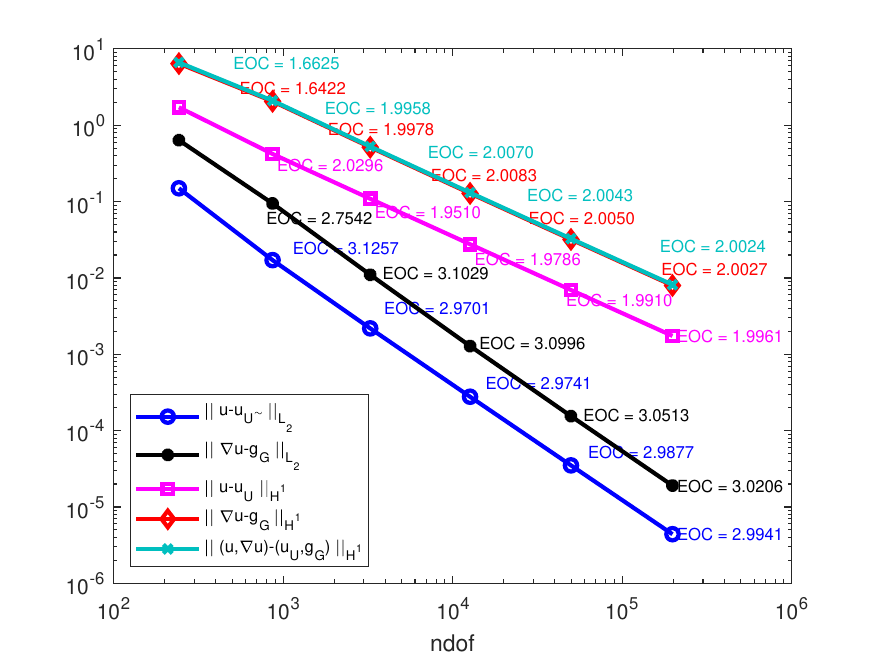}}
    \caption[Rates for Pbm \ref{test:non-homogeneous-boundary}]{%
      Convergence rates for Test problem~\ref{test:non-homogeneous-boundary}. 
    } 
    \label{fig:rate-convergence-non-homogen},
  \end{center} 
\end{figure}
\subsection{Test problem with singular solution}
\label{test:adaptive-disk-domain}
Consider (\ref{eq:HJB-homogeneous}) in the unit disk domain with
the exact solution
\begin{equation}
  \label{fanc:exact-solution-HJB}
  u(\vec x)
  =
  \begin{cases}
    \norm{\vec x}^{5/3} 
    (1-\norm{\vec x})^{5/2}
    (\sin(2\varphi(\vec x)/3))^{5/2}
    &\text{if } 0<  \varphi(\vec x) < 3\pic /2,
    \\
    0
    &\text{otherwise, }
  \end{cases}
\end{equation}
the control set
$\mathcal{A}=\sogroup2\isomorphicto\reals/2\pic\integers=\clopinter0{2\pic}$
and
\begin{equation*}
  \label{coef:test-HJB}
  \mat{A}^\alpha(\vec x) = 
\begin{bmatrix}[r]
  \cos\alpha & \sin\alpha
  \\
  -\sin\alpha & \cos\alpha
\end{bmatrix}
\begin{bmatrix}[r]
  1+(x_1^2 + x_2^2) & 0.005
  \\
  0.005 & 1.01-(x_1^2 + x_2^2)
\end{bmatrix}
\begin{bmatrix}[r]
  \cos\alpha & -\sin\alpha
  \\
  \sin\alpha & \cos\alpha
\end{bmatrix},

\end{equation*}
\begin{equation*}
  \label{data:test-HJB}
  \vec b^\alpha = \vec 0, 
  \quad
  c^\alpha = 0 %
  \quad 
  f^\alpha = \linop L^\alpha u  -(1 - \cos(2\alpha - \pic(x_1+x_2))),
\end{equation*}
where $(\norm{\vec x}, \varphi(\vec x))$ are polar coordinates centered at
the origin. The near degenerate diffusion $\mat A^\alpha$ satisfies
the Cordes condition~(\ref{def:special-Cordes-condition}) with
$\varepsilon = 0.008$. Note that the solution $u$ belongs to
$\sobh{s}(\W)$, for any $s<8/3$. To adaptive refinement, we follow
Algorithm~\ref{alg:adaptive} by choosing $\beta = 0.3$, $\tol_{\bf a}=
10^{-7}$ and $\maxiter_{\bf a}= 8$ and tracking Howard's algorithm to
quadratic($\poly{2}$) finite element space. We let a coarse
quasi-uniform, unstructured mesh as a initial and the mesh generated
by adaptive refinement is shown in
Fig.~\ref{fig:adaptive-uniform-disk-domain}.(A). Since $u\in
\sobh2(\W)$, we do not expect the advantage of the adaptive scheme
over than the uniform refinement for $\sobh1(\W)$-norm error of
$u_{\fespace u}$; it is shown in
Fig.~\ref{fig:adaptive-uniform-disk-domain}.(B). While, since $\vec g
= \grad u$ does not have such smoothness, we observe the superiority
performance of the adaptive scheme rather than the uniform refinement
for $\sobh1(\W)$-norm error of $\vec g_{\fespace g}$ and as well as
$\sobh1(\W)$-norm error of $(u_{\fespace u}, \vec g_{\fespace g})$.
in Fig.~\ref{fig:adaptive-uniform-disk-domain}.(C), (D).
\begin{figure}[h]
  \begin{center}
    \subfloat[]{\includegraphics[scale=.32]{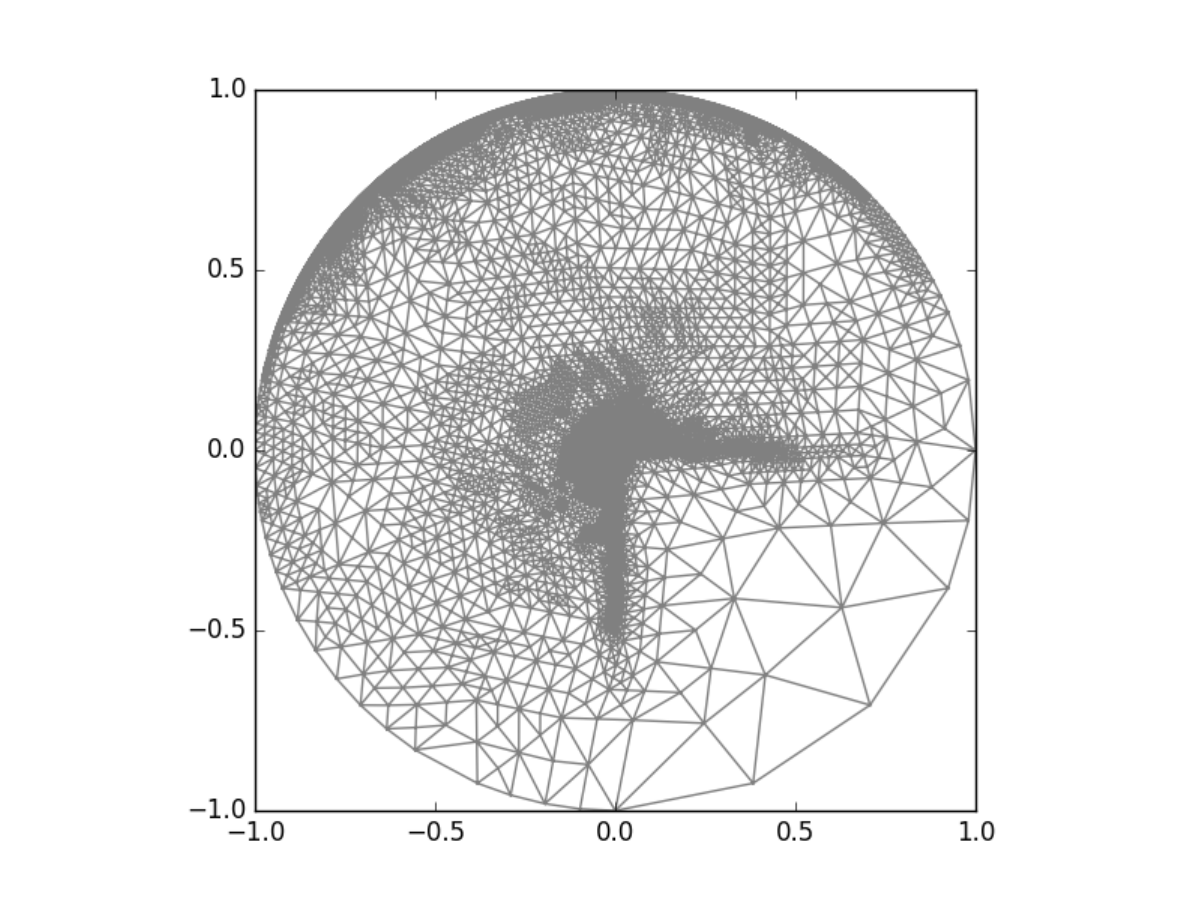}}
    \subfloat[]{\includegraphics[scale=.45]{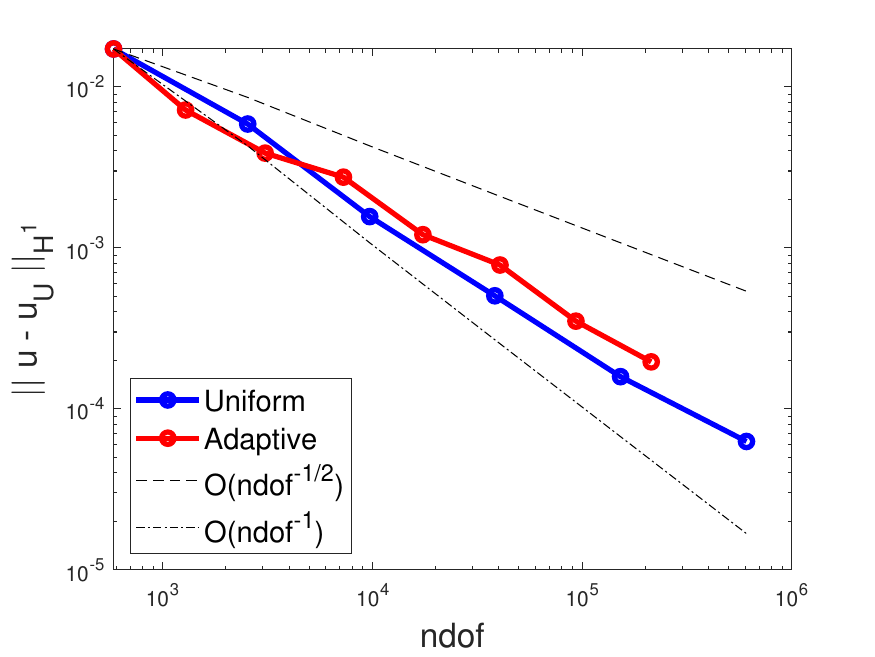}}
    \\
    \subfloat[]{\includegraphics[scale=.45]{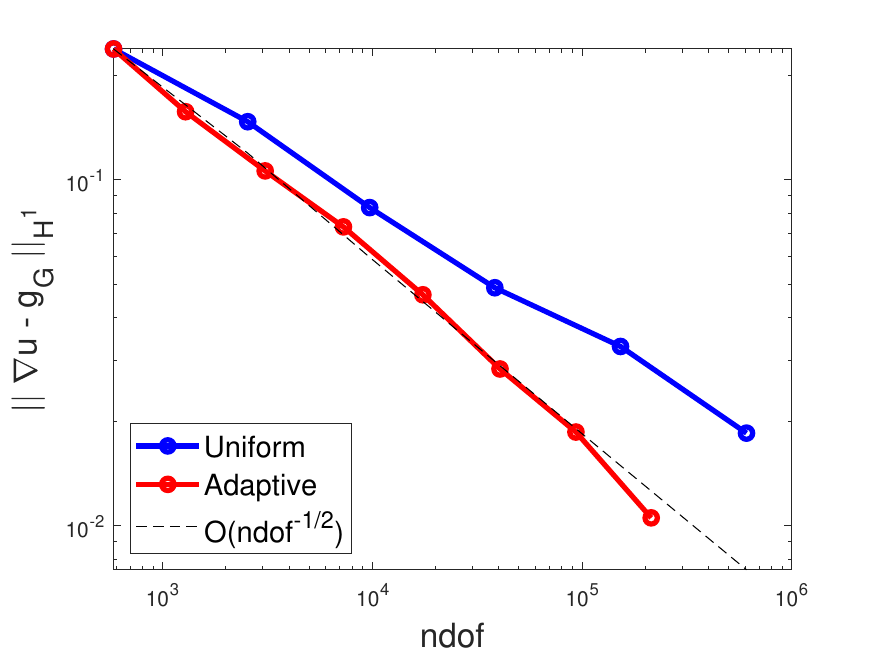}}
    \subfloat[]{\includegraphics[scale=.45]{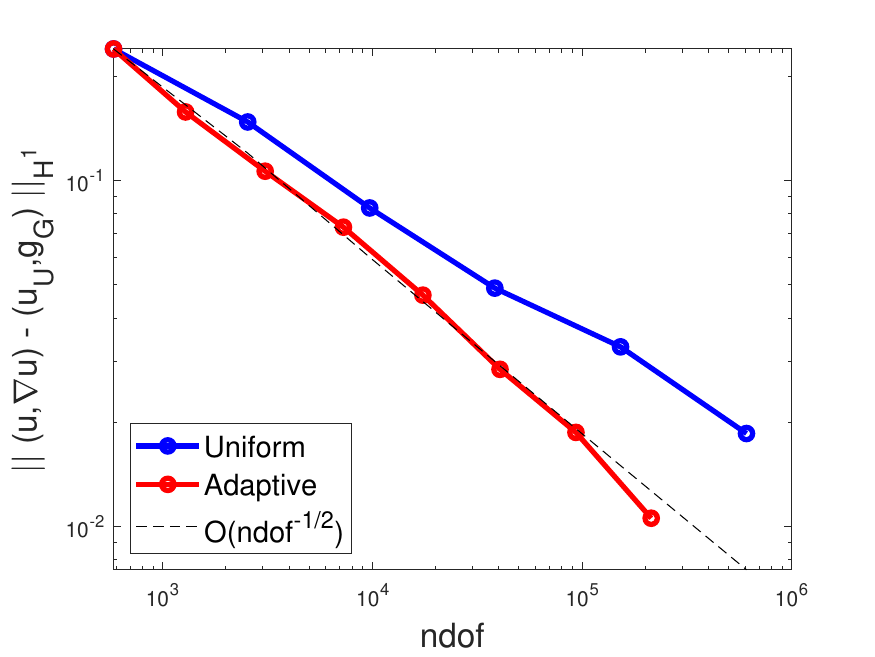}}
    \caption[Adaptive meshes for Pbm \ref{test:adaptive-disk-domain}]{%
      {\scriptsize (A)}: Generated adaptive mesh and {\scriptsize (B), (C), (D)}: convergence 
      rate in the uniform and adaptive refinement for Test problem~\ref{test:adaptive-disk-domain}
      with $\poly{2}$ elements.}
    \label{fig:adaptive-uniform-disk-domain}
  \end{center}
\end{figure} 
\section{Conclusion}
\label{sec:conclusion}
This study efficiently and practically approximated the strong
solution of the fully nonlinear HJB equation in two steps.
First, we linearized the nonlinear HJB equation using the semismooth
Newton method. We demonstrated the Newton differentiability of the HJB
operator from $\sobh2(\W)$ to $\leb2(\W)$ without requiring
ellipticity or the Cordes condition. This suggests that this
linearization method can be applied to approximate the strong solution
of a well-posed HJB equation under less stringent conditions.
Through semismooth Newton linearization, the fully nonlinear HJB
problem transforms into a recursive linear problem in nondivergence
form, with its strong solution converging to the strong solution of
the HJB equation with a superlinear rate.
\newline
With appropriate modifications, one can extend Newton
differentiability and the semismooth Newton linearization to nonconvex
Isaacs equations.

In the second step of the approximation, we discretized the recursive
linear problem in each iteration using a mixed finite element method, employing a least-squares approach with gradient recovery. 
A notable advantage of the least-squares approach is the ability to enforce constraints on the unknowns by incorporating square terms into the cost functional.
However, each discretization that works for the linear equation in nondivergence form can also be applied.

Additionally, we have conducted an error analysis of the approximation, obtaining both \apriori and \aposteriori error bounds. The a priori error bound serves to demonstrate the convergence of the discretization, while the a posteriori error bound facilitates the application of adaptive refinement procedures.
\textnote[AM]{Should we
  also mention here our considerations regarding the error analysis?}
\textnote[OL]{Yes, we need to say something}
\clearpage
\bibliographystyle{abbrvnat}%
\ifthenelse{\boolean{shownotes}}{\printindex}{}
\end{document}